\documentclass[10pt,a4paper]{amsart}
\usepackage[utf8]{inputenc}
\usepackage[pagebackref,colorlinks,linkcolor=red,citecolor=blue,urlcolor=blue,hypertexnames=false]{hyperref}
\usepackage{csquotes}

\usepackage{amsmath,amsthm,amssymb,amsopn}
\usepackage{amsrefs}
\usepackage{amscd}
\input xy
\xyoption{all}
\newdir{ >}{{}*!/-5pt/@{>}}
\usepackage{mathtools}
\let\overfence\overbrace % \overfence is similar to \overbrace
\let\downfencefill\downbracefill % match components of \overbrace
\patchcmd{\overfence}{\downbracefill}{\downfencefill}{}{}% patch \overfence...
\patchcmd{\downfencefill}{\braceru \bracelu}{}{}{}%... and \downfencefill

\usepackage{etoolbox}
\apptocmd{\sloppy}{\hbadness 10000\relax}{}{}

\newcommand{\comment}[1]{} 
	
\usepackage{tikz-cd}

\def\N{\mathbb{N}} 
\def\Z{\mathbb{Z}}

\def\C{\mathbb{C}}

\def\ob{\overfence}
\def\topdf{\texorpdfstring}
\theoremstyle{plain}
\newtheorem{thm}[equation]{Theorem}
\newtheorem{lem}[equation]{Lemma}
\newtheorem{coro}[equation]{Corollary} 
\newtheorem{prop}[equation]{Proposition}

\theoremstyle{definition}

\newtheorem{ex}[equation]{Example}

\theoremstyle{remark} 

 \newtheorem{rem}[equation]{Remark}

\newtheorem{lambada}[equation]{$\lambda$-assumption}
  \numberwithin{equation}{section}
\newtheorem*{ack}{Acknowledgements}
\newtheorem*{dedi}{Dedication} 
 
\newcommand{\booh}{\Upsilon}

\newcommand{\cE}{\mathcal E}

\newcommand{\cH}{\mathcal H}
\newcommand{\cI}{\mathcal I}

\newcommand{\cK}{\mathcal K}

\newcommand{\cN}{\mathcal N}

\newcommand{\cP}{\mathcal P}

\newcommand{\cR}{\mathcal R}
\newcommand{\cS}{\mathcal S}
\newcommand{\cT}{\mathcal T}
\newcommand{\cU}{\mathcal U}
\newcommand{\cV}{\mathcal V}

\def\fA{\mathfrak{A}}
\def\fB{\mathfrak{B}}
\def\fC{\mathfrak{C}}
\def\fE{\mathfrak{E}}
\def\fF{\mathfrak{F}}

\def\fI{\mathfrak{I}}

\def\fP{\mathfrak{P}}

\def\fT{\mathfrak{T}}

\newcommand{\BF}{\fB\fF}

\def\ab{\mathfrak{Ab}}
\def\abmon{\ab\mathfrak{Mon}}
\def\proj{\mathfrak{Proj}}
\def\Mod{\operatorname{Mod}}

\def\forg{\operatorname{forg}}
\def\can{\operatorname{can}}
\def\chara{\operatorname{char}}
\def\gr{\operatorname{gr}}

\def\alg{\mathrm{Alg}}
\def\Gr{\mathrm{Gr}}
\newcommand{\aha}{{\alg_\ell}}
\newcommand{\ahas}{{{\rm Alg}^*_\ell}}
\newcommand{\lra}{\longrightarrow}
\newcommand{\iso}{\overset{\sim}{\lra}}

\newcommand{\onto}{\twoheadrightarrow}

\newcommand{\ol}{\overline}

\def\reg{\operatorname{reg}}
\def\sing{\operatorname{sing}}
\def\sink{\operatorname{sink}}
\def\inf{\operatorname{inf}}
\def\sour{\operatorname{sour}}
\def\triqui{\vartriangleleft}

\def\supp{\operatorname{supp}}
\def\inc{\operatorname{inc}}

\def\rk{\operatorname{rk}}

\def\diag{\operatorname{diag}}
\def\Gl{\operatorname{GL}}

\def\ad{\operatorname{ad}}

\def\ev{\operatorname{ev}}
\def\id{\operatorname{id}}
\def\inv{\operatorname{inv}}

\newcommand{\coker}{{\rm Coker}}
\renewcommand{\ker}{{\rm Ker}}
\newcommand{\im}{\mathrm{Im}}
\newcommand{\dom}{\mathrm{Dom}}
\newcommand{\op}{\mathrm{op}}

\DeclareMathOperator*{\colim}{colim}

\def\tor{\operatorname{Tor}}
\def\Hom{\operatorname{Hom}}

\begin{document}
\hfuzz=22pt
\vfuzz=22pt
\hbadness=2000
\vbadness=\maxdimen

\author{Guillermo Cortiñas}
\thanks{Supported by CONICET and partially supported by grants PICT 2017-1395 from Agencia Nacional de Promoci\'on Cient\'\i fica y T\'ecnica, UBACyT 0150BA from Universidad de Buenos Aires, and PGC2018-096446-B-C21 from Ministerio de Ciencia e Innovaci\'on}
\title[Classifying Leavitt path algebras]{Classifying Leavitt path algebras up to involution preserving homotopy}

\begin{abstract} 
We prove that the Bowen-Franks group classifies the Leavitt path algebras of purely infinite simple finite graphs over a regular supercoherent commutative ring with involution where $2$ is invertible, equipped with their standard involutions, up to matricial stabilization and involution preserving homotopy equivalence. We also consider a twisting of the standard involution on Leavitt path algebras and obtain partial results in the same direction for purely infinite simple graphs. Our tools are $K$-theoretic, and we prove several results about (Hermitian, bivariant) $K$-theory of Leavitt path algebras, such as Poincar\'e duality, which are of independent interest. 
\end{abstract}

\maketitle

\section{Introduction}
A directed graph $E$ consists of a set $E^0$ of vertices and a set $E^1$ of edges together with source and range functions $r,s:E^1\to E^0$. This article is concerned with the Leavitt path algebra $L(E)$ of a directed graph $E$ over a commutative ring $\ell$ with involution \cite{lpabook}. When $\ell=\C$, with complex conjugation as involution, $L(E)$ is a normed $*$-algebra; its completion is the graph $C^*$-algebra $C^*(E)$.  A graph $E$ is called finite if both $E^0$ and $E^1$ are; a finite graph $E$ is purely infinite simple if and only if $C^*(E)$ is, which in turn is equivalent to a combination of graph-theoretic conditions on $E$ \cite{lpabook}*{Section 5.6}. A result of Cuntz and R{\o}rdam (\cite{ror}*{Theorem 6.5}) says that $C^*$-algebras of finite purely infinite simple graphs, i.e. purely infinite simple Cuntz-Krieger algebras, are classified up to (stable) isomorphism by the Bowen-Franks group $\BF$ of the corresponding graph, defined as
\[
\BF(E)=\coker(I-A_E^t).
\]
Here $A_E$ is the incidence matrix of $E$. It is an open question whether a similar result holds for Leavitt path algebras 
\cite{alps}. Here we consider the question for the Leavitt path algebra $L(E)$ over a fixed commutative ring $\ell$ with involution, viewed as a $*$-algebra over $\ell$ by means of its usual $\ell$-semilinear involution $*$. 

The following classification theorem follows from Theorem \ref{thm:class2} (see Example \ref{ex:apliclass}); the vocabulary and notations therein are explained right after the theorem. 

\begin{thm}\label{intro:class2}
Let $E$ and $F$ be finite, purely infinite simple graphs and let $\ell$ be regular supercoherent such that $2$ is invertible and $-1$ is positive in $\ell$.  Let $\xi_0:\BF(E)\iso \BF(F)$ be a group isomorphism. Then there exist $*$-homomorphisms $\phi:LE\to LF$ and $\psi:LF\to LE$, compatible with $\xi_0$ and $\xi_0^{-1}$, both very full, and such that $\psi\circ\phi\sim^*_{M_{\pm 2}}\id_{LE}$ and $\phi\circ\psi\sim^*_{M_{\pm 2}}\id_{LF}$. If furthermore $\xi_0([1]_E)=[1]_F$, then $\phi$ and $\psi$ can be chosen to be unital. 
\end{thm}
In the theorem above,
\begin{equation}\label{intro:1E}
[1]_E=\sum_{v\in E^0}[v]\in\BF(E).
\end{equation}
A ring $R$ is \emph{regular} if every $R$-module admits a projective resolution of finite length, \emph{coherent} if the category of finitely presented $R$-modules is abelian, and \emph{supercoherent} if $R[t_1,\dots,t_n]$ is coherent for all $n\ge 0$. The hypothesis that $-1$ be positive in $\ell$ means that there exist $n\ge 1$ and $x_1,\dots,x_n\in\ell$ such that 
\begin{equation}\label{intro:posi}
-1=\sum_{i=1}^nx_ix_i^*.
\end{equation}
Hypothesis \eqref{intro:posi} is satisfied, for example, when $\ell=\C$ with trivial involution, but fails for $\C$ with complex conjugation as involution. A $*$-homomorphism is an involution preserving homomorphism between $*$-algebras over $\ell$. A projection $p$ in a unital $*$-algebra $R$ is \emph{very full} if there is $x\in pR$ such that $x^*x=1$. If $E$ is purely infinite simple and finite and $R$ is a unital $*$-algebra, a $*$-homomorphism $\phi:L(E)\to R$ is very full if $\phi(1)$ is very full. The compatibility condition above is made explicit in Theorem \ref{thm:class2} below. We write $M_2R$ for the algebra of $2\times 2$ matrices equipped with the conjugate transpose involution, and $M_{\pm}R$ for the same algebra but with the following involution
\begin{equation}\label{map:invopm}
\begin{bmatrix}a&b\\ c&d\end{bmatrix}^*=\begin{bmatrix}a^*&-c^*\\ -b^*&d^*\end{bmatrix}.
\end{equation}
Consider the upper left hand corner inclusions $\iota_1:R\to M_2R$ and $\iota_+:R\to M_{\pm}R$.
 We write $f\sim^*_{M_{\pm 2}}g$ to indicate that two $*$-homomorphisms $f,g:A\to R$ become homotopic via an involution preserving homotopy upon composing both of them with $\iota_+\iota_1$. 

In the absence of the positivity hypothesis \eqref{intro:posi}, we need to stabilize further.

Theorem \ref{thm:class1} implies the following (see Example \ref{ex:apliclass}).

\begin{thm}\label{intro:class1} 
Let $E$ and $F$ be purely infinite simple, finite graphs. Assume that $2$ is invertible in $\ell$. Let $\xi_0:\BF(E)\iso \BF(F)$ be an isomorphism. Then there are $*$-homomorphisms $\phi:L(E)\to M_{\pm}L(F)$ and $\psi:L(F)\to M_{\pm}L(E)$ compatible with $\xi_0$ and $\xi_0^{-1}$, both very full, and such that $M_{\pm}(\psi)\circ\phi\sim^s_{M_{\pm 2}}\iota_+^2:LE\to M_{\pm}M_{\pm}L(E)$ and $M_{\pm}(\phi)\circ\psi\sim^s_{M_{\pm 2}}\iota_+^2:L(F)\to M_{\pm}M_{\pm}L(F)$.
\end{thm} 

Here $\sim^s_{M_{\pm 2}}$ refers to stable $M_{\pm 2}$-$*$-homotopy, defined in Section \ref{sec:k0spi}, right before \eqref{defi:stably}. For very full $*$-homomorphisms it agrees with $M_{\pm 2}$-$*$-homotopy. 

Theorems \ref{intro:class2} and \ref{intro:class1} complement \cite{cm2}*{Theorem 6.1} where a similar classification result was obtained for Leavitt path algebras over a field, up to not necessarily involution preserving homotopy. In this  paper we also prove a version of the latter result that is valid for Leavitt path algebras over any regular supercoherent ring, Theorem \ref{thm:class3}. 
All these results can be seen as algebraic versions of a theorem of Cuntz (predating R\o rdam's article \cite{ror}) which says that $\BF$ classifies purely infinite simple Cuntz-Krieger $C^*$-algebras up to homotopy \cite{cclass}*{Proposition 1}.

We also consider the problem of the existence of homomorphisms connecting the Leavitt path algebras $L_2$ of the graph $\cR_2$ consisting of one vertex and two loops, and $L_{2^-}$ of its Cuntz splice. It is well-known that if $\ell$ is a field, then there is no isomorphism $L_2\iso L_{2^-}$ that is homogeneous with respect to the standard $\Z$-grading. Associated to this $\Z$-grading, we have a $\Z/2\Z$, even-odd grading. We show in Proposition \ref{prop:nograded} that if $\ell$ is regular supercoherent and such that the canonical map $\Z\to K_0(\ell)$ is an isomorphism (e.g. if $\ell=\Z$ or a field) then there is no $\Z/2\Z$-homogeneous unital $\ell$-algebra homomorphism from $L_2$ to $L_{2^-}$ nor in the opposite direction. R. Johansen and A. S\o rensen have shown in \cite{splice} that in the case $\ell=\Z$, $L_2$ and $L_{2^-}$, equipped with their standard involutions, are not isomorphic as $*$-rings. Here we consider, for any graph $E$ the \emph{signed} involution
\begin{equation}\label{intro:sign}
\ol{\,}:L(E)\to L(E)^{\op},\,\, \ol{v}=v,\,\, \ol{e}=-e^*,\,\, (v\in E^0, \,\, e\in E^1).
\end{equation}
We write $\ol{L(E)}$ for $L(E)$ equipped with the signed involution. Proposition \ref{prop:nograded} also shows that if $\ell=\C$ equipped with complex conjugation, then there is no unital $*$-homomorphism connecting $\ol{L_2}$ and $\ol{L_{2^-}}$.

The classification theorem for graph $C^*$-algebras of \cite{ror} uses Kasparov's bivariant $C^*$-algebra $K$-theory, which takes involutions into consideration. The classification result for Leavitt path algebras in \cite{cm2} uses the bivariant algebraic $K$-theory  of \cite{ct}, which is defined on the category of all algebras. Here we use the hermitian bivariant $K$-theory $kk^h$ introduced in \cite{cv}; it is a bivariant $K$-theory for $*$-algebras over a commutative ring $\ell$ with involution. The ring $\ell$ is assumed to satisfy the following.

\begin{lambada}\label{stan:lambda}
The ground ring $\ell$ contains an element $\lambda$ such that $\lambda+\lambda^*=1$.
\end{lambada}

For such $\ell$ there are a triangulated category $kk^h$ and a functor
$j^h:\ahas\to kk^h$ from the category of $*$-algebras which is homotopy invariant, matricially stable, hermitian stable, and satisfies excision and is universal with those properties. We write $[-1]$ for the suspension in $kk^h$ and set
\begin{equation}\label{intro:kkhn}
kk^h_n(A,B)=\hom_{kk^h}(j^h(A),j^h(B)[n]), \,\, kk^h(A,B)=kk_0^h(A,B).
\end{equation}
The main technical result about $kk^h$ that we prove in this paper, which is key to the proofs of Theorems \ref{intro:class2} and \ref{intro:class1}, is Theorem \ref{thm:kklift}, which we reproduce in part below as Theorem \ref{intro:kklift}. 
We call a $*$-algebra $R$ \emph{strictly properly infinite} if it contains two orthogonal isometries. A ring $R$ is regular with respect to a functor $F$ if $F(R)\to F(R[t_1,\dots,t_n])$ is an isomorphism for all $n\ge 1$. We write $K_0^h(R)$ for the Witt-Grothendieck group of $R$ (see Section \ref{sec:k0}). We write $[L(E),R]^f_{M_{\pm 2}}$ for the set of $M_{\pm 2}$-homotopy classes of very full $*$-homomorphisms; if $R$ is strictly properly infinite, this set has a natural semigroup structure (see Remark \ref{rem:Psemigroup}). 

\begin{thm}\label{intro:kklift}
Let $E$ be a finite, purely infinite simple graph and $R\in\ahas$ a $K_0^h$-regular, strictly properly infinite $*$-algebra over a ring $\ell$ satisying the $\lambda$-assumption. Assume that $-1$ is positive in $R$. Then the map
\begin{equation}\label{intromap:kklift}
j^h:[L(E),R]^f_{M_{\pm 2}}\to kk^h(L(E),R)
\end{equation}
is a semigroup isomorphism. 
\end{thm}

We show in Corollary \ref{coro:spipis} that if $F$ is finite purely infinite simple then $L(F)$ is strictly properly infinite. If $\ell$ is regular supercoherent, then $L(F)$ is $K_0$-regular, and even $K_0^h$-regular if $2$ is invertible in $\ell$ (Lemma \ref{lem:lekreg}). Thus if $-1$ is positive in $\ell$, Theorem \ref{intro:kklift} reduces the proof of Theorem \ref{intro:class2}, to showing that an isomorphism $\BF(E)\iso \BF(F)$ lifts to an isomorphism $j^h(L(E))\iso j^h(L(F))$. If $-1$ is not positive in $\ell$, then it cannot be positive in $L(E)$ (Lemma \ref{lem:-1}), hence we cannot apply Theorem \ref{intro:kklift} with $R=L(F)$. Observe however that $-1$ is positive in $M_{\pm}$ (see Example \ref{ex:mpm-1}); the proof of Theorem \ref{intro:class1} uses Theorem \ref{intro:kklift} applied to $R=M_{\pm}L(F)$ and the fact that $\iota_+$ is an isomorphism in $kk^h$. 

If $F$ is finite and regular, then $-1$ is positive in $\ol{L(F)}$ (see \eqref{eq:-1ole}). We show in Corollary \ref{coro:spipis} that if $F$ is finite and purely infinite simple, then $\ol{L(F)}$ is strictly properly infinite, hence Theorem \ref{intro:kklift} applies to $R=\ol{L(F)}$. In order to obtain a classification result similar to those of Theorems \ref{intro:class2} and \ref{thm:class1} for Leavitt path algebras of finite purely infinite simple graphs equipped with the involution \eqref{intro:sign}, we would need to have an analogue of Theorem \ref{intro:kklift} with $\ol{L(E)}$ substituted for $L(E)$. The proof of the surjectivity of \eqref{intromap:kklift} uses Theorem \ref{thm:k0lift} which in turn relies on the fact that the edges of $E$ are partial isometries of $L(E)$. This is no longer the case in $\ol{L(E)}$, as $e\ol{e}e=-e$ for $e\in E^1$, so our argument does not apply. However we do manage to show, in Theorem \ref{thm:olkklift}, that for $E$ and $R$ as in Theorem \ref{intro:kklift}, the map
\begin{equation}\label{intromap:olkklift}
j^h:[\ol{L(E)},R]^f_{M_{\pm 2}}\to kk^h(\ol{L(E)},R)
\end{equation}
is injective. The proof of this injectivity result as well as of the injectivity  of \eqref{intromap:kklift} uses Poincar\'e duality for Leavitt path algebras, established in Theorem \ref{thm:duality}, which says that if $E$ is finite and has no sinks and no sources and $E_t$ is the dual graph, then for any two $*$-algebras $R$ and $S$ there are isomorphisms
\begin{equation}\label{intro:pd}
kk^h(R\otimes L(E),S)\cong kk_1^h(R,S\otimes L(E_t)),\,\, kk^h(R\otimes \ol{L(E)},S)\cong kk_1^h(R,S\otimes \ol{L(E_t)}).
\end{equation}
This is a purely algebraic version of the analogue statement for graph $C^*$-algebras established by J. Kaminker and I. Putnam in \cite{kamiput}.   

The rest of this article is organized as follows. Section \ref{sec:prelis} contains some basic material about $*$-algebras, bivariant hermitian $K$-theory and Cohn and Leavitt path algebras. In particular we recall from \cite{cv} that  $kk_*^h(\ell,R)=KH^h_*(R)$ is a Weibel-style \cite{kh} homotopy invariant version of hermitian $K$-theory $K^h$, and is related to the latter via a comparison map $K_*^h(R)\to KH_*^h(R)$ that is an isomorphism when $R$ is $K^h$-regular. For example we show in Lemma \ref{lem:lekreg} that if $R$ is regular supercoherent and $E$ is countable, then $L(E)\otimes R$ is $K$-regular, and even $K^h$-regular if furthermore $2$ is invertible in $R$. In Lemma \ref{lem:-1} we show that if $E^0$ is finite then $-1$ is positive in $L(E)$ if and only if it is positive in $\ell$, in which case $L(E)\cong \ol{L(E)}$. 

The category $kk^h$ is enriched over $KH_0^h(\ell)=kk^h(\ell,\ell)$, which is an algebra over $\Z[\sigma]$, the group ring of the group with two elements. In Section \ref{sec:lpaskkh} we prove Theorem \ref{thm:gentri}, which says that if $\cT$ is a triangulated category and $X:\ahas\to\cT$ is a homotopy invariant, matricially and hermitian stable and excisive which commutes with direct sums of at most $|E^0|$ summands, then there are distinguished triangles in $\cT$
\begin{gather}
\xymatrix{X(\ell)^{(\reg(E))}\ar[r]^{I-A_E^t}&X(\ell)^{(E^0)}\ar[r]& X(L(E))}\label{intro:xtri1}\\
\xymatrix{X(\ell)^{(\reg(E))}\ar[r]^(.55){I-\sigma A_E^t}&X(\ell)^{(E^0)}\ar[r]& X(\ol{L(E)}).}\label{intro:xtri2}
\end{gather}
In particular this applies to $j^h$ when $E^0$ is finite (Theorem \ref{thm:fundtriang}). Both triangles above are obtained from the standard presentation $L(E)=C(E)/\cK(E)$ as a quotient of the Cohn algebra by an ultramatricial ideal. The key calculation is done in Theorem \ref{thm:cekk}, where we show that $j^h(C(E))\cong j^h(\ol{C(E)})\cong j^h(\ell)^{E^0}$. 

In Section \ref{sec:kkhbf} we introduce, for a graph $E$, the $\Z[\sigma]$-module
\[
\ol{\BF}(E)=\coker(I-\sigma A_E^t).
\]
We use the results of the previous section, and $K^h$-regularity Lemma \ref{lem:lekreg}, to make some hermitian $K$-theory computations in Corollaries \ref{coro:kseq} and
\ref{coro:kseq2}. For example, the latter corollary proves that if $\ell$ is a field with $\chara(\ell)\ne 2$ and $E$ is countable, then 
\begin{equation}\label{intro:kh0}
K^h_0(L(E))=\BF(E)\otimes_\Z K_0^h(\ell), \, K^h_0(\ol{L(E)})=\ol{\BF}(E)\otimes_{\Z[\sigma]}K_0^h(\ell). 
\end{equation}

We also show, in Lemma \ref{lem:phantom}, that if $E^0$ is finite then $j^h(L(E))=0$ if and only if $\BF(E)\otimes KH_0^h(\ell)=0$ and that $j^h(\ol{L(E)})=0$ if and only if $\ol{\BF}(E)\otimes_{\Z[\sigma]}KH_0^h(\ell)=0$. In Example \ref{ex:phantom} we exhibit a purely infinite simple finite graph $\Upsilon$ such that $j^h(L(\Upsilon))=j^h(\ol{L(\Upsilon)})=0$.

In Section \ref{sec:nonex} we observe that if $\ell$ is regular supercoherent and $E^0$ is finite, then for the $K$-theory of $\Z/2\Z$-graded projective modules, 
\begin{equation}\label{intro:kgr0}
K_0^{\gr}(L(E))=\ol{\BF}(E)\otimes_{\Z}K_0(\ell).
\end{equation}
The results about non-existence of unital homomorphisms connecting $L_2$ and $L_{2^-}$ mentioned above are proved in Proposition \ref{prop:nograded}. The proof combines \eqref{intro:kh0}, \eqref{intro:kgr0} and the computation
\[
\ol{\BF}(\cR_2)=\Z/2\Z,\,\, \ol{\BF}(\cR_{2^-})=\Z/7\Z.
\]
In Section \ref{sec:struct} we classify $j^h(L(E))$ and $j^h(\ol{L(E)})$ up to isomorphism. For example, we show in Theorem \ref{thm:isokk}  that if $E$ and $F$ have finitely many vertices and the same number of singular vertices, then any group isomorphism $\BF(E)\iso \BF(F)$ lifts to an isomorphism $j^h(L(E))\iso j^h(L(F))$; this is key to the strategy, explained above, of the proof of Theorem \ref{thm:kklift}. We also show, in Theorem \ref{thm:isokk2}, that if furthermore $\ker(I-\sigma A_E^t)=\ker(I-\sigma A_F^t)=0$, then any $\Z[\sigma]$-module isomorphism $\ol{\BF}(E)\iso\ol{\BF}(F)$ lifts to an isomorphism $j^h(\ol{L(E)})\iso j^h(\ol{L(F)})$. In Proposition \ref{prop:bfolbf} and Remark \ref{rem:bfolbf2}, we also characterize, under some hypothesis on $\ell$, those finite graphs $E$ such that $\BF(E)$ and $\ol{\BF}(E)$ are finite and isomorphic, an show that for such $E$, $j^h(\ol{L(E)})$ is a direct summand of $j^h(L(E))$. 

In Section \ref{sec:k0} we consider, for a unital $*$-algebra $R$, the group completion $K_0(R)^*$ of the monoid $\cV_\infty(R)^*$ of Murray-von Neumann equivalences classes of projections in $M_\infty R$. If $R$ is a $C^*$-algebra, $K_0(R)^*$ is just $K_0(R)$. For any unital $*$-algebra $R$, the hermitian Witt-Grothendieck group of $R$ is $K_0^h(R)=K_0(M_{\pm} R)^*$. The inclusion $\iota_+$ induces a natural transformation $K_0(R)^*\to K_0^h(R)$, which is an isomorphism if $-1$ is positive in $R$. We consider the concept of a strictly full projection and show in Lemma \ref{lem:vmorite} that if $p\in R$ is strictly full, then the inclusion $pRp\subset R$ induces an isomorphism in $\cV_\infty(-)^*$ and $K_0(-)^*$. For example if $-1$ is positive in $R$ then $\iota_+(1)\in M_{\pm}R$ is strictly full, and so the natural map $K_0(R)^*\to K_0^h(R)$ is an isomorphism. Remark \ref{rem:can'} %and \ref{rem:olcan} 
introduces, for a graph $E$ with finite $E^0$, a canonical group homomorphism $\can':\BF(E)\to K_0(L(E))^*$.% and a %canonical $\Z[\sigma]$-module homomorphism $\ol{\can}:\ol{\BF(E)}\to K_0^h(\ol{L(E)})$. 

In Section \ref{sec:k0spi} we prove several technical lemmas concerning projections and strictly properly infinite algebras. For example we show in Corollary \ref{coro:spipis} that if $E$ is finite and purely infinite simple, then $L(E)$ and $\ol{L(E)}$ are strictly properly infinite. This, together with the $K^h$-regularity result from Section \ref{sec:prelis} mentioned above (Lemma \ref{lem:lekreg}) allows us to apply Theorem \ref{intro:kklift} to $R=L(E), L(F)$ in the proof of Theorem \ref{intro:class2} and to $R=M_{\pm}L(E), M_{\pm}L(F)$ in that of Theorem \ref{intro:class1}.

In Section \ref{sec:k0lift} we prove Theorem \ref{thm:k0lift}, which says that if $E$ is a countable graph and $R$ a strictly properly infinite $*$-algebra,
then any group homomorphism 

\goodbreak

\noindent $\xi:\BF(E)\to K_0(R)^*$ lifts to a very full $*$-homomorphism $\phi:L(E)\to R$ such that $K_0(\phi)^*\circ\can'=\xi$. This result is used later on, in the proof of Theorem \ref{thm:kklift}, to show that the map \eqref{intromap:kklift} is onto. 

In Section \ref{sec:k1} we define the group $K_1(R)^*$ of a unital $*$-algebra $R$ as the abelianization of the infinite unitary group $\cU_\infty(R)=\bigcup_{n=1}^\infty \cU(M_nR)$. We show that $K_1(-)^*$ is invariant under passing to strictly full corners (Lemma \ref{lem:k1morite}). The usual hermitian $K_1$ is obtained as $K_1^h(R)=K_1(M_{\pm}R)^*$, and the natural map $\iota_+:K_1(R)^*\to K_1^h(R)$ is an isomorphism whenever $-1$ is positive in $R$ (Remark \ref{rem:k1=kh1}). We also consider a homotopy invariant, Karoubi-Villamayor style \cites{kv1,kv2} version of $K_1(-)^*$, which we call $KV_1(-)^*$. There is a canonical surjective map $K_1(R)^*\to KV_1(R)^*$ which is an isomorphism if $-1$ is positive in $R$ and $R$ is $K^h_0$-regular. Proposition \ref{prop:k1pi} and Lemma \ref{lem:kv1} describe $K_1(R)^*$ and $KV_1(R)^*$ for $R$ strictly properly infinite, as quotients of $\cU(R)$. This description is later used in the proofs of Theorems \ref{thm:kklift} and \ref{thm:olkklift}.

Section \ref{sec:duality} is concerned with Poincar\'e duality \eqref{intro:pd}, proved in Theorem \ref{thm:duality}.

In Section \ref{sec:uct} we use the triangles \eqref{intro:xtri1} and \eqref{intro:xtri2} to prove Theorem \ref{thm:uct}, which says that if $E$ is a graph with $|E^0|<\infty$ and $R$ is a $*$-algebra, then for 
\[
\ol{\BF}^\vee(E)=\coker(I^t-\sigma A_E)\,\text{ and }\, \BF^\vee(E)=\ol{\BF}^\vee(E)\otimes_{\Z[\sigma]}\Z,
\]
we have exact sequences
\begin{gather}\label{introseq:kkler}
0\to KH_1^h(R)\otimes\BF^\vee(E)\to kk^h(L(E),R)\to \hom(\BF(E),KH_0^h(R))\to 0,\\
0\to KH_1^h(R)\otimes_{\Z[\sigma]}\ol{\BF}^\vee(E)\to kk^h(\ol{L(E)},R)\to \hom_{\Z[\sigma]}(\ol{\BF}(E),KH_0^h(R))\to 0.\nonumber
\end{gather}
If $E$ is finite with no sinks and no sources, then $\BF^\vee(E)=\BF(E_t)$ and $\ol{\BF}^\vee(E)=\ol{\BF}(E_t)$. Combining this with Poincaré duality we obtain a description of the injective maps in \eqref{introseq:kkler} that is key in the proofs of Theorems \ref{thm:kklift} and \ref{thm:olkklift}.

Theorem \ref{intro:kklift} is proved in Section \ref{sec:kklift} as Theorem \ref{thm:kklift}. Theorem \ref{thm:olkklift} shows that the map \eqref{intromap:olkklift} is injective. The proofs use material from all of the previous sections. 

Section \ref{sec:class} uses Theorem \ref{intro:kklift} to prove Theorems \ref{intro:class2} and \ref{intro:class1}, which are respectively Theorem \ref{thm:class2} and Theorem \ref{thm:class1}. They classify Leavitt path algebras of finite, purely infinite simple graphs up to involution preserving homotopy and matricial stabilization. Theorem \ref{thm:class3} gives an analogous result under not necessarily involution preserving homotopy.

\begin{ack} I wish to express my gratitude towards Santiago Vega, whose contributions go beyond our collaboration in \cite{cv}. Thanks to Guido Arnone, who read several earlier drafts of this article and gave useful feedback.
Thanks also to Max Karoubi, Jonathan Rosenberg, Marco Schlichting and Chuck Weibel, whom I bothered (especially insistingly in the case of Marco) with several questions about Hermitian $K$-theory. I am also indebted to Joachim Cuntz for pointing out the relevance of some of his classic results and arguments to the present work. 
\end{ack}

\begin{dedi} I dedicate this article to the memory of my parents, Julio Corti\~nas (1933--2019) and Lina Villar (1934--2020).
\end{dedi}

\section{Preliminaries}\label{sec:prelis}

\subsection{Algebras and involutions}
\numberwithin{equation}{subsection}
 A commutative unital ring $\ell$ with involution $*$ is fixed throughout the article. A \emph{$*$-algebra} is an $\ell$-algebra $R$ equipped with an involution $*:R\to R^{\op}$ that is semilinear with respect to the $\ell$-module action; $(\lambda a)^*=\lambda^*a^*$ for all
$\lambda\in\ell$ and $a\in R$. We use the term $*$-ring for a $*$-$\Z$-algebra. A \emph{$*$-homomorphism} between $*$-algebras is an algebra homomorphism which commutes with involutions. We write $\aha$ for the category of algebras and algebra homomorphisms and 
$\ahas$ for the subcategory of $*$-algebras and $*$-homomorphisms.

Tensor products of algebras are taken over $\ell$; we write $\otimes$ for $\otimes_\ell$. We also use $\otimes$ for tensor products of abelian groups, e.g. $K_0(R)\otimes K_0(S)=K_0(R)\otimes_{\Z} K_0(S)$. If $R$ and $S$ are $*$-algebras, we regard $R\otimes S$ as a $*$-algebra with the tensor product involution $(a\otimes b)^*=a^*\otimes b^*$. If $L$ is a $*$-algebra and $A\in\ahas$, we shall often write $LA$ for $L\otimes A$.

\begin{ex}\label{ex:inv}
Let $A$ be a ring. Put $\inv(A)=A\oplus A^{\op}$ for the $*$-ring with the coordinatewise operations and involution $(a,b)\mapsto (b,a)$. If $B$ is a $*$-ring, then $\inv(A)\otimes_\Z B\to \inv(A\otimes_\Z B)$ $(a_1\otimes b_1,a_2\otimes b_2)\mapsto (a_1\otimes b_1,a_2\otimes b_2^*)$ is a $*$-isomorphism.  If $A$ is an algebra over a ground ring $\ell$, then $\inv(A)$ is a $*$-$\inv(\ell)$ algebra; this gives rise to a category equivalence 
$\inv:\aha\iso\alg^*_{\inv(\ell)}$. 
\end{ex}

The polynomial ring $B[t]=B\otimes_{\Z}\Z[t]$ with coefficients in a $*$-algebra is equipped with the tensor product involution, where
$\Z[t]$ has the trivial involution. 

An \emph{elementary $*$-homotopy} between two $*$-homomorphisms $f_0,f_1:A\to B$ is a $*$-homomorphism $H:A\to B[t]$ such that 
$\ev_i\circ H=f_i$ ($i=0,1$). We say that two $*$-homomorphisms $f,g:A\to B$ are \emph{$*$-homotopic}, and write $f\sim^*g$, if there is a finite sequence $f=f_0,\dots,f_n=g$ such that for each $i$ there is an elementary $*$-homotopy between $f_i$ and $f_{i+1}$. We write $[A,B]^*$ for the set of $*$-homotopy classes of $*$-homomorphisms $A\to B$.  

A $*$-ideal $I$ in a $*$-algebra $R$, denoted $I\triqui R$, is an $\ell$-$*$-subalgebra $I\subset R$ which is also a two sided ideal. An element $a$ in a $*$-algebra $R$ is \emph{self-adjoint} if $a^*=a$.

Let $X$ be a set; consider \emph{Karoubi's cone algebra} $\Gamma_X$ of square matrices indexed by $X$ with a finite coefficient set and a global bound on the size of the support of its rows and columns. Viewing $X\times X$-matrices with coefficients in $\ell$ as functions
$X\times X\to\ell$, the elements of $\Gamma_X$ are all the functions $a:X\times X\to\ell$ that satisfy the following conditions:
\[
|\im(a)|<\infty\,\text{ and }(\exists N) \max\{|\supp a(x,-)|,|\supp(a(-,x))|:x\in X\}\le N.
\]
 We equip $\Gamma_X$ with the product of matrices, or what is the same, with the convolution product of functions. Observe that $\Gamma_X$ contains the algebra of finitely supported matrices as an ideal
\[
M_X=\{a:X\times X\to \ell\colon\, |\supp(a)|<\infty\}\triqui\Gamma_X.
\]
The standard involution on $\Gamma_X$ is 
\[
*:\Gamma_X\to\Gamma_X,\, a^*(x,y)=a(y,x)^*\, (x,y\in X).
\]
Observe that $M_X$ is a $*$-ideal with respect to the standard involution. \emph{Karoubi's suspension algebra} is the quotient $*$-algebra
\[
\Sigma_X=\Gamma_X/M_X.
\]
If $(x,y)\in X\times X$, write $\epsilon_{x,y}$ for the matrix unit and $\iota_x:R\to M_XR$ for the corner embedding $\iota_x(a)=\epsilon_{x,x}a$. Observe that $\iota_x$ is a $*$-homomorphism for the standard involution on $M_X$ as well as for any other involution which makes $\epsilon_{x,x}$ self-adjoint. 
When $X=\{1,\dots,n\}$, we write $M_n$ for $M_X$. 
If $R$ is a $*$-algebra, we write $\Gamma_XR$ and $M_XR$ for $\Gamma_X\otimes R$ and $M_X\otimes R$. 

\subsection{\topdf{$\Z/2\Z$}{Z2}-gradings}\label{subsec:z2grad}

Let $G$ be an abelian group and $A=\bigoplus_{g\in G}A_g$ a $G$-graded algebra; the opposite algebra $A^{\op}$ is also $G$-graded with $A^{\op}_g=A_{-g}$. By a \emph{$G$-graded $*$-algebra} we understand a $G$-graded algebra equipped with an involution that $*:A\to A^{\op}$ is homogeneous of degree $0$. If $A=A_0\oplus A_1$ is a $\Z/2\Z$-graded $*$-algebra, then
\begin{equation}\label{map:eltau}
\tau:A\to A, \,\,\tau(a_0+a_1)=a_0-a_1
\end{equation}
is a $*$-automorphism. Composing it with the involution, we obtain a new involution 
\begin{equation}\label{invol}
\ol{\,}:A\to A^{\op},\,\, \ol{a_0+a_1}=a_0^*-a_1^*.
\end{equation}
Write $\ol{A}$ for $A$ equipped with the involution $\ol{\, }$. If $B=B_0\oplus B_1$ is another $\Z/2\Z$-graded $*$-algebra and 
$f:A\to B$ is a homogeneous $*$-homomorphism of degree $0$, then $f(\ol{a})=\ol{f(a)}$ for all $a\in A$. Hence the function $f$ defines a $*$-homomorphism 
\begin{equation}\label{map:olf}
\ol{f}:\ol{A}\to \ol{B},\,\, \ol{f}(a)=f(a).
\end{equation}
Note also that $A\mapsto \ol{A}$ commutes with tensor products of graded $*$-algebras; we have
\begin{equation}\label{eq:oltenso}
\ol{A\otimes B}=\ol{A}\otimes \ol{B}.
\end{equation}

\begin{ex}\label{ex:mpm} The algebra $M_2$ admits a $\Z/2\Z$-grading, where
$|\epsilon_{i,j}|\equiv i-j \mod (2)$. We write $M_{\pm}=\ol{M}_2$. 
 Thus $M_{\pm}$ is the algebra of $2\times 2$-matrices with coefficients in $\ell$, equipped with the involution 
\eqref{map:invopm}.
We write $\iota_+$ and $\iota_-$ for the upper left and lower right corner inclusions $\ell\to M_{\pm}$, respectively. For a $*$-algebra $R$, we put $M_{\pm}R=M_{\pm}\otimes R$ for the matrix algebra with the tensor product involution; we also abuse notation and write $\iota_{\pm}:R\to M_{\pm}R$, for $\iota_{\pm}\otimes\id_{R}$. It follows from \eqref{eq:oltenso} that 
\begin{equation}\label{eq:olmpm}
\ol{M_{\pm}R}=M_2\ol{R}. 
\end{equation}
\end{ex}

\begin{ex}\label{ex:x0x1} 
If $X$ is a set then any function $l:X\to \{0,1\}$ induces a $\Z/2\Z$-grading on $\Gamma_X$ that makes
$M_X$ a homogeneous ideal, where
$$
|\epsilon_{x,y}|\equiv l(x)+l(y) \mod (2).
$$
 
The grading of $M_2$ in Example \ref{ex:mpm} is  particular case of this. Observe also that, by passage to the quotient, we also obtain a $\Z/2\Z$-grading
on $\Sigma_X$. 
\end{ex}

\begin{ex}\label{ex:lemaptoolle}
If $A$ is any $\Z/2\Z$-graded $*$-algebra, then so are $M_2A$ and $M_{\pm}A$, with the tensor product grading. The involutions of
$M_2A$ and $M_\pm \ol{A}$ agree on the common $*$-subalgebra
\begin{equation}\label{fourier}
\hat{A}=\begin{bmatrix} A_0& A_1\\ A_1& A_0\end{bmatrix}
\end{equation}
We have a $*$-homomorphism
\begin{equation}\label{map:fourier}
A\to \hat{A},\,\, a_0+a_1\mapsto \begin{bmatrix} a_0& a_1\\ a_1& a_0\end{bmatrix}.
\end{equation}
Composing with the inclusions $\hat{A}\subset M_{\pm}\ol{A}$ and $\hat{A}\subset M_2A$ we get $*$-homomorphisms
\begin{equation}\label{map:delta} 
\Delta_A:A\to M_{\pm}\ol{A},\,\,\Delta'_A:A\to M_2A.
\end{equation}
A calculation shows that the following diagrams commute
\begin{align}\label{diag:delta}
\xymatrix{M_{2}A\ar[d]_{\Delta'_{M_{2}}}&A\ar[l]_(.4){\Delta'_A}\ar[dl]^{(\iota_1+\iota_2)\circ \Delta'_A} \\
 M_2M_2A&} &\xymatrix{A\ar[r]^{\Delta_A}\ar[d]_{(\iota_++\iota_-)\circ\Delta'_A}& M_{\pm}\ol{A}\ar[dl]^{\Delta_{M_{\pm}\ol{A}}}\\
M_{\pm}M_2A&}
\end{align}
\end{ex}
\begin{rem}\label{rem:ollegrade}
Let $\sigma$ be the generator of $\Z/2\Z$, written multiplicatively. Then
\[
\Z[\sigma]=\Z[\Z/2\Z]
\]
is a Hopf ring, and a $\Z/2\Z$-graded ring is the same thing as a comodule-algebra over $\Z[\sigma]$. One checks that the algebra $\hat{A}$ of \eqref{fourier} is the crossed product of $A$ with $\Z/2\Z$ under this coaction, as defined for example in \cite{gradstein}*{Definition 2.1}. In particular, $\hat{A}$ is equipped with a $\Z/2\Z$ action; the generator $\sigma$ acts by the automorphism
\[
\begin{bmatrix} a& b\\ c& d\end{bmatrix}\mapsto \begin{bmatrix} d& c\\ b& a\end{bmatrix}.
\] 
Observe that \eqref{map:fourier} is an isomorphism from $A$ onto the fixed ring under the automorphism above. Let $M=M_0\oplus M_1$ be a graded $A$-module; regarding an element $m_0+m_1\in M_0\oplus M_1$ as a column vector and using the matricial product, one obtains a $\hat{A}$-module $\hat{M}$ with the same underlying $\ell$-module $M$. Next assume that $A$ has graded local units in the sense of \cite{gradstein}*{Section 2.1}, and let $\Gr_{\Z/2\Z}\Mod A$ and $\Mod \hat{A}$ be the categories of graded and ungraded modules that are unital in the sense of \emph{loc.cit.} By \cite{gradstein}*{Proposition 2.5}, the functor 
\begin{equation}\label{map:gradfun}
\Gr_{\Z/2\Z}\Mod A\to \Mod \hat{A},\,\, M\mapsto \hat{M}
\end{equation}
is an isomorphism of categories, and maps the shift functor $M_*\mapsto M_{*+1}$ to the action of $\sigma$. 
\end{rem}

\subsection{Positivity}
Let $a\in A\in\ahas$ and $n\ge 1$; we call $a$ \emph{$n$-positive} if $a\ne 0$ and can be written as a sum $a=\sum_{i=1}^nx_ix_i^*$ for some $x_1,\dots,x_n\in A$, and \emph{positive} if it is $n$-positive for some $n$. We call $a$ \emph{negative} if $-a$ is positive. In a general $*$-algebra it can happen that an element is positive and negative at the same time. 

\begin{ex}\label{ex:field>0}
Assume that $\ell$ is a field of $\chara(k)\ne 2$. Then every self-adjoint element in $\ell$ can be written
as a difference of two $1$-positive elements. Hence $-1$ is positive if and only if every self-adjoint element of $\ell$ is positive. 
\end{ex}

\begin{ex}\label{ex:mpm-1}
The element
\[
x=\left[\begin{matrix}0& -1\\ 1&0\end{matrix}\right]\in M_{\pm}
\]
is self-adjoint and satisfies $x^2=-1$. Hence if $R$ is any unital $*$-algebra, then $-1$ is $1$-positive in $M_{\pm}R$.
\end{ex}

\begin{ex}\label{ex:zgr-1}
Let $L_1=\ell[t,t^{-1}]$ be the Laurent polynomials, with involution $t\mapsto t^{-1}$. A $\Z$-graded algebra $A=\bigoplus_{n\in\Z}A_n$ is an $L_1$-comodule algebra, and the comultiplication map $A\to A\otimes L_1$ which sends a homogeneous element $a$ to $at^{|a|}$ is a $*$-homomorphism. Now regard $A$ and $L_1$ as $\Z/2\Z$-graded, via their even/odd gradings. Then comultiplication defines a $*$-homomorphism
\[
c:\ol{A}\to A\otimes \ol{L}_1. 
\]
If now $R$ is a unital $*$-algebra and $x\in R$ a central element such that $xx^*=-1$, then $\mu_x:L_1\otimes R\to R$,
$t\otimes a\mapsto xa$ is a $*$-homomorphism. 
One checks that
\[
\theta_x:=(A\otimes \mu_x)\circ(c\otimes R):\ol{A}\otimes R\to A\otimes R
\] 
is a $*$-isomorphism with inverse $\theta_{x^{-1}}$. Similarly, the map
\[
\theta^x:\Hom_{\ahas}(A,R)\to \Hom_{\ahas}(\ol{A},R),\,\,\theta^x(f)=\mu_x\circ(f\otimes\ol{L_1})\circ c
\]
is bijective with inverse $\theta^{x^{-1}}$. 
\end{ex}

\begin{rem}\label{rem:xcentral} 
The hypothesis that $x$ be central --which is not satisfied in Example \ref{ex:mpm-1}-- is essential in Example \ref{ex:zgr-1}. For example $M_{\pm}\ol{A}$ and $M_{\pm}A$ are not isomorphic in general. Indeed, they have the same Hermitian $K_0$-groups $K_0^h$ as $\ol{A}$ and $A$, respectively, but in general, $K^h_0(A)\not\cong K^h_0(\ol{A})$; see Example \ref{ex:l=c}.  
\end{rem}

\begin{lem}\label{lem:-gamma}
Let $S$ be a set; equip $\Gamma_S$ with the standard involution. Then $-1$ is positive in $\Gamma_S$ if and only if it is positive in $\ell$. 
\end{lem}
\begin{proof}
The if direction is clear. Assume conversely that $-1$ is positive in $\Gamma_S$. Then there are elements $y(1),\dots, y(n)\in \Gamma_S$ such that $-1=\sum_{i=1}^ny(i)^*y(i)$. Hence for $N=\max\{|\supp(y(i)_{*,1})|:1\le i\le n\}$, we have the following identity between elements of $\ell$
\[
-1=(\sum_{i=1}^ny(i)^*y(i))_{1,1}=\sum_{i=1}^n\sum_{j=1}^N (y(i)_{j,1})^\ast y(i)_{j,1}.
\]
Thus $-1$ is positive in $\ell$.
\end{proof}

\subsection{Hermitian bivariant \topdf{$K$}{K}-theory}

In this subsection we assume that $\ell$ satisfies the $\lambda$-assumption \ref{stan:lambda}. 

An \emph{extension} of $*$-algebras is a sequence of $*$-homomorphisms
\begin{equation}\label{seq:ext}
(E)\qquad \xymatrix{A\ar@{ >-}[r]^{i}&B\ar@{>>}[r]^{p}&C}
\end{equation}
with $i$ injective and $p$ surjective and which is exact as a sequence of $\ell$-modules. An extension is \emph{semi-split} if it is split as a sequence of $\ell$-modules. Under our standing Assumption \ref{stan:lambda}, if \eqref{seq:ext} is semisplit, then there exists an involution preserving linear map $s:C\to B$ such that $p\circ s=\id_C$. 
Let $\fT$ be a triangulated category; write $[-n]$ for the $n$-fold suspension in $\fT$. Let $\cE$ be the class of all semi-split extensions \eqref{seq:ext}. 
An \emph{excisive homology theory} on $\ahas$ with values in $\fT$ is a functor $\cH:\ahas\to\fT$ together with a family of maps
$\{\partial_E:\cH(C)[1]\to \cH(A)| E\in\cE\}$ such that for every $E\in\cE$, 
\[
\xymatrix{\cH(C)[1]\ar[r]^{\partial_E}&\cH(A)\ar[r]^{\cH(i)}& \cH(B)\ar[r]^{\cH(p)}& \cH(C)}
\]
is a triangle in $\fT$, and such that $\partial$ is compatible with maps of extensions in the sense of \cite{ct}*{Section 6.6}. 
Let $\cH:\ahas\to\fT$ be an excisive homology theory, $X$ an infinite set, $x\in X$ and $A\in\ahas$. Consider the natural evaluation and corner inclusion maps $\ev_0:A[t]\to A$ and $\iota_+:A\to M_{\pm}A$,
$\iota_x:A\to M_XA$. We say that a $\cH$ is \emph{homotopy invariant}, \emph{$\iota_+$-stable} and \emph{$M_X$-stable} if for every $A\in\ahas$, the maps  $\cH(\ev_0)$, $\cH(\iota_+)$ and $\cH(\iota_x)$ are isomorphisms in $\fT$. By \cite{cv}*{Lemma 2.4.1}, $M_X$-stability is independent of the choice of the element $x$ in the previous definition. 
It was proved in \cite{cv}*{Proposition 6.2.7} that for any infinite set  $X$, there exists an excisive, homotopy invariant, $\iota_+$-stable and $M_X$-stable homology theory $j^h:\ahas\to kk^h$, depending on $X$, such that if $\cH:\ahas\to\fT$ is any other excisive, homotopy invariant, $\iota_+$-stable and $M_X$-stable homology theory, then there exists a unique triangulated functor $\bar{\cH}:kk^h\to\fT$ such that $\bar{\cH}\circ j^h=\cH$. We fix such an $X$ and for $A,B\in\ahas$ and $n\in \Z$, we write
\[
kk^h_n(A,B)=\hom_{kk^h}(j^h(A),j^h(B)[n]),\, kk^h(A,B)=kk^h_0(A,B). 
\]
The suspension in $kk^h$ is represented by the Karoubi suspension; for any infinite set $Y$ with $|Y|\le |X|$ and any $A\in\ahas$, we have
\[
j^h(A)[-1]=j^h(\Sigma_YA).
\]
The inverse suspension is obtained by tensoring with $\Omega=(1-t)t\ell[t]$; for all $A\in\ahas$ we have
\[
j^h(A)[+1]=j^h(\Omega A).
\]
It was shown in \cite{cv}*{Proposition 8.1} that $kk^h$ recovers Weibel-style homotopy algebraic Hermitian $K$-theory; we have
\begin{equation}\label{eq:kkh=khh}
kk^h_n(\ell,B)=KH^h_n(B). 
\end{equation}
For a definition of $KH^h$ and its relation to the more standard Hermitian $K$-theory defined by Karoubi (sometimes called Grothendieck-Witt theory) see \cite{cv}*{Section ~3}; $K_0^h$ is discussed in Section \ref{sec:k0}. There is a natural map
\begin{equation}\label{map:khtokhh}
K_n^h(B)\to KH^h_n(B).
\end{equation}
Recall that if $F$ is a functor defined on a subcategory $\fC\subset\aha$ closed under tensoring with $\ell[t]$ and $A\in\fC$, then 
$A$ is \emph{$F$-regular} if for every $n\ge 1$, $F$ maps the inclusion $A\subset A[t_1,\dots,t_n]$ to an isomorphism. $A$ is $K^h$ or $K$-regular if it is $K^h_m$ or $K_m$-regular for every $m$. The map \eqref{map:khtokhh} is an isomorphism for all $n$ whenever $B$ is $K^h$-regular. If $B=\inv(B_0)$ for some unital ring $B_0$, then $B$ is $K^h$-regular if and only if it is $K$-regular. If multiplication by $2$ is invertible in $B$ and $B$ is $K$-regular and unital 
then it is $K^h$-regular \cite{cv}*{Lemma 3.8}. The last two assertions hold more generally for $B$ $K$-excisive in the sense of \cite{cv}*{Section 3}. 
 
It follows from \eqref{eq:kkh=khh} that for all $A,B\in\ahas$, the groups $kk^h(A,B)$ are modules over the ring $kk^h(\ell,\ell)=KH^h_0(\ell)$, and thus over $K^h_0(\ell)$ via the canonical map $K^h_0(\ell)\to KH^h_0(\ell)$. As will be recalled in Section \ref{sec:k0} below, $K_0^h(\ell)$ is the group completion of the monoid of equivalence classes of self-adjoint idempotent finite matrices over $M_{\pm}$. There is a unital ring homomorphism $\Z[\sigma]\to K^h_0(\ell)$ mapping $1$ to the class of $\iota_+(1)$ and $\sigma$ to the class of $\iota_{-}(1)$. Thus $kk^h(A,B)$ is
a $\Z[\sigma]$-module for all $A,B\in\ahas$. If $-1$ is $1$-positive in $\ell$, then $\sigma$ acts trivially in $kk^h$; $\sigma\xi=\xi$ for all
$\xi\in kk^h(A,B)$ and all $A,B\in\ahas$. 

\begin{rem}\label{rem:kkvkkh}
Let $\ell_0$ be a commutative ring and let $j:\alg_{\ell_0}\to kk(\ell_0)$ be the universal homotopy stable, $M_X$-stable and excisive homology theory of \cite{ct}. Let $\ell=\inv(\ell_0)$ and let $j^h:\ahas\to kk^h(\ell)$ be the universal homotopy stable, $M_X$-stable, hermitian stable and excisive homology theory. Then by \cite{cv}*{Example 6.2.11}, the category equivalence $\inv:\alg_{\ell_0}\to\ahas$ of Example \ref{ex:inv} induces a triangle equivalence $kk(\ell_0)\to kk^h(\ell)$. Thus $kk$ is a particular case of $kk^h$. 

\end{rem}  
\begin{lem}\label{lem:hermstab}
Assume that $\ell$ satisfies the $\lambda$-assumption \ref{stan:lambda}. Let $X$ be an infinite set and let $j^h:\ahas\to kk^h$ be the universal excisive, homotopy invariant, $\iota_+$-stable and $M_X$-stable homology theory. 
Let $l:X\to \{0,1\}$ be a function and let $x\in X$. Equip $M_X$ with the $\Z/2\Z$-grading induced by $l$, as in Example \ref{ex:x0x1}. Consider the corner embedding $\ol{\iota}_x:\ell\to \ol{M}_X$, $\ol{\iota}_x(a)=\epsilon_{x,x}a$. Then $j^h(\ol{\iota}_x)$ is an isomorphism for every $x\in X$, and if $l(x)=l(y)$, then 
$j^h(\ol{\iota}_x)=j^h(\ol{\iota}_y)$.
\end{lem}
\begin{proof} 
Let $i\in\{0,1\}$, $X_i=l^{-1}\{i\}$ and let $\inc:M_{X_i}\to M_X$ be the map induced by the inclusion $X_i\subset X$. We have a commutative diagram
\[
\xymatrix{& \ol{M}_X\\
\ell\ar[ur]^{\ol{\iota}_x}\ar[r]_{\iota_x}&M_{X_i}.\ar[u]^{\inc}}
\]
The map $j^h(\iota_x)$ is an isomorphism by $M_X$-stability; by \cite{cv}*{Lemma 2.4.1} it is independent of $x\in X_i$. The map $j^h(\inc)$ is an isomorphism by \cite{cv}*{Lemma 2.4.3}.
\end{proof}

\begin{prop}\label{prop:deltatau}
Let $A$ be a $\Z/2\Z$-graded $*$-algebra and let $\Delta=\Delta_A:A\to M_{\pm}\ol{A}$ and $\Delta'=\Delta_A':A\to M_2A$ be as in \eqref{map:delta}. Set
$\ob{\Delta_A}=j^h(\iota_+)^{-1}\circ j^h(\Delta_A)\in kk^h(A,\ol{A})$ and $\ob{\Delta'_A}=j^h(\iota_1)^{-1}\circ j^h(\Delta'_A)\in kk^h(A,A)$. Then
\[
{\ob{\Delta}}_{\ol{A}}\circ \ob{\Delta_A}=(1+\sigma)\circ\ob{\Delta'_A}, \quad\quad {\ob{\Delta'_A}}\circ{\ob{\Delta'_A}}=2\ob{\Delta'_A}.
\]
If furthermore $2$ is invertible and $1$-positive in $\ell$, then for $\tau$ as in \eqref{map:eltau}, we have $\ob{\Delta'}=1+j^h(\tau)$.
\end{prop}
\begin{proof} The displayed identities follow from the commutative diagrams \eqref{diag:delta} and naturality of $\Delta$, using Lemma \ref{lem:hermstab}. Next assume that $2=xx^*$ for some invertible $x\in \ell$, and consider
\[
u=\begin{bmatrix} 1/x& 1/x\\ 1/x^*& -1/x^*\end{bmatrix}.
\]
A calculation shows that $\ad(u)\circ\Delta'=\iota_1+\iota_2\tau$, whence $\ob{\Delta'}=1+j^h(\tau)$. 
\end{proof} 
In the next corollary and elsewhere we write $kk^h[1/2]$ for the idempotent completion of the Verdier quotient of $kk^h$ by the full subcategory of those objects $C$ such that $j^h(\id_C)$ is $2$-torsion, $j^h[1/2]:\ahas\to kk^h\to kk^h[1/2]$ for the composite of the canonical functors, and $kk^h[1/2](A,B)=\hom_{kk^h[1/2]}(j^h[1/2](A),j^h[1/2](B))$ for $A,B\in\ahas$. 

\begin{coro}\label{coro:deltatau}
Let $p_+=(1+\sigma)/2\in \Z[1/2,\sigma]$. If $A$ is $\Z/2\Z$-graded then $({\ob{\Delta}}_A/2)$ and $({\ob{\Delta}}_{\ol{A}}/2)$ induce inverse isomorphisms $\im(p_+\circ({\ob{\Delta'}}_A/2))\leftrightarrows \im(p_+\circ({\ob{\Delta'}}_{\ol{A}}/2))$. 
\end{coro}

\bigskip

\subsection{Cohn and Leavitt path algebras}

A (directed) \emph{graph} is a quadruple $E=(s,r:E^1\rightrightarrows E^0)$ consisting of sets $E^0$ and $E^1$ of \emph{vertices} and \emph{edges}
and \emph{source} and \emph{range} maps $s$ and $r$. A vertex $v\in E^0$ is a \emph{sink} if $s^{-1}(\{v\})=\emptyset$, a \emph{source} if $r^{-1}(\{v\})=\emptyset$, an \emph{infinite emitter}
if $s^{-1}(\{v\})$ is infinite and a \emph{singular} vertex it is either a sink or an infinite emitter. Vertices which are not singular are called \emph{regular}. We write $\sink(E),\sour(E),\inf(E)$ for the sets of sinks, sources, and infinite emitters, and $\sing(E)$ and $\reg(E)$ for those of singular and regular vertices. We say that $E$ is \emph{regular} if $E^0=\reg(E)$. A graph $E$ is \emph{countable} or \emph{finite} if both $E^0$ and $E^1$ are. The \emph{reduced incidence matrix} of a graph $E$ is the matrix $A_E$ with nonnegative integer coefficients, indexed by $\reg(E)\times E^0$, whose $(v,w)$ entry is the number of edges with source $v$ and range $w$:
\[
(A_E)_{v,w}=|s^{-1}(v)\cap r^{-1}(w)|.
\]
Our conventions are such that we will mainly deal with the transpose $A_E^t$. We abuse notation and write $I$ for the ${E^0\times \reg(E)}$-matrix obtained from the identity matrix of $M_{E^0}\Z$ upon removing the columns corresponding to the singular vertices. Thus $I-A_E^t$ is a well-defined integral matrix indexed by $E^0\times\reg(E)$. 

We write $C(E)$ and $L(E)$ for the Cohn and Leavitt path algebras over $\ell$ \cite{lpabook}*{Definitions 1.5.1 and 1.2.3}. 
Each of these carries a standard $\ell$-semilinear involution $a\mapsto a^*$ which fixes the vertices and maps each edge $e$ to the corresponding phantom edge $e^*$. Let $\cP=\cP(E)$ be the set of all finite paths in $E$ (\cite{lpabook}*{Definitions 1.2.2}); we write $|\alpha|$ for the \emph{length} of a path $\alpha\in\cP(E)$. For $v\in E^0$, set
\begin{equation}\label{parribajo}
\cP_v=\{\mu \in \cP \mid r(\mu) = v \}, \quad \cP^v=\{\mu\in\cP\mid s(\mu)=v\}.
\end{equation}
Let 
\begin{gather}\label{map:rho}
\rho:C(E)\to \Gamma_\cP,\\
\rho(v)=\sum_{\alpha\in\cP^v}\epsilon_{\alpha,\alpha},\quad \rho(e)=\sum_{\alpha\in \cP^{r(e)}}\epsilon_{e\alpha,\alpha},\nonumber\\
\rho(e^*)=\sum_{\alpha\in\cP^{r(e)}}\epsilon_{\alpha,e\alpha},\quad (v\in E^0, e\in E^1).\nonumber
\end{gather}
Observe that $\rho$ is a $*$-homomorphism for the standard involutions on $C(E)$ and $\Gamma_{\cP(E)}$. Recall that $C(E)$ carries a natural $\Z$-grading, where $C(E)_n$ is generated by all $\alpha\beta^*$ with $|\alpha|-|\beta|=n$. Hence we may regard $C(E)$ as $\Z/2\Z$-graded, via the even/odd grading. The twisted involution $\bar{\,}$ of \eqref{invol} is the algebra homomorphism
\begin{gather}\label{map:*twist}
\bar{\,}:C(E)\to C(E)^{\op},\, \ol{v}=v,\, \ol{e}=-e^*,\,\ol{e^*}=-e,\,\,(v\in E^0,e\in E^1).
\end{gather}
Similarly, the length modulo $2$ induces a $\Z/2\Z$-grading on $\Gamma_{\cP(E)}$, and $\rho$ is homogeneous for this grading. Because $\rho$ is homogeneous and a $*$-homomorphism for the standard involution, it defines a $*$-homomorphism $\ol{\rho}:\ol{C(E)}\to \ol{\Gamma}_{\cP(E)}$ as in \eqref{map:olf}.

Let $\cK(E)=\ker(C(E)\to L(E))$ be the kernel of the canonical surjection. We have semi-split extensions of $*$-algebras
\begin{gather}\label{seq:cohnext}
0\to \cK(E)\to C(E)\to L(E)\to 0\\
0\to \ol{\cK(E)}\to \ol{C(E)}\to \ol{L(E)}\to 0.\label{seq:cohnext2}
\end{gather}

An important feature of the involution $a\mapsto \ol{a}$ is that if $E$ is finite and regular, then the following identity holds in $\ol{L(E)}$
\begin{equation}\label{eq:-1ole}
-1=\sum_{e\in E^1}e\ol{e}.
\end{equation}
This says that $-1$, which is clearly negative, is also positive in $\ol{LE}$. By contrast, $-1$ is positive in $L(E)$ if and only if this happens already in $\ell$, as shown by the next lemma.

\begin{lem}\label{lem:-1}
Let $E$ be a graph with finitely many vertices. Then $-1$ is positive in $L(E)$ if and only if it is positive in $\ell$. If $-1$ is $1$-positive in $\ell$, then $\ol{L(E)}\cong L(E)$ in $\ahas$. 
\end{lem}
\begin{proof} The if direction is clear. To prove the converse, it suffices, in view of Lemma 
\ref{lem:-gamma}, to find a set $X$ and a unital $*$-homomorphism $f:L(E)\to \Gamma_X$. Let $\cS(E)=\{\alpha\beta^*:\alpha,\beta\in \cP(E)\}$ be the inverse semigroup associated to $E$. Let $X$ be an infinite set of cardinality $|X|\ge |E^1|$ and $\cI(X)$ the inverse semigroup of all partially defined injections $X\supset\dom(f)\overset{f}\lra X$. Proceed as in the proof of \cite{cr}*{Proposition 4.11} to find a
semigroup homomorphism $\mu:\cS(E)\to\cI(X)$ such that the associated action of $\cS(E)$ on $X$ is tight in the sense of Exel \cite{cr}*{Section 3}. By \cite{cr}*{Lemma 3.1} and \cite{ac}*{Lemma 6.1}, $\mu$ induces an algebra homomorphism $L(E)\to \Gamma_X$. One checks further that $\mu$ is a $*$-homomorphism. This completes the proof of the first assertion. The second assertion is immediate from Example \ref{ex:zgr-1}. 
\end{proof}

\section{Leavitt path algebras in \topdf{$kk^h$}{kkh}}\label{sec:lpaskkh}
\numberwithin{equation}{section}

This section is concerned with establishing the triangles \eqref{intro:xtri1} and \eqref{intro:xtri2}, the first of which is analogous to that for $C^*$-algebras established by Joachim Cuntz in \cite{ck2}. The general strategy for proving these results, using the presentation of the Leavitt path algebra as a quotient of the Cohn algebra, is similar to that used in \cite{cm2} to establish analogous ones in $kk$, and ultimately goes back to the pioneering work of Cuntz on the computation of $K$-theory of Cuntz algebras \cite{on}. However some technical difficulties appear in that we need all maps and homotopies to preserve involutions.

\bigskip

Throughout this section, we assume that $\ell$ satisfies the $\lambda$-assumption \ref{stan:lambda}.

\bigskip

For a set $X$ and a $*$-algebra $R$, we write $R^X$ for the $*$-algebra of all functions $X\to R$ with pointwise operations and pointwise involution, and $R^{(X)}\subset R^X$ for the $*$-ideal of finitely supported functions. If $x\in X$ and $a\in R$, we write $a\chi_x$ for the function supported in $\{x\}$ which maps $x\mapsto a$.

Let $E$ be a graph and $C(E)$ the Cohn algebra of $E$. The assignment  
\begin{equation}\label{map:elphi}
\ell^{(E^0)} \to C(E),\, \chi_v\mapsto v
\end{equation}
defines $*$-homomorphisms $\phi:\ell^{(E^0)}\to C(E)$ and $\ol{\phi}:\ell^{(E^0)}\to\ol{C(E)}$. 

We shall say that a homology theory is \emph{$E$-stable} if it is stable with respect to a set $X$ of cardinality $|E^0\sqcup E^1\sqcup\N|$.  

\begin{thm}\label{thm:cekk} Assume that $\ell$ satisfies the $\lambda$-assumption \ref{stan:lambda}. Let $j^h:\ahas\to kk^h$ be the universal homotopy invariant, excisive, Hermitian stable and $E$-stable homology theory. 
Let $E$ be a graph and let $\phi$ be as in \eqref{map:elphi}. Then $j^h(\phi)$ and $j^h(\ol{\phi})$ are isomorphisms. In particular, 
$j^h(C(E))\cong j^h(\ol{C(E)})$. 
\end{thm}
\begin{proof}
The analogue statement for the universal homotopy invariant, excisive and $E$-stable homology theory $j:\aha\to kk$ was proved in \cite{cm1}*{Theorem 4.2}. The same proof goes through here with minor adjustments and works for both choices of involution on $C(E)$. The adjustments are in the homotopies occurring in \cite{cm1}*{Lemma 4.17} and \cite{cm1}*{Lemma 4.21}, which are not $*$-homomorphisms. In both cases the problem is fixed by applying the trick of \cite{cv}*{Lemma 5.4}, as we shall explain presently. Lemma 4.17 of \cite{cm1} says that a certain map 
$\hat{\iota}_\tau:C(E)\to M_{\cP(E)}C(E)$ is sent by $j^h$ to the same isomorphism as $\iota_{\alpha}$ for any $\alpha\in\cP(E)$. Observe that Lemma \ref{lem:hermstab} guarantees that $j^h(\ol{\iota}_\alpha)=j^h(\ol{\iota}_\beta)$ only when $|\alpha|\equiv|\beta|\mod (2)$. This is however not a problem, as the lemma in \cite{cm1} is needed only to establish that $j^h(\hat{\iota}_\tau)$ is an isomorphism, so it suffices to prove it when $|\alpha|=0$. In the proof of \cite{cm1}*{Lemma 4.17}, a vertex $w\in E^0$ is fixed and elements $A_v,B_v\in M_{\cP(E)}C(E)$ are defined for each $v\in E^0$. With notation as in \cite{cv}*{Lemma 5.4}, put $C_v=c(A_v,B_v)$ and let $H:C(E)\to M_{\pm}M_{\cP(E)}C(E)[t]$ be the homomorphism determined by $H(v)=C_v\iota_+(\epsilon_{v,v}\otimes v)C_v^*$, $H(e)=C_{s(e)} \iota_+(\epsilon_{s(e),r(e)}\otimes e) C_{r(e)}^*$, $H(e^*)=H(e)^*$. One checks that $H$ is a $*$-algebra homomorphism for both choices of involution. It follows that $H$ is an elementary $*$-homotopy between the composite of $\iota_+$ with the maps $\hat{\iota}_\tau$ and $\iota_w$ of \cite{cm1}*{Lemma 4.17}, again for both choices of involution. Next we pass to the analogue of \cite{cm1}*{Lemma 4.21}. With notations as in \emph{loc.cit.}, for each $e\in E^1$ consider the following elements of $\fA[t]$
\begin{align*}
U_e=&\epsilon_{s(e),s(e)} (1-t^2) ee^* + \epsilon_{e,s(e)}te^*,\\
V_e=&\epsilon_{s(e),s(e)}(1-t^2)ee^*+\epsilon_{s(e),e}(2t-t^3)e.
\end{align*}
One checks that the homotopy $H^+:C(E)\to D[t]$ defined in the proof of \cite{cm1}*{Lemma 4.21} satisfies the following identity for each $e\in E^1$. 
\[
H^+(e)=(em_{r(e)},U_e\epsilon_{s(e),r(e)}e),\, H^+(e^*)=(m_{r(e)}e^*,\epsilon_{r(e),s(e)}e^*V_e).
\]
Put $W_e=c(U_e,V_e)$. Let 
\begin{gather*}
H:C(E)\to M_{\pm}D[t],\\ H(e)=\iota_+(em_{r(e)},0)+(0,W_e)(0,\iota_+(\epsilon_{s(e),r(e)}e)),\\
H(e^*)=H(e)^*,\, H(v)=(m_v,0)\quad (v\in E^0,\, e\in E^1).
\end{gather*}
One checks that $H$ is a  $*$-algebra homomorphism for both choices of involution, so that for both choices of involution, $H$ is a $*$-homotopy between the maps $\psi_0$ and $\psi_{1/2}$ of \cite{cm1}*{Lemma 4.21}. This finishes the proof.
\end{proof}

\begin{rem}\label{rem:ackvega} The fact that $j^h(\phi)$ is an isomorphism, proved in Theorem \ref{thm:cekk}, is joint work with Santiago Vega, and is part of his PhD thesis, \cite{santi}.
\end{rem}
Let $S$ be a set, $\fT$ a triangulated category, and $\cH:\ahas\to\fT$ an excisive homology theory. We say that $\cH$ is
\emph{$S$-additive} if direct sums of at most $|S|$ factors exist in $\fT$, and for any set $T$ with cardinality $|T|\le |S|$ and any family of $*$-algebras $\{A_t:t\in T\}$, the canonical map $\bigoplus_{t\in T}\cH(A_t)\to \cH(\bigoplus_{t\in T}A_t)$ is an isomorphism.

\begin{thm}\label{thm:gentri}
 Assume that $\ell$ satisfies the $\lambda$-assumption \ref{stan:lambda}. Let $X:\ahas\to\fT$ be an excisive, homotopy invariant, Hermitian stable, $E$-stable and $E^0$-additive homology theory and let $R\in\ahas$. Then \eqref{seq:cohnext} and \eqref{seq:cohnext2} induce the following distinguished triangles in $\fT$
\begin{gather*}
\xymatrix{X(R)^{(\reg(E))}\ar[r]^{I-A_E^t}& X(R)^{(E^0)}\ar[r]& X(L(E)\otimes R)}\\
\xymatrix{X(R)^{(\reg(E))}\ar[r]^(.55){I-\sigma A_E^t}& X(R)^{(E^0)}\ar[r]& X(\ol{L(E)}\otimes R)}.
\end{gather*} 
\end{thm}
\begin{proof}
It follows from the universal property of $j^h$ that $\otimes R$ induces a triangulated functor in $kk^h$. Hence upon replacing $X$ by $X(-\otimes R)$ if necessary, we may assume that $R=\ell$.  
Let $\inc:\cK(E)\to C(E)$ and $\ol{\inc}:\ol{\cK(E)}\to \ol{C(E)}$ be the inclusions. For each vertex $w\in E^0$, let $p_w:\ell^{(E^0)}\leftrightarrows \ell:\chi_w$ be the projection onto the $w$-coordinate and the inclusion into the $w$-summand. If $v\in \reg(E)$, write $m_v=\sum_{s(e)=v}ee^*\in C(E)$ and let $q:\ell^{(\reg(E))}\to \cK(E)$, $v\mapsto q_v=v-m_v$. We shall abuse notation and write $q_v$ also for the map $\ell\to \cK(E)$, $v\mapsto q_v$. In view of Theorem \ref{thm:cekk} and the additivity hypothesis on $X$, it suffices to identify, for each pair $(v,w)\in\reg(E)\times E^0$, the composite $j^h(p_w)j^h(\phi)^{-1}j^h(\inc\circ q_v)\in KH^h_0(\ell)$, and similarly for $\ol{\phi}$, $\ol{\inc}$ and $\ol{q_v}$ substituted for $\phi$, $\inc$ and $q_v$. In the case of the standard involution, for each $e\in E^1$ the projection $ee^*$ is M-vN equivalent to $e^*e=r(e)$ and the same calculation as in \cite{cm1}*{Proposition 5.2}
goes through. However in the case of the twisted involution, this is no longer true. In fact, taking into account \eqref{map:invopm}
and \eqref{map:*twist} and writing ${}^*$ for the involution of $M_\pm \ol{C(E)}$, we obtain the following identities
\begin{align}\label{eq:ee*re}
\iota_+(ee^*)=&\left[\begin{matrix}0& e\\ 0& 0\end{matrix}\right]\cdot \left[\begin{matrix}0& 0\\ e^*& 0\end{matrix}\right]\nonumber\\
             =&\left[\begin{matrix}0& e\\ 0& 0\end{matrix}\right]\cdot\left[\begin{matrix}0& e\\ 0& 0\end{matrix}\right]^*\\
\iota_-(r(e))=&\left[\begin{matrix}0& e\\ 0& 0\end{matrix}\right]^*\cdot\left[\begin{matrix}0& e\\ 0& 0\end{matrix}\right]\nonumber
\end{align}
It follows that $[ee^*]=\sigma [r(e)]$ in $KH^h_0(\ol{C(E)})$, so from the orthogonal sum $v=q_v+m_v$ we get that 
$j^h(\ol{\inc q_v})=j^h(\ol{\phi\chi_v})-\sum_{s(e)=v}\sigma j^h(\ol{\phi \chi_{r(e)}})$. Hence $j^h(\ol{\phi})^{-1}\circ j^h(\ol{\inc})\circ j^h(\ol{q})=I-\sigma A_E^t$. 
\end{proof}

\begin{thm}\label{thm:fundtriang}
Let $E$ be a graph. Assume that $|E^0|<\infty$ and that $\ell$ satisfies the $\lambda$-assumption \ref{stan:lambda}. Then there are distinguished triangles in $kk^h$
\begin{equation*}
\xymatrix{
j^h(\ell^{\reg(E)})\ar[r]^{I-A_E^t}&j^h(\ell^{E^0})\ar[r]& j^h(LE)\\
j^h(\ell^{\reg(E)})\ar[r]^{I-\sigma A_E^t}&j^h(\ell^{E^0})\ar[r]& j^h(\ol{LE}).
}
\end{equation*}
\end{thm}

\begin{proof}
Apply Theorem \ref{thm:gentri} to $X=j^h$. 
\end{proof}

\begin{coro}\label{coro:fundtau}
Let $E$ be as in Theorem \ref{thm:fundtriang} $\tau:L(E)\to L(E)$ as in \eqref{map:eltau} and $\ol{\tau}$ as in \eqref{map:olf}. Then
$$2(j^h(\tau)-j^h(\id_{L(E)}))=2(j^h(\ol{\tau})-j^h(\id_{\ol{L(E)}}))=0.$$ 
\end{coro}
\begin{proof} The restrictions of $\tau:C(E)\to C(E)$ to the images of $\phi:\ell^{E^0}\to C(E)$ and $q:\ell^{\reg(E)}\to \cK(E)\subset C(E)$ are the identity maps. Hence writing $1$ for all identity maps, we have a map of triangles
\[
\xymatrix{j^h(\ell^{\reg(E)})\ar[d]^1\ar[r]^{I-A_E^t}&j^h(\ell^{E^0})\ar[d]^1\ar[r]^p& j^h(LE)\ar[d]^{j^h(\tau)}\ar[r]^(.4)\partial&j^h(\ell^{\reg(E)})[-1]\ar[d]^1\\
j^h(\ell^{\reg(E)})\ar[r]^{I-A_E^t}&j^h(\ell^{E^0})\ar[r]^p& j^h(LE)\ar[r]^(.4)\partial& j^h(\ell^{\reg(E)})[-1]}
\]
From the exact sequences obtained by applying $kk^h(LE,-)$ and $kk^h(-,LE)$ to the triangles above we obtain factorizations $1-j^h(\tau)=p\circ\xi=\eta\circ\partial$. It follows that 
\[
0=\eta\circ\partial\circ p\circ \xi=(1-j^h(\tau))^2=2(1-j^h(\tau)).
\]
The same argument shows that $2(1-j^h(\ol{\tau}))=0$. 
\end{proof}
Let $E$ be a finite graph; let $B_E\in \{ 0,1 \}^{E^1 \times (E^1 \sqcup \sink(E))}$,
\begin{equation}\label{mat:B}
(B_E)_{e,x} =\left\{\begin{matrix} \delta_{r(e),s(x)}& x\in E^1\\ \delta_{r(e),x} & x\in\sink(E)\end{matrix}\right.
\end{equation}
Also let $J\in \Z^{(E^1\sqcup\sink(E))\times E^1}$,
\begin{equation}\label{mat:IB}
J_{x,e}=\left\{\begin{matrix}\delta_{s(x),r(e)}& x\in E^1\\ \delta_{x,r(e)} & x\in\sink(E)\end{matrix}\right.
\end{equation}
\begin{coro}\label{coro:fundtriang}
Let $E$ be a finite graph and $B_E$ and $J$ as in \eqref{mat:B} and \eqref{mat:IB}. Then there are distinguished triangles in $kk^h$
\[
\xymatrix{
\ell^{E^1}\ar[r]^(.4){J-B_E^t}&\ell^{E^1\sqcup\sink(E)}\ar[r]& LE\\
\ell^{E^1}\ar[r]^(.4){J-\sigma B_E^t}&\ell^{E^1\sqcup\sink(E)}\ar[r]& \ol{LE}.
}
\]
\end{coro}
\begin{proof}
As observed in \cite{cm1}*{Remark 5.7}, $B_E=A_{E_s}$ for the out-split graph $E_s$ of \cite{lpabook}*{Definition 6.3.23}.
An explicit $\Z$-graded $*$-algebra isomorphism $f:LE\iso L(E_s)$ is constructed in the proof of \cite{aalp}*{Theorem 2.8}; one checks that 
$f(\ol{a})=\ol{f(a)}$ for all $a\in L(E)$. Given all this, the corollary is immediate from Theorem \ref{thm:fundtriang}.  
\end{proof}

\section{Hermitian \topdf{$K$}{K}-theory and Bowen-Franks groups}\label{sec:kkhbf}

The \emph{Bowen-Franks} group of a graph $E$ is
\[
\BF(E)=\coker(I-A_E^t).
\]
We shall also consider the following $\Z[\sigma]$-module
\[
\ol{\BF}(E)=\coker(I-\sigma A_E^t).
\]
\begin{rem}\label{rem:olbfreg}
Let $E$ be a regular graph and let $E^2$ be the graph with the same vertices and where an edge is a path of length $2$ in $E$. Then we have a group isomorphism 
\[
\ol{\BF}(E)\cong\coker(I-(A_E^t)^2)=\BF(E^2).
\]
Under the isomorphism above, the action of $\sigma$ becomes multiplication by $A_E^t$. 
\end{rem}

\begin{thm}\label{thm:khseq} 
Let $E$ be a graph and $R\in\ahas$. Then there are exact sequences

\begin{gather*}
0\to \BF(E)\otimes KH_n^h(R)\to KH^h_n(L(E)\otimes R)\to \ker((I-A_E^t)\otimes KH^h_{n-1}(R))\to 0\\
0\to \ol{\BF}(E)\otimes_{\Z[\sigma]} KH_n^h(R)\to KH^h_n(\ol{L(E)}\otimes R)\to \ker((I-\sigma A_E^t)\otimes KH^h_{n-1}(R))\to 0
\end{gather*}

\end{thm}
\begin{proof}
Apply Theorem \ref{thm:gentri} to the functor $KH^h$ from $\ahas$ to the homotopy category of spectra that sends $A\in\ahas$ to the homotopy Hermitian $K$-theory spectrum $KH^h(A)$; then take homotopy groups. 
\end{proof}

\begin{lem}\label{lem:lekreg}
Let $E$ be a countable graph and $R$ a unital algebra. If $R$ is regular supercoherent, then $L(E)\otimes R$ is 
$K$-regular. If moreover $2$ is invertible in $\ell$, then $L(E)\otimes R$ is $K^h$-regular. 
\end{lem}
\begin{proof} The hypothesis on $R$ implies that $R[t_1,\dots,t_n]$ is regular supercoherent for all $n$, by the argument of \cite{abc}*{beginning of Section 7}.  Hence $L_\Z(E)\otimes_\Z R=L(E)\otimes R$ is $K$-regular for every regular supercoherent $R$ and every row-finite graph $E$, by \cite{abc}*{Theorem 7.6}. In the general case, because $E$ is countable, we may choose a  desingularization $E_\delta$ of $E$ as in \cite{aap}*{Section 5}. Theorem 5.6 of \cite{aap} says that if $\ell$ is a field, then $L(E)$ and $L(E_\delta)$ are Morita equivalent as rings with local units; the proof actually works over arbitrary commutative $\ell$, and in particular for $\ell=\Z$. We claim that $M_\infty L_\Z(E)\cong M_\infty L_\Z(E_\delta)$. Since $L(E)\otimes R=L_\Z(E)\otimes_\Z R$, the claim together with the row-finite case already proved shows that $L(E)\otimes R$ is $K$-regular. By \cite{amorir}*{Theorem 5 and Remark 1}, two rings $A$ and $B$ with countable systems of local units which are topologically projective in the sense of \cite{amorir}*{Definition on page 409} are Morita equivalent if and only if $M_\infty A\cong M_\infty B$. Hence it suffices to show that $L_\Z(E)$ is topologically projective for any graph $E$. Let $U=\bigoplus_{v\in E^0}LEv$ and let $\pi:U\to LE$ be the sum of the canonical inclusions. Let $s:L(E)\to U$, $s(a)_v= av$. For $a\in L(E)$, put $E\supset F=\supp(s(a))$ and  $p=\sum_{v\in F}v$. Then $s(a+b-bp)_v=s(a)_v$ for all $b\in L(E)$ and all $v\in F$. This proves that $L(E)$ is topologically projective, concluding the proof of the first assertion of the lemma. The second assertion follows from the first using \cite{cv}*{Lemma 3.8}.
\end{proof}

\begin{coro}\label{coro:kseq}
Let $R\in\ahas$ be unital and $E$ a countable graph. Assume that $2$ is invertible in $R$ and that $R$ is regular supercoherent. Then the maps $K^h_*(L(E)\otimes R)\to KH^h_*(L(E)\otimes R)$ and $K^h_*(\ol{L(E)}\otimes R)\to KH^h_*(\ol{L(E)}\otimes R)$ of \eqref{map:khtokhh} are isomorphisms and the exact sequences of Theorem \ref{thm:khseq} also hold for $L(E)\otimes R$ and $\ol{L(E)}\otimes R$ with $KH^h$ replaced by $K^h$. 
\end{coro}
\begin{proof}
 Because $R$ is unital, $L(E)\otimes R$ has local units and therefore $K$-excisive in the sense of \cite{cv}*{Section 3}. Now apply \cite{cv}*{Lemma 3.8}.
\end{proof}

\begin{coro}\label{coro:kseq2}
Assume that $E$ is countable and that $\ell$ is a field of characteristic $\chara(\ell)\ne 2$. Then
\[
K^h_0(L(E))=\BF(E)\otimes_\Z K_0^h(\ell), \, K^h_0(\ol{L(E)})=\ol{\BF}(E)\otimes_{\Z[\sigma]}K_0^h(\ell). 
\]
\end{coro}
\begin{proof}
In view of Corollary \ref{coro:kseq}, it suffices to show that $K^h_{-1}(\ell)=0$. By \cite{walter}*{Theorems 7.1 and 8.1} and \cite{marcofest}*{Proposition 6.3} $K^h_{-1}(\ell)$ agrees with Ranicki's $U_{-1}(\ell)$ which vanishes by \cite{ranisemi}*{Proposition}.
\end{proof}

\begin{ex}\label{ex:l=c}
Let $\ell$ be a field of $\chara(\ell)\ne 2$. Assume that the canonical map $\Z[\sigma]\to KH_0^h(\ell)$ is an isomorphism. It follows from Corollary \ref{coro:kseq2} that $K_0^h(\ol{L(E)})=\ol{\BF}(E)$ and $K_0^h(L(E))=\BF(E)\otimes\Z[\sigma]$. This is the case, for example, when $\ell=\C$ equipped with complex conjugation as involution. 
\end{ex}

\begin{lem}\label{lem:phantom}
Let $E$ and $\ell$ be as in Theorem \ref{thm:fundtriang}. Then 
\item[i)] $j^h(\ol{L(E)})=0$ $\iff$ $\ol{\BF}(E)\otimes_{\Z[\sigma]}KH_0^h(\ell)=0$.
\item[ii)] $j^h(L(E))=0\iff \BF(E)\otimes KH_0^h(\ell)=0$.
\item[iii)] $j^h(\ol{L(E)})=0\Rightarrow j^h(L(E))=0$.
\item[iv)] If $j^h(L(E))=0$ then $E$ is regular. 
\end{lem}
\begin{proof} By Theorem \ref{thm:fundtriang}, $j^h(L(E))=0$ if and only if the image of $I-A_E^t$ in
$KH_0^h(\ell)^{E^0\times \reg(E)}$ under the map induced by $\Z\subset \Z[\sigma]\to KH_0^h(\ell)$ is an invertible matrix. Because $KH_0^h(\ell)$ is a commutative ring, the latter condition implies that
$\reg(E)=E^0$ and is equivalent
to having $\BF(L(E))\otimes KH_0^h(\ell)=0$. Similarly, $j^h(\ol{L(E)})=0$ is equivalent to $\ol{\BF}(L(E))\otimes_{\Z[\sigma]}KH_0^h(\ell)=0$ and implies that $j^h(L(E))=0$, since by definition $\BF(E)$ is a quotient of $\ol{\BF}(E)$.
\end{proof}

\begin{ex}\label{ex:phantom}
Let $\booh$ be the following graph
\[
\xymatrix{
\bullet\ar@(dl,ul)[]\ar@/^/[r]&\bullet\ar@/^/[l]
}
\]
Then $\det(I-\sigma A_{\booh}^t)=-\sigma$, so $j^h(\ol{L(\booh)})=j^h(L(\booh))=0$. 
\end{ex}

\section{Hermitian \topdf{$K$}{K}-theory of \topdf{$\ol{L(E)}$}{LE} vs. \topdf{$\Z/2\Z$}{Z2}-graded \topdf{$K$}{K}-theory of \topdf{$L(E)$}{LE}}\label{sec:nonex}

Let $E$ be a graph and let $\hat{E}$ be the graph with $\hat{E}=E^i\times\Z/2\Z$, $(i=0,1)$ with source and range functions
$s(e,i)=(s(e),i)$, $r(e,i)=(r(e),i+1)$. Observe that $\Z/2\Z$ acts on $\hat{E}$ by translation on the second component. We have an isomorphism of $\Z[\sigma]$-modules $\Z^{(\hat{E}^0)}\cong \Z[\sigma]\otimes \Z^{(E^0)}$ which restricts to an isomorphism
$\Z^{\reg(\hat{E})}\cong \Z[\sigma]\otimes \Z^{(\reg(E))}$. Under these identifications, $A_{\hat{E}}$ becomes $\sigma A_E$. Hence we have an isomorphism of $\Z[\sigma]$-modules
\begin{equation}\label{olbf=bfe2}
\BF(\hat{E})\cong\ol{\BF}(E).
\end{equation}
Hence by \cite{gradstein}*{Corollary 5.3} and Remark \ref{rem:ollegrade}, for $\widehat{L(E)}$ as in Example \ref{ex:lemaptoolle}, we have a $*$-isomorphism
\begin{equation}\label{map:grade}
L(\hat{E})\iso \widehat{L(E)}.
\end{equation}
Write $K^{\gr}$ for the graded $K$-theory of $\Z/2\Z$-rings. It follows from \eqref{map:grade}, Remark \ref{rem:ollegrade} and Theorem \ref{thm:khseq} applied to $\inv(\ell)$, that if $E$ has finitely many vertices and $\ell$ is regular supercoherent, then
\begin{equation}\label{eq:olbfk0gr}
K_0^{\gr}(L(E))=\ol{\BF}(E)\otimes K_0(\ell).
\end{equation} 
In particular, if $\ell$ is as in Example \ref{ex:l=c}, then we have 
\begin{equation}\label{map:l=c}
K_0^h(\ol{L(E)})\cong K_0^{\gr}(L(E)).
\end{equation}
\begin{rem}\label{rem:grade}
It follows from Example \ref{ex:lemaptoolle} that $L(\hat{E})$ is a $*$-subalgebra of $M_{\pm}\ol{L(E)}$. Write $\inc$ for inclusion map; we have a commutative diagram
\[
\xymatrix{\ol{\BF}(E)\otimes KH_0^h(\ell)\ar@{->>}[d]\ar@{ >-}[r]& KH_0^h(L(\hat{E}))\ar[d]^{KH_0^h(\iota_+)^{-1}\circ KH_0^h(\inc)}\\
           \ol{\BF}(E)\otimes_{\Z[\sigma]}KH_0^h(\ell)\ar@{ >-}[r]&KH_0^h(\ol{L(E)}).}
\]
Here the rows are the monomorphisms of Theorem \ref{thm:khseq} and the left column is the canonical surjection. In particular, $KH_0^h(\inc)$ is not injective in general.
\end{rem}

In the next proposition and elsewhere, we write $\cR_n$ for the graph consisting of a single vertex and $n$ loops, $\cR_{n^-}$ for its \emph{Cuntz splice} (\cite{alps}*{Definition 2.11}), $L_n=L(\cR_n)$ and $L_{n^-}=L(\cR_{n^-})$. 

\begin{prop}\label{prop:nograded}
Assume that $\ell$ is regular supercoherent and such that the canonical map $\Z\to K_0(\ell)$ is an isomorphism. Then there is no $\Z/2\Z$-homogeneous unital $\ell$-algebra homomorphism from $L_2$ to $L_{2^-}$ nor in the opposite direction. If furthermore, $\ell$ is a field with $\chara(\ell)\ne 2$ and $\Z[\sigma]\to K_0^h(\ell)$ is an isomorphism, then there is no unital $*$-algebra homomorphism
$\ol{L}_2\to \ol{L}_{-2}$ nor in the opposite direction.   
\end{prop}
\begin{proof}
By \eqref{eq:olbfk0gr} and the hypothesis on $\ell$, we have $K_0(L(E))=\BF(E)$ for any graph $E$.
Next we compute (e.g. using Remark \ref{rem:olbfreg}) that
\begin{equation}\label{cuental2}
\ol{\BF}(\cR_2)=\Z/3\Z,\,\, \ol{\BF}(\cR_{2^-})=\Z/7\Z.
\end{equation} 
By \eqref{olbf=bfe2} and Remark \ref{rem:ollegrade}, any unital, $\Z/2\Z$-homogeneous algebra homomorphism between $L_2$ and $L_{2^-}$ would induce a homomorphism between $\Z/3\Z$ and
$\Z/7\Z$ mapping $\bar{1}\mapsto\bar{1}$, but there is no such homomorphism. This proves the first assertion. If $\ell$ is as in the the second assertion, then \eqref{map:l=c} and \eqref{cuental2} together imply
\begin{equation}\label{cuental3}
K_0^h(\ol{L_2})=\Z/3\Z,\,\, K_0^h(\ol{L_{2^-}})=\Z/7\Z.
\end{equation}
The proof is now immediate.
\end{proof}

\section{Structure theorems for Leavitt path algebras in \topdf{$kk^h$}{kkh}}\label{sec:struct}

Let $n_0,n_1\ge 1$, $M\in\Z^{n_0\times n_1}$, and 
\begin{equation}\label{seq:elMR}
\xymatrix{j^h(\ell)^{n_1}\ar[r]^{M}&j^h(\ell)^{n_0}\ar[r]& j^h(R)}
\end{equation}
a distinguished triangle in $kk^h$. Applying $kk^h(\ell, -)$ we obtain a monomorphism $\coker(M)\otimes KH_0^h(\ell)\to KH_0^h(R)$; composing it with $-\otimes [1]:\coker(M)\to \coker(M)\otimes KH_0^h(\ell)$ we obtain a canonical map
\begin{equation}\label{map:cano}
\can: \coker(M)\to KH_0^h(R).
\end{equation}
Hence for every $S\in \ahas$, there is an evaluation map
\begin{equation}\label{map:elev}
\ev:kk^h(R,S)\to\hom(\coker(M), KH^h_0(S)),\,\, \ev(\xi)=KH^h_0(\xi)\circ\can.
\end{equation}
\begin{lem}\label{lem:resol} Let $E$ be a graph with $|E^0|<\infty$, and let $n_0,n_1\ge 1$, $M\in\Z^{n_0\times n_1}$, and $R$ be as in \eqref{seq:elMR}. 
Assume that $\rk(\ker M)=\rk(\ker(I-A_E^t))$. Let $\xi_0:\BF(E)\iso \coker(M)$ be a group isomorphism. Then there exists an isomorphism $\xi:j^h(L(E))\iso j^h(R)$ such that $\ev(\xi)=\can\circ \xi_0$. 
\end{lem}
\begin{proof}
Put $m_0=|E^0|$, $m_1=|\reg(E)|$. Because the free abelian groups $\ker(M)$ and $\ker(I-A_E^t)$ have the same rank, there is an isomorphism $\xi_1:\ker(I-A_E^t)\iso \ker(M)$. Because $\im(I-A_E^t)$ and $\im(M)$ are free, the surjections $\Z^{m_1}\onto \im(I-A_E^t)$ and $\Z^{n_1}\onto \im(M)$ admit sections $s$ and $t$. Let $f_0:\Z^{m_0}\to\Z^{n_0}$ be any lift of $\xi_0$; put 
\begin{equation}\label{eq:elf1}
f_1(x)=\xi_1(x-s(I-A_E^t)x)+t(f_0((I-A_E^t)x)).
\end{equation}
Then
\[
\xymatrix{\Z^{m_1}\ar[d]^{f_1}\ar[r]^{I-A_E^t}&\Z^{m_0}\ar[d]^{f_0}\\
          \Z^{n_1}\ar[r]_M& \Z^{n_0}}
\]
is a quasi-isomorphism $f$ with $H_i(f)=\xi_i$ ($i=0,1$). Using the canonical map $\Z\to KH_0^h(\ell)$ we can associate to $f$ a  commutative solid arrow diagram
\[
\xymatrix{
j^h(\ell)^{m_1}\ar[d]_{f_1}\ar[r]^{I-A_E^t}&j^h(\ell)^{m_0}\ar[r]^p\ar[d]^{f_0}& j^h(L(E))\ar@{.>}[d]^{\xi}\ar[r]^\partial& j^h(\ell)^{m_1}[-1]\ar[d]^{f_1[-1]}\\ 
j^h(\ell)^{n_1}\ar[r]^{M}&j^h(\ell)^{n_0}\ar[r]& j^h(R)\ar[r]& j^h(\ell)^{n_1}[-1]
}
\]
Because $kk^h$ is triangulated there exists a filler $\xi$ as above. By construction, the following diagram commutes
\[
\xymatrix{
\BF(E)\otimes KH_0^h(\ell)\ar[r]^(.55)\cup\ar[d]^{\xi_0\otimes 1}& KH_0^h(L(E))\ar[d]^{KH_0^h(\xi)}\\
\coker(M)\otimes KH_0^h(\ell)\ar[r]^(.57)\cup & KH_0^h(R)
}
\]
It follows that $\ev(\xi)=\can\xi_0$. It remains to show that $\xi$ is an isomorphism. Proceeding as above, we obtain a quasi-isomorphism 
$g:\Z^{n_*}\to \Z^{m_*}$ with $H_*(g)=\xi^{-1}_*$, such that for all $y\in \Z^{n_1}$,
\begin{equation}\label{eq:elg1}
g_1(y)=\xi_1^{-1}(y-tMy)+s(g_0(My)).
\end{equation}
As before, the chain map $g$ induces a map of triangles 
\[
\xymatrix{
j^h(\ell)^{n_1}\ar[d]_{g_1}\ar[r]^{M}&j^h(\ell)^{n_0}\ar[r]\ar[d]^{g_0}& j^h(R)\ar@{.>}[d]^{\eta}\ar[r]^{\partial}& j^h(\ell)^{n_1}[-1]\ar[d]^{g_1[-1]}\\
j^h(\ell)^{m_1}\ar[r]^{I-A_E^t}&j^h(\ell)^{m_0}\ar[r]^p& j^h(L(E))\ar[r]^(.45){\partial}& j^h(\ell)^{m_1}[-1].
}
\]
It follows from \eqref{eq:elf1} and \eqref{eq:elg1} that $g_1f_1$ restricts to the identity on $\ker(I-A_E^t)$ and that there is a homomorphism $h:\Z^{m_0}\to \Z^{m_1}$ with $h\circ(I-A_E^t)=\id-g_1f_1$ and $(I-A_E^t)\circ h=\id-g_0f_0$. Hence
\begin{gather*}
\partial(1-\eta\xi)=(1-g_1f_1[-1])\partial=(h\circ(I-A_E^t))[-1]\circ\partial=0\\
(1-\eta\xi)p=p(1-f_0g_0)=p(1-A_E^t)h=0.
\end{gather*}
Therefore there exist $\zeta_0\in kk^h(L(E),\ell^{m_0})$ and $\zeta_1\in kk_1^h(\ell^{m_1},L(E))$ such that $1-\eta\xi=p\zeta_0= \zeta_1\partial$. In particular, $(1-\eta\xi)^2=\zeta_1\partial p\zeta_0=0$,
and therefore $\eta\xi$ is an isomorphism. Similarly, $\xi\eta$ is an isomorphism, and so $\xi$ is an isomorphism, concluding the proof. 
\end{proof}

\begin{thm}\label{thm:isokk}
Let $E$ and $F$ be graphs with finitely many vertices and such that $|\sing(E)|=|\sing(F)|$. Let $\xi_0:\BF(E)\iso \BF(F)$ be an isomorphism. Assume that $\ell$ satisfies the $\lambda$-assumption \ref{stan:lambda}. Then there exists an isomorphism $\xi:j^h(L(E))\iso j^h(L(F))$ such that $\ev(\xi)=\can\circ \xi_0$. 
\end{thm}
\begin{proof}
Apply Lemma \ref{lem:resol} with $M=I-A_F^t$ and $R=L(F)$.
\end{proof}

\begin{rem}\label{rem:resol}
One may ask whether a $KH_0^h(\ell)$-module isomorphism $\xi_0:\BF(E)\otimes KH_0^h(\ell)\iso \BF(F)\otimes KH_0^h(\ell)$ lifts to a $kk^h$-isomorphism
$\xi:j^h(L(E))\iso j^h(L(F))$  if $|\sing(E)|=|\sing(F)|$. The argument of Lemma \ref{lem:resol} proves that this is indeed the case if we additionally assume that $\tor_1^{\Z}(\BF(E),KH_0^h(\ell))=0$. 
 \end{rem}

Let $E$ be a graph with $|E^0|<\infty$. Then $\BF(E)$ is finitely generated. Thus there are $r,n\ge 0$ and $2\le d_1,\dots, d_n$ with $d_i\backslash d_{i+1}$ for all $i$ such that
\begin{equation}\label{invfact}
\BF(E)=\Z^r\oplus\bigoplus_{i=1}^n\Z/d_i\Z.    
\end{equation}

\begin{thm}\label{thm:struct}
Let $E$ be a graph such that $E^0$ is finite. Assume that $\ell$ satisfies the $\lambda$-assumption \ref{stan:lambda}. Let $r$, $n$ and $d_1,\dots,d_n$ be as in \eqref{invfact} and let $s=|\sing(E)|$ be the number of singular vertices.
Then
\[
j^h(L(E))\cong j^h(L_0^s\oplus L_1^{r-s} \oplus \bigoplus_{i=1}^nL_{d_i+1})
\]
\end{thm}
\begin{proof}
Apply Lemma \ref{lem:resol} with $M$ the Smith normal form of $I-A_E^t$  and $R=L_0^s\oplus L_1^{r-s} \oplus \bigoplus_{i=1}^nL_{d_i+1}$.
\end{proof}

As previously recalled from \cite{cm1} and \cite{cv}, for any graph $E$ we have $E$-stable variants of the universal homology theories $j:\aha\to kk$ and $j^h:\ahas\to kk^h$. In the following corollary, as well as in any other statement involving two graphs $E$ and $F$ with possibly infinitely many edges, $j$ and $j^h$ are understood to be the $E\sqcup F$-stable ones.

\begin{thm}\label{thm:struiso}
Let $E$ and $F$ be graphs with finitely many vertices. Assume that $\ell$ satisfies the $\lambda$-assumption \ref{stan:lambda}, that $KH_{-1}(\ell)=0$ and that the canonical map $\Z\to KH_0(\ell)$ is an isomorphism. 
Then the following are equivalent.
\item[i)] $j(LE)\cong j(LF)$ in $kk$.
\item[ii)]$j^h(LE)\cong j^h(LF)$ in $kk^h$.
\end{thm}
\begin{proof}
If ii) holds, then the forgetful functor $kk^h\to kk$ sends the isomorphism $j^h(LE)\cong j^h(LF)$ to an isomorphism $j(LE)\cong j(LF)$. If i) holds, then by \cite{cm1}*{Corollary 6.11}, $\BF(E)\cong \BF(F)$ and $|\sing(E)|=|\sing(F)|$, so $j^h(LE)\cong j^h(LF)$ by Theorem \ref{thm:isokk}.
\end{proof}

\begin{rem}\label{rem:nohold}
The analogue of Theorem \ref{thm:struiso} with $\ol{LE}$ and $\ol{LF}$ substituted for $LE$ and $LF$ does not hold. Indeed, for any $\ell$, $j(L_2)=j(L_{2^-})=0$ but for $\ell$ as in Example \ref{ex:l=c}, $j^h(\ol{L_2})\not\cong j^h(\ol{L_{2^-}})$, by \eqref{cuental3}. 
\end{rem}

The proof of Lemma \ref{lem:resol} does not work for the analogue of the lemma with a matrix $M$ with coefficients in $\Z[\sigma]$ and $\ol{\BF}(E)$ substituted for $\BF(E)$. This is because a submodule of a free $\Z[\sigma]$-module need not be free or even projective. However the problem disappears if $\ker(M)=\ker(I-\sigma A_E^t)=0$, and we have the following.

\begin{lem}\label{lem:resol2}
Let $E$ be a graph with $|E^0|<\infty$, and let $n_0,n_1\ge 1$, $M\in\Z[\sigma]^{n_0\times n_1}$, and $R$ as in \eqref{seq:elMR}. 
Assume that $\ker M=\ker(I-\sigma A_E^t)=0$. Let $\xi_0:\ol{\BF}(E)\iso \coker(M)$ be a group isomorphism. Then there exists an isomorphism 

\goodbreak

\noindent $\xi:j^h(\ol{L(E)})\iso j^h(R)$ such that $\ev(\xi)=\can\circ \xi_0$.
\end{lem}
\begin{proof} The argument of the proof of Lemma \ref{lem:resol} shows this.
\end{proof}

\begin{thm}\label{thm:isokk2}
Let $E$ and $F$ be graphs with finitely many vertices and such that $\ker(I-\sigma A_E^t)=\ker(I-\sigma A_F^t)=0$. Let $\xi_0:\ol{\BF}(E)\iso \ol{\BF}(F)$ be an isomorphism. Assume that $\ell$ satisfies the $\lambda$-assumption \ref{stan:lambda}. Then there exists an isomorphism $\xi:j^h(\ol{L(E)})\iso j^h(\ol{L(F)})$ such that $\ev(\xi)=\can\xi_0$. 
\end{thm}
\begin{proof}
Apply Lemma \ref{lem:resol2} with $M=I-\sigma A_F^t$ and $R=\ol{L(F)}$.
\end{proof}

\begin{prop}\label{prop:bfolbf}
Assume that $2$ is invertible and $1$-positive in $\ell$. Let $E$ be a finite graph such that the $\Z[\sigma]$-modules $\ol{\BF}(E)$ and $\BF(E)$ are finite and isomorphic. Let $\ob{\Delta}$ be as in Proposition \ref{prop:deltatau}. Then $2\id_{L(E)}$ and $2\id_{\ol{L(E)}}$ are invertible elements of 
$kk^h(L(E),L(E))$ and $kk^h(\ol{L(E)},\ol{L(E)})$, and we have 
\[
({\ob{\Delta}}_{L(E)})/2)\circ ({\ob{\Delta}}_{\ol{L(E)}}/2)=j^h(\id_{\ol{L(E)}}).
\]
In particular, $j^h(\ol{L(E)})$ is a retract of $j^h(L(E))$. Moreover the following sequence is exact
\[
\xymatrix{0\ar[r]& j^h(\ol{L(E)})\ar[r]^{{\ob{\Delta}}_{\ol{L(E)}}}&j^h(L(E))\ar[r]^{1-\sigma}&j^h(L(E)).} 
\]
\end{prop}
\begin{proof}  If $\ol{\BF}(E)$ is finite, then $E$ is regular, so by Remark \ref{rem:olbfreg}, $\ol{\BF}(E)$ is isomorphic as a $\Z[\sigma]$-module to $M=\coker(I-(A_E^t)^2)$ where $\sigma$ acts as multiplication by $A_E^t$. Then $M\cong \BF(E)$ as $\Z[\sigma]$-modules $\iff$ $A_E^t$ acts trivially on $M$, in which case $I+A_E^t$ descends to multiplication by $2$, and we have
\begin{equation}\label{eq:displayed}
\im(I-A_E^t)=\im(I-(A_E^t)^2)=(I+A_E^t)(\im(I-A_E^t)).
\end{equation}
Hence $I+A_E^t$ restricts to an injection on $\im(I-A_E^t)$, which has rank $|E^0|$ since $\BF(E)$ is finite. Thus $\det(I+A_E^t)\ne 0$, which combined with \eqref{eq:displayed} implies that $(I+A_E^t)x$ goes to zero in $M$ if and only if $x$ does. It follows that multiplication by $2$ on $M$ is injective and therefore bijective since $M$ is finite. Hence $I+A_E^t$ is invertible and $n=|M|$ is odd. Write $n=2q-1$; then by the argument of Lemma \ref{lem:resol}, $n^2j^h(\id_{L(E)})=n^2j^h(\id_{\ol{L(E)}})=0$, so $2j^h(\id_{L(E)})$ and $2j^h(\id_{\ol{L(E)}})$ are isomorphisms. The same argument but with $\Z[\sigma]$ substituted for $\Z$, shows that $$2(1-\sigma)\id_{\ol{L(E)}}=(1-\sigma)^2\id_{\ol{L(E)}}=0,$$ which by what we have already seen implies that $(1-\sigma)\id_{\ol{L(E)}}=0$. The proposition now follows from Proposition \ref{prop:deltatau} and Corollary \ref{coro:fundtau}.
\end{proof}

\begin{rem}\label{rem:bfolbf1} Let $E$ and $\ell$ be as in Proposition \ref{prop:bfolbf}. Observe that $\sigma$ acts as $1\otimes\sigma$ on $\BF(E)\otimes KH_0^h(\ell)\subset KH_0^h(L(E))$; this action is nontrivial in general, so in particular $0\ne 1-\sigma\in kk^h(L(E),L(E))$. For example $\ell=\C$ satifies the hypothesis of the proposition, and $\sigma$ acts on $\BF(E)\oplus \sigma\BF(E)=\BF(E)\otimes KH_0^h(\C)$ by interchanging the summands.
\end{rem}

\begin{rem}\label{rem:bfolbf2} The proof of Proposition \ref{prop:bfolbf} shows that if a regular finite graph $E$ with incidence matrix $A$
satisfies the hypothesis of the proposition and $\BF(E)\ne 0$, then 
\begin{equation}\label{eq:bfolbf}
\det(I+A)\in\{\pm 1\} \text{ and }\det(I-A)\in\Z\setminus\{0,\pm 1\}.
\end{equation} The converse is also true; if \eqref{eq:bfolbf} holds, then $I+A^t:\BF(E)\to \BF(E^2)\cong \ol{\BF}(E)$ is a $\Z[\sigma]$-module isomorphism. A concrete example is the graph with incidence matrix
\[
A_E=\begin{bmatrix} 1 &3\\ 1& 1\end{bmatrix}.
\] 
\end{rem}
\section{\topdf{$K_0$}{K0} invariants for \topdf{$*$}{*}-algebras}\label{sec:k0}

Let $A$ be a $*$-algebra and $x\in A$. A  \emph{projection} in $A$ is a self-adjoint idempotent element; we write $\proj(A)$ for the set of all projections in $A$. If $p,q\in\proj(A)$, we write $p\ge q$ whenever the identities $q=pq=qp$ hold. Two projections $p,q\in A$ are (Murray-von Neumann) \emph{equivalent} --and we write $p\sim q$-- if there is an element $x\in A$ such that $x^*x=p$, $xx^*=q$. Such an $x$ can be chosen to be in $pAq$, in which case we call it an ($MvN$-) \emph{equivalence} and write  $x:p\iso q$.  All the basic properties about equivalence of idempotents proved in \cite{black}*{Section 4.2} hold verbatim for projections, with the same proofs. A \emph{partial isometry} is an element $x\in A$ such that $xx^*x=x$. Let $R$ be a unital $*$-algebra. An \emph{isometry} in $R$ is an element $u$ such that $u^*u=1$; a \emph{unitary} is an invertible isometry. For $1\le n\le \infty$, put $\proj_n(R)=\proj(M_nR)$ and consider the set of equivalence classes $\cV_n(R)^*=\proj_n(R)/\sim$. If $n<m$, then the map $\cV_n(R)^*\to \cV_m(R)^*$ is injective, and $\cV_\infty^*(R)=\bigcup_n \cV_n^*(R)$. The set $\cV_\infty(R)^*$
forms an abelian monoid under \emph{orthogonal sum}; if $x,y\in\proj_\infty(R)$ then there are orthogonal projections $p,q\in\proj_\infty(R)$ with $x=[p]$ and $y=[q]$ and $x+y:=[p+q]$. This is well-defined by the projection analogue of \cite{black}*{Proposition 4.2.4}. A projection $p\in \proj(R)$ is \emph{strictly full} if there exist $n\ge 1$ and $(x_1,\dots,x_n)\in R^n$ such that 
$\sum_{i=1}^nx_i^*px_i=1$. 

\begin{lem}\label{lem:vmorite}
Let $R$ be a unital $*$-algebra and $p\in\proj(R)$ a strictly full projection. Then the inclusion $pRp\subset R$ induces an isomorphism $\cV_\infty(pRp)^*\iso \cV_\infty(R)^*$.
\end{lem}
\begin{proof}
Consider the category $\mathfrak{P}(R)$ whose set of objects is $\proj_\infty(R)$ and where a homomorphism $q_1\to q_2$ is an element $x\in M_\infty R$ such that $x^*x=q_1$ and $xx^*\le q_2$. We have a functor to abelian monoids $F:\fP(R)\to\abmon$ which sends 
$q\mapsto \cV_1(qM_{\infty}Rq)^*$. Note that any two homomorphisms $q_1\to q_2$ in $\fP(R)$ induce the same monoid homomorphism upon applying $F$. Hence if $\fC\subset \fP(R)$ is a subcategory such that for every $q\in\fP(R)$ there is $q\to q'$ with $q'\in\fC$, then $\colim_\fC F=\colim_{\fP(R)}F$. Apply this to the subcategory $\fI(R)\subset\fP(R)$ whose objects are the identity matrices $1_n$ and where a homomorphism $1_n\to 1_m$ is a matrix $x\in R^{m\times n}$ whose columns are $n$ consecutive vectors of the canonical basis of $R^m$. Then $\colim_{\fP(R)}F=\colim_{\fI(R)}\cV_n(R)^*=\cV_{\infty}(R)^*$. Similarly, if $p\in \proj(R)$ is strictly full, then $\cV_\infty(pRp)^*=\colim_{\fP(pRp)}F\iso \colim_{\fP(R)}F=\cV_\infty(R)^*$.
\end{proof}

Consider the group completion 
\[
K_0(R)^*:=(\cV_\infty(R)^*)^+
\]

Now assume that the center of $R$ satisfies the $\lambda$-assumption \ref{stan:lambda}. The \emph{Hermitian Witt-Grothendieck group} $K_0^h(R)$ is defined as the group completion of the monoid 
 \begin{equation}\label{vh=v}
 \cV^h_\infty(R)=\cV_\infty (M_{\pm}R)^*
 \end{equation}
 thus we have
\begin{equation}\label{k0hproj}
K_0^h(R)=K_0(M_\pm(R))^*.
\end{equation}
The inclusion $\iota_+:R\to M_{\pm}R$ induces a canonical homomorphism $\cV_\infty(R)^*\to \cV_\infty^h(R)$. We may also regard $\cV^h_\infty(R)$ as the monoid of unitary isomorphism 
 classes of all pairs $[(P,\phi)]$ consisting of a finitely generated projective right module equipped with a nondegenerate Hermitian form $\phi$. In the same fashion, the monoid $\cV_\infty(R)^*$ consists of the unitary classes of those $P=(P,\phi)$ for which there is an $n\ge 1$ such that $P$ embeds as an orthogonal direct summand of the free module of rank $n$ equipped with the standard Hermitian form
 \[
 \langle x,y\rangle=\sum_{i=1}^nx_i^*y_i.
 \]
Observe that $-1$ is positive in $R$ if and only if there is some $n\ge 1$ such that the equation
 \begin{equation}\label{-sol}
 \langle x,x\rangle=-1 
 \end{equation}
 has a solution $x\in R^n$. In this case $p:R^n\to R^n$, $p(y)=-x\langle x,y\rangle $ is an orthogonal projection onto $xR$.  In particular the hyperbolic module of rank $2$, $H(R)$, embeds 
 as an orthogonal summand in $R^{n+1}$, and $M_\pm R\cong (1\oplus p)M_{n+1}R(1\oplus p)$. One checks that 
\[
\sum_{1\le i,j\le n}\epsilon_{i,j}p\epsilon_{i,j}^*=\sum_{i=1}^n\epsilon_{i,i}\sum_{j=1}^n -x_jx_j^*=I_n
\]
Hence $p$ is a strictly full projection of $M_nR$, and therefore $1\oplus p$ is a strictly full projection of $M_{n+1}R$. Thus by Lemma \ref{lem:vmorite} the canonical map $\cV_\infty(R)^*\to\cV^h_\infty(R)$ is an isomorphism, as is the induced map $K_0(R)^*\to K_0^h(R)$.

\begin{rem}\label{rem:can'}
Let $R$ be a unital $*$-algebra. The involution defines an action of $\Z/2\Z$ on $K_0(R)$. With no assumptions on $\ell$, for any unital $*$-algebra $R$, we have a canonical forgetful map $\forg: K_0(R)^*\to K_0(R)^{\Z/2\Z}$ which maps the class of a projection modulo M-vN equivalence of projections to its class modulo Mv-N equivalence of idempotents. If $E$ is a graph with finitely many vertices and $v\in\reg(E)$, then the identity $[v]=\sum_{s(e)=v}[r(v)]$ holds in $K_0(LE)^*$. Hence we have a canonical group homomorphism
\[
\can':\BF(E)\to K_0(LE)^*.
\]
If $\ell$ is regular supercoherent and $\Z\to K_0(\ell)$ is an isomorphism, then $\forg\circ\can'$ is the well-known isomorphism $\BF(E)\iso K_0(L(E))$. In particular, $\can'$ is a split monomorphism in this case. 
\end{rem}  

\begin{rem}\label{rem:olcan}
The identities \eqref{eq:ee*re} show that if $E$ is a graph with finitely many vertices, then $[ee^*]=\sigma[r(e)]\in K_0^h(\ol{L(E)})$ for every $e\in E^1$. Hence there is a canonical $\Z[\sigma]$-module homomorphism
\[
\ol{\can}:\ol{\BF(E)}\to K_0^h(\ol{L(E)}). 
\]
\end{rem}

\section{Strictly properly infinite \topdf{$*$}{*}-algebras, purely infinite simple graphs, and their \topdf{$K_0$}{K0} invariants}\label{sec:k0spi}
 
Let $C_n$ be the Cohn path algebra of $\cR_n$ $(0\le n\le \infty)$; observe that $C_\infty=L_\infty$, the Leavit path algebra of $\cR_\infty$.  A unital algebra $R$ is \emph{properly infinite} there is a unital algebra homomorphism $C_2\to R$, or equivalently, an algebra homomorphism $L_\infty\to R$. (Properly infinite algebras were called sum algebras in 
\cite{fawag} and renamed $C_2$-algebras in \cite{cm1}*{Section 2}; we further rename them to conform to standard $C^*$-algebra terminology.) A $*$-algebra $R$ is \emph{strictly properly infinite} if there is a unital $*$-homomorphism $C_2\to R$, or equivalently, a unital $*$-homomorphism 
$L_\infty\to R$. 
Observe that if $R$ is strictly properly infinite and $\phi:R\to S$ is a unital $*$-homomorphism, then $S$ is strictly properly infinite too. Thus if $R$ and $S$ are unital $*$-algebras and $R$ is strictly properly infinite, so is $R\otimes S$. In particular, $L_\infty\otimes R$ is strictly properly infinite for every unital $*$-algebra $R$, and if $R$ is strictly properly infinite, then $M_nR$ and $M_{\pm}M_nR$ are strictly properly infinite for all $1\le n<\infty$. A projection $p\in R$ is \emph{strictly properly infinite} if $pRp$ is strictly properly infinite. Equivalently $p$ is strictly properly infinite if there are nonzero orthogonal projections $p_1,p_2\in pRp$, such that $p\sim p_1\sim p_2$.

\begin{rem}\label{rem:sumspipro}
The sum of orthogonal strictly properly infinite projections in a unital $*$-algebra is again strictly properly infinite. In particular, if $R$ contains orthogonal strictly properly infinite projections $p_1,\dots,p_n$ such that $\sum_{i=1}^np_i=1$, then $R$ is strictly properly infinite.
\end{rem}

We say that a graph $E$ is \emph{simple} if it is cofinal \cite{lpabook}*{Definitions 2.9.4} and every cycle in $E$ has an exit \cite{lpabook}*{Definitions 2.2.2}. A simple graph having at least one cycle is called \emph{purely infinite simple}. If $\ell$ is a field,
then $L(E)$ is (purely infinite) simple if and only if $E$ is \cite{lpabook}*{Theorems 2.9.1 and 3.1.10 and Lemma 2.9.6}.

\begin{ex}\label{ex:phantompis}
The graph $\booh$ of Example \ref{ex:phantom} is purely infinite simple. 
\end{ex}
\begin{lem}\label{lem:remove}
Let $E$ be a finite graph and $v\in\sour(E)\setminus\sink(E)$. Let $E_{\slash v}$ be the source elimination graph of $E$ (\cite{lpabook}*{Definition 6.3.26}). Then 
\begin{itemize}
\item[i)] The element $p=1-v\in L(E)$ is a strictly full projection. 
\item[ii)] The inclusion $E_{\slash v}\subset E$ induces $*$-isomorphisms $L(E_{\slash v})\to pL(E)p$ and $\ol{L(E_{\slash v})}\to p\ol{L(E)}p$. 
\item[iii)] $E$ is (purely infinite) simple if and only if $E_{\slash v}$ is.
\end{itemize}
\end{lem}
\begin{proof}
The argument of the proof of \cite{alps}*{Proposition 1.4} shows that the canonical homomorphism $L(E_{\slash v})\to L(E)$ induced by the inclusion, which is a $*$-homomorphism with respect to both the standard and the twisted involution, corestricts to an isomorphism onto $pL(E)p$; this proves ii). Because $v$ is a source but not a sink, we have $v=\sum_{s(e)=v}epe^*$, whence $1=p+\sum_{s(e)=v}epe^*$ and therefore $p$ is strictly full, proving i). If $\ell$ is a field, the (pure infinite) simplicity of $E$ is equivalent to that of $L(E)$, which by ii), is Morita equivalent to $L(E_{\slash v})$. The proof of iii) and of the lemma is concluded by using that simplicity and purely infinite simplicity are preserved by Morita equivalence \cite{k0pis}*{Corollary 1.7}. 
\end{proof}

\begin{lem}\label{lem:spipis}
Let $E$ be a finite, purely infinite simple graph. Then every vertex of $E$ is a strictly properly infinite projection of $L(E)$ and $\ol{L(E)}$. 
\end{lem}
\begin{proof}
If $v\in E^0$ is in a cycle, then there is a cycle $\alpha^1_v$ based at $v$. Let $\alpha_v^2$ be a closed path starting at $v$, following $\alpha_v^1$ up to an exit, taking the exit, then coming back to the cycle --as is possible due to cofinality of $E$-- and following $\alpha^1_v$ again until $v$. Upon replacing $\alpha_v^1$ and $\alpha_v^2$ by their squares, if necessary, we may assume that their lengths are even, so that $\ol{\alpha_v^i}=(\alpha_v^i)^*$ for $i=1,2$. Then $(\alpha^i_v)^*\alpha^j_v=\delta_{i,j}v$, and $v$ is a strictly properly infinite projection of both $L(E)$ and $\ol{L(E)}$. If $\sour(E)=\emptyset$, every vertex is in a cycle, and the lemma follows. Otherwise, we can proceed by source elimination until we arrive to a purely infinite simple graph without sources. At each step, the source eliminated is equal to a sum of projections of the form $ee^*$ with $e\in E^1$ and $r(e)\in E^0_{\slash v}$. If $x_1,x_2\in r(e)L(E)r(e)$ are orthogonal isometries for either of the involutions, then $ex_1e^*,ex_2e^*\in ee^*L(E)ee^*$ are again orthogonal isometries, for the same involution. Since the sum of orthogonal strictly properly infinite projections is again strictly properly infinite, we get that every vertex is a strictly properly infinite projection. 
\end{proof}

\begin{coro}\label{coro:spipis}
If $E$ is finite and purely infinite simple, then $L(E)$ and $\ol{L(E)}$ are strictly properly infinite.
\end{coro}
\begin{proof}
Immediate from Lemma \ref{lem:spipis} and Remark \ref{rem:sumspipro}.
\end{proof}

Let $R\in\ahas$ be unital and $p\in\proj(R)$. Following \cite{black}*{Section 6.11} we say that $p$ is \emph{very full} if there exists $q\in \proj(R)$ such that $p\ge q\sim 1$. Observe that any projection equivalent to a very full one is again very full. 
We write $\proj_f(R)\subset\proj(R)$ for the subset of very full projections.

\begin{ex}\label{ex:vfull}
Let $R,S\in\ahas$ be unital, $p\in\proj_f(R)$, $q\in\proj_f(S)$. Then $p\otimes q\in\proj_f(R\otimes S)$. If $\phi:R\to S$ is a $*$-homomorphism with $\phi(1)\in\proj_f(R)$, then $\phi(p)\in\proj_f(S)$.  
\end{ex}

\begin{lem}\label{lem:vveryfull}
Let $E$ be a finite, purely infinite simple graph and let $v\in E^0$. Then $v$ and $1\otimes v$ are very full projections of $L(E)$ and $M_\pm \ol{L(E)}$. 
\end{lem}
\begin{proof} Let $v\in E^0$. By Lemma \ref{lem:spipis} we may choose an element $x_w\in vL(E)v$ for each $w\in E^0\setminus\sour(E)$ such that $x_w^*x_{w'}=\delta_{w,w'}v$. Because $E$ is purely infinite simple, for every $w\in E^0\setminus\sour(E)$, there is a path $\alpha_w$ from $v$ to $w$. Then 
\begin{equation}\label{eq:vveryfull}
x=\sum_{w\in E^0\setminus\sour(E)}x_w\alpha_w
\end{equation}
satisfies $xx^*\le v$ and 
$p:=x^*x=\sum_{w\in E^0\setminus\sour(E)}w$. By Lemma \ref{lem:spipis} and Remark \ref{rem:sumspipro} $R=pL(E)p$ is strictly properly infinite. Hence we may choose a family of orthogonal isometries $R\supset\{y_e:s(e)\in\sour(E)\}\cup \{y_w:w\in E^0\setminus\sour(E)\}$; put $y=\sum_{s(e)\in\sour(E)}y_ee^*+\sum_{w\in E^0\setminus\sour(E)}y_ww$. Then $yy^*\le p$ and $y^*y=1$; hence $p$ is very full and therefore $v$ is a very full projection of $L(E)$. It follows that $1\otimes v\in M_\pm \ol{L(E)}$ is very full since it is the image of $v$ under the unital $*$-homomorphism $\Delta$ of \eqref{map:delta}.
\end{proof}
 
Let $R$ be a unital, strictly pure infinite $*$-algebra. Consider the set of equivalence classes
\[
\cV_1(R)\supset \cV_f(R)=\{[p]: p\in R \text{ very full}\}.
\]

\begin{prop}[cf. \cite{black}*{Theorem 6.11.7}]\label{prop:vfull}
Let $R$ be a unital, strictly properly infinite $*$-algebra.  Then the orthogonal sum makes $\cV_f(R)$ into a group, canonically isomorphic to $K_0(R)^*$.
\end{prop}
\begin{proof} The argument of the proof of \cite{black}*{Theorem 6.11.7} shows this. 
\end{proof}
\begin{coro}\label{coro:vfull}
Let $R$ be as in Proposition \ref{prop:vfull}. Further assume that the center of $R$ satisfies 
the $\lambda$-assumption \ref{stan:lambda}. Then $\cV_f(M_{\pm}R)$ is a group, canonically isomorphic to $K_0^h(R)$.
\end{coro}
\begin{proof}
Combine Proposition \ref{prop:vfull} and \eqref{k0hproj}.
\end{proof}

For the rest of this section we assume that $\ell$ satisfies the $\lambda$-assumption \ref{stan:lambda}. 

\begin{lem}\label{lem:i1=i2}
Let $\iota_i:\ell\to M_2$, $\iota_i(x)=\epsilon_{i,i}x$. We have $\iota_+\circ\iota_1\sim^*\iota_+\circ\iota_2$.
\end{lem}
\begin{proof} By \cite{ct}*{Section 3.4}, there is a matrix $g(t)\in\Gl_2(\ell)$ such that $g(0)=1$, $g(1)$ is unitary, and $\iota_2=\ad(g(1))\circ\iota_1$. Hence $\iota_+\circ\iota_1\sim^*\iota_+\iota_2$, by \cite{cv}*{Lemma 5.4}. 
\end{proof}

Let $A,B\in\ahas$; two $*$-homomorphisms $\phi,\psi:A\to B$ are \emph{$M_{\pm 2}$-$*$-homotopic} --and we write $\phi\sim_{M_{\pm 2}}^*\psi$-- if $\iota_+\circ\iota_1\circ\phi$ and $\iota_+\circ\iota_1\circ\psi$ are $*$-homotopic. We write $[A,B]^*_{M_{\pm 2}}$ for the set of $M_{\pm 2}$-$*$-homotopy classes of $*$-homomorphisms.

\begin{lem}\label{lem:abc}
Let $A,B\subset C$ $*$-subalgebras and $\inc_A$ and $\inc_B$ the inclusion maps. Let $x\in C$ such that $xAx^*\subset B$ and $ax^*xa'=aa'$ for all $a,a'\in A$. Then $\ad(x):A\to B$, $\ad(x)(a)=xax^*$ is a $*$-homomorphism and $\inc_B\ad(x)\sim^*_{M_{\pm 2}}\inc_A$. If moreover $A=B$ and $Ax\subset A$, then $\ad(x)\sim^*_{M_{\pm 2}}\id_A$. 
\end{lem}
\begin{proof} Combine the argument of \cite{cm1}*{Lemma 2.3} with Lemma \ref{lem:i1=i2} above.
\end{proof}

\begin{lem}\label{lem:morite}
Let $R$ be a unital $*$-algebra and $p\in R$ a strictly full projection and let $\iota_p:pRp\to R$ be the inclusion.
Then $j^h(\iota_p)$ is an isomorphism in $kk^h$. 
\end{lem}

\begin{proof}
Because $p$ is strictly full, there are $n\ge 1$ and $x\in pR^n$ such that $\sum_{i=1}^nx^*_ix_i=1$. Set $y=\sum_{j=1}^n\epsilon_{1,j}x_j\in M_npR\subset M_nR$ and $z=yp\in M_n(pRp)$. Then  $\phi:R\to M_n(pRp)$, $\phi(a)=y\iota_1(a)y^*=\sum_{i,j}\epsilon_{i,j}x_iax_j^*$ is a $*$-homomorphism. Let $\iota_1=\iota_1^{pRp}:pRp\to M_npRp$, $\iota_1(a)=\epsilon_{1,1}a$. Then  $\phi\circ\inc_p=\ad(z)\circ\iota_1$, which by Lemma \ref{lem:abc} is $M_{\pm 2}-*$-homotopic to $\iota_1$, so $j^h(\phi\circ\inc_p)$ is an isomorphism in $kk^h$. Similarly, writing $\iota_1$ now for $\iota_1^R$,  
$M_n(\inc_p)\circ\phi=\ad(x)\circ\iota_1$ is $M_{\pm 2}-*$-homotopic to $\iota_1$ and thus $j^h(M_n(\inc_p)\circ\phi)$ is an isomorphism too. Hence $j^h(\phi)$ is an isomorphism, which by what we have already proved implies that $j^h(\inc_p)$ is an isomorphism, concluding the proof. 
\end{proof}

Let $R$ be a strictly properly infinite $*$-algebra. By definition, there are $s_1,s_2\in R$ such that $s_i^*s_j=\delta_{i,j}$. Let
\[
\boxplus:R\oplus R\to R,\quad a\boxplus b=s_1as_1^*+s_2bs_2^*.
\]
Let $\phi,\psi:A\to R$ be $*$-homomorphisms. Put
\begin{equation}\label{sumamorfi}
\phi\boxplus \psi:A\to R, \quad (\phi\boxplus\psi)(a)=\phi(a)\boxplus\psi(a).
\end{equation}

\begin{lem}\label{lem:M2htpy}
Let $A$ and $R\in\ahas$, with $R$ strictly properly infinite. Then \eqref{sumamorfi} makes $[A,R]^*_{M_{\pm 2}}$ into an abelian monoid.
\end{lem} 
\begin{proof}
Combine Lemma \ref{lem:abc} with the argument of \cite{cm1}*{Lemma 2.5}. 
\end{proof}

Let $A$ and $R$ be as in Lemma \ref{lem:M2htpy} and let $\phi_0,\phi_1:A\to R$ be $*$-homomorphisms; we say that $\phi_0$ and $\phi_1$ are \emph{stably $M_{\pm 2}$-homotopic}, and write $\phi_0\sim^s_{M_{\pm 2}} \phi_1$, if there exists a $*$-homomorphism $\psi:A\to R$ such that 
\begin{equation}\label{defi:stably}
\phi_0\boxplus \psi\sim_{M_{\pm 2}}^* \phi_1\boxplus\psi.
\end{equation}
In other words, $\phi_0\sim^s_{M_{\pm 2}}\phi_1$ means that the $M_{\pm 2}$-homotopy classes of $\phi_0$ and $\phi_1$ go to the same element in the group completion 
\begin{equation}\label{map:complete}
[A,R]^*_{M_{\pm 2}}\to ([A,R]^*_{M_{\pm 2}})^+.
\end{equation}

\begin{rem}\label{rem:i1=i2}
Assume $\ell=\inv(\ell_0)$ for some commutative ring $\ell_0$. Then every $R\in\ahas$ is of the form $\inv(R_0)$ for some $R_0\in\alg_{\ell_0}$, projections in $R$ correspond to idempotents in $R_0$, MvN equivalence and very and strict fullness of projections to very fullness and fullness of idempotents, and $R$ is strictly properly infinite if and only if $R_0$ is properly infinite. We claim that, furthermore, $\iota_+$ can be dropped in Lemma \ref{lem:i1=i2}. For $a\in M_2\ell_0$, let $a^*$ be the transpose matrix. Then $M_2=M_2\ell_0\oplus M_2\ell_0$ equipped with the involution $(a,b)^*=(b^*,a^*)$. Let $g(t)\in \Gl_2(\ell_0)$ be as in the proof of Lemma \ref{lem:i1=i2} and let $h(t)=(g(t),(g(t)^*)^{-1})$. Then $h(t)\in\cU_2(\ell)$, $h(0)=(1,1)$ and $\ad(h(1))\circ \iota_1=\iota_2$, and so we have $\iota_1\sim^*\iota_2$, proving the claim. It follows that in \ref{lem:abc} and \ref{lem:M2htpy} we may replace $M_{\pm 2}$-$*$-homotopy by $M_2$-$*$-homotopy, defined in the obvious way, which under the category equivalence $\inv:\alg_{\ell_0}\to \ahas$ corresponds to $M_2$-homotopy as defined in \cite{cm1}*{Section 2}. Thus the lemmas above specialize to \cite{cm1}*{Lemmas 2.1, 2.3 and 2.5}. 
\end{rem}

\section{Lifting \topdf{$K_0$}{K0}-maps to \topdf{$*$}{*}-algebra maps}\label{sec:k0lift}

Let $E$ be a graph, $R$ a strictly properly infinite unital $*$-algebra and $\phi:L(E)\to R$ an algebra homomorphism. We say that $\phi$ is \emph{very full} if 
\begin{equation}\label{pptyP}
\{\phi(ee^*):e\in E^1\}\cup\{\phi(v):v\in\sing(E)\}\subset\proj_f(R). 
\end{equation}

\begin{ex}\label{ex:pptyp} Let $E$ be a finite, purely infinite graph, let $\phi:L(E)\to R$ be $*$-homomorphism and put $p=\phi(1)$. By Lemma \ref{lem:vveryfull}, for every element $q$ in the left hand side of the inclusion \eqref{pptyP} there is a projection $q'\le q$ such that $q'\sim p$. Hence  $\phi$ is very full if and only if $p\in\proj_f(R)$. 
\end{ex}

\begin{rem}\label{rem:Psemigroup}
If $\phi,\psi:LE\to R$ are very full $*$-homomorphisms, then so is their sum \eqref{sumamorfi}. Thus the subset
\[
[LE,R]^*_{M_{\pm 2}}\supset [LE,R]_{M_{\pm 2}}^f=\{[\phi]:\phi \text{ is very full }\}
\]
is a subsemigroup.
\end{rem}

\begin{thm}\label{thm:k0lift}
Let $E$ be a countable graph and $R$ a unital $*$-algebra. Assume that $R$ is strictly properly infinite. Let $\xi:\BF(E)\to K_0(R)^*$ be a group homomorphism and let $\can':\BF(E)\to K_0(LE)^*$ be the canonical map of Remark \ref{rem:can'}. Then there is a very full $*$-homomorphism $\phi:L(E)\to R$ such that $K_0(\phi)^*\circ\can'=\xi$. If furthermore $E^0$ is finite, $[1]_E$ is as in \eqref{intro:1E} and $p\in\proj_f(R)$ is such that $\xi([1_E])=[p]$, then $\phi$ can be chosen so that, in addition to the above properties, also satisfies $\phi(1)=p$.  
\end{thm}
\begin{proof}
Because $R$ is strictly properly infinite by assumption, it has a sequence of orthogonal projections equivalent to $1$. Hence in view of Proposition \ref{prop:vfull} and the countability assumption on $E$, there are orthogonal very full projections $\{p_e:e\in E^1\}\cup\{p_v:v\in\sing(E)\}\subset\proj_f(R)$ such that, in $\cV_f(R)=K_0(R)^*$, $\xi[v]=[p_v]$ and $\xi[ee^*]=[p_e]$ for all $v\in \sing(E)$ and $e\in E^1$. If $E$ is row-finite, then proceeding as in the proof of \cite{cm2}*{Theorem 3.1} we obtain a $*$-homomorphism $\phi:LE\to R$ as required. For general countable $E$, we may choose a desingularization $E_\delta$; there is a canonical $*$-homomorphism $\iota_\delta:L(E)\to L(E_\delta)$ \cite{aap}*{Proposition 5.5} which maps vertices to vertices and edges to paths. Hence if $\psi:L(E_\delta)\to R$ is a very full $*$-homomorphism, so is $\phi\circ\iota_\delta$ and the diagram below commutes
\[
\xymatrix{
K_0(L(E))^*\ar[r]^{K_0(\phi)^*}&K_0(L(E_\delta))^*\\
\BF(L(E))\ar[u]^{\can'}\ar[r]_{\BF(\iota_\gamma)}&\BF(E_\delta)\ar[u]_{\can'}.}
\]
The proof of \cite{dritom}*{Lemma 2.3} shows that the bottom row in the diagram above is an isomorphism. 
Thus the general countable case reduces to the row-finite case. Finally if $E^0$ is finite,
$p\in\proj_1(R)$ and $\xi([1])=[p]$, then by what we have just seen there is a very full $*$-homomorphism $\psi:LE\to R$  such that $K_0^*(\psi)\circ\can'=\xi$.
Let $q=\psi(1)$; choose an MvN equivalence $y:q\iso p$. Consider the $*$-homomorphism $\ad(y):qRq\to pRp\subset R$. 

Let $\phi:=\ad(y)\circ\psi:L(E)\to R$; then $\phi$ is very full and $\phi(1)=p$.  Moreover,  $K_0(\phi)^*=K_0(\psi)^*$, hence we also have
$K_0(\phi)^*\circ\can'=\xi$. 
\end{proof}

\begin{rem}\label{rem:k0liftol}
One may ask whether in the situation of Theorem \ref{thm:k0lift}, if we further require that $E$ be finite and that $-1$ be positive in $R$, any homomorphism of $\Z[\sigma]$-modules $\ol{\BF(E)}\to K_0(R)^*$ can be lifted to a very full $*$-homomorphism $\ol{L(E)}\to R$. The argument of \ref{thm:k0lift} does not work for this purpose, since it uses the fact that the edges of $E$ are partial isometries in $L(E)$, and this is no longer true in $\ol{L(E)}$. Note however that 
it follows from Theorem \ref{thm:k0lift} and the isomorphism \eqref{olbf=bfe2} that any group homomorphism $\ol{\BF(E)}\to K_0(R)^*$ lifts to a $*$-homomorphism $\phi:L(\hat{E})\to R$. 
\end{rem}

\begin{coro}\label{coro:k0lift}
Let $E$, $R$ and $\ell$ be as in Theorem \ref{thm:k0lift}. Further assume that $\ell$ 
satisfies the $\lambda$-assumption \ref{stan:lambda}. Let $\xi:\BF(E)\to K^h_0(R)$ be a group homomorphism. Then there is a very full $*$-homomorphism $\psi:L(E)\to M_{\pm}R$ such that 
$K_0^h(\psi)\circ\can=K^h_0(\iota_+)\circ\xi$. If furthermore $E^0$ is finite and $\xi([1])=[p]$ for some $p\in \proj_f(M_{\pm}R)$, then we can choose $\psi$ so that, in addition to the above properties, also satisfies $\psi(1)=p$.
\end{coro}
\begin{proof}
By Theorem \ref{thm:k0lift} there is a $*$-homomorphism $\psi:LE\to M_{\pm}R$ --which can be chosen so that $\psi(1)=p$-- such that $K_0(\psi)^*\circ\can'=\xi$ . Hence 
\[
K^h_0(\psi)\circ\can=K_0(M_{\pm}\psi)^*\circ K_0(\iota_+)^*\circ\can'=K_0(\iota_+)^*\circ K_0(\psi)^*\circ\can'=K_0(\iota_+)^*\circ\xi.
\] 
 Let
\[
u=\begin{bmatrix}1&0&0&0\\ 0&0&-1&0\\0&1&0&0\\ 0&0&0&1\end{bmatrix}\in M_{\pm}M_{\pm}
\]
One checks that $u$ is unitary and that $\ad(u)\circ\iota_+=M_{\pm}\iota_+$. Hence $K_0(\iota_+)^*=K_0^h(\iota_+)=K_0(M_\pm\iota_+)^*$ on $K_0^h(R)=K_0(M_\pm)^*$, concluding the proof. 
\end{proof}

In the next two lemmas and elsewhere, if $E$ is a finite graph, we write $DL(E)$ for the \emph{diagonal subalgebra} of $L(E)$,
\[
DL(E)=(\bigoplus_{v\in\sink(E)}\ell v)\oplus(\bigoplus_{e\in E^1}\ell ee^*).
\]
The following lemmas will be used later on, in the proofs of Theorems \ref{thm:kklift} and \ref{thm:olkklift}.
\begin{lem}\label{lem:agree*}
Let $E$ be finite graph, $R\in\ahas$ strictly properly infinite and $\phi,\psi:L(E)\to R$ very full $*$-homomorphisms. If $K_0(\phi)^*\circ\can'=K_0(\psi)^*\circ\can'$ then there exists a very full $*$-homomorphism $\psi':L(E)\to R$ such that 
$K_0(\phi)^*\circ\can'=K_0(\psi')^*\circ\can'$, $\psi'\sim^*_{M_{\pm 2}} \psi$ and $\psi'_{|DL(E)}=\phi_{|DL(E)}$.
\end{lem}
\begin{proof} For every $\alpha\in \Theta=\sink(E)\cup E^1$, choose an MvN equivalence $x_\alpha:\psi(\alpha\alpha^*)\iso\phi(\alpha\alpha^*)$. Put $x=\sum_{\alpha\in\Theta}x_\alpha$; one checks that $x$ is an MvN equivalence $p=\psi(1)\to \phi(1)$, that $\ad(x):pRp\to qRq$ is a $*$-homomorphism and, using Lemma \ref{lem:abc}, that $\psi'=\ad(x)\circ\psi$ satisfies the requirements of the lemma.
\end{proof}
\begin{lem}\label{lem:olagree*}
Let $E$ and $R\in\ahas$ be as in Lemma \ref{lem:agree*}; further assume that $-1$ is positive in $R$. Let  $\phi,\psi:\ol{L(E)}\to R$ be very full $*$-homomorphisms. If $K^h_0(\phi)\circ\can=K^h_0(\psi)\circ\can$ then there exists a very full $*$-homomorphism $\psi':L(E)\to R$ such that 
$K^h_0(\phi)\circ\can'=K^h_0(\psi')\circ\can$, $\psi'\sim^*_{M_{\pm 2}} \psi$ and $\psi'_{|DL(E)}=\phi_{|DL(E)}$. 
\end{lem}
\begin{proof}
Let $e\in E^1$; by \eqref{eq:ee*re}, $[ee^*]=\sigma[r(e)]\in K_0^h(\ol{L(E)})$. Thus the identity $K^h_0(\phi)\circ\can=K^h_0(\psi)\circ\can$ implies that $K^h_0(\phi)[ee^*]=K^h_0(\psi)[ee^*]$. Hence we can proceed as in the proof of Lemma \ref{lem:agree*} above. 
\end{proof}

\begin{rem}\label{rem:dle}
Assume $\ell=\inv(\ell_0)$ for some commutative ring $\ell_0$. Then in the proof of Lemma \ref{lem:agree*}, $\psi'=\ad(x)\circ\psi$ is $M_2$-$*$-homotopic to $\psi$, by Remark \ref{rem:i1=i2}. Hence by the same remark, Lemma \ref{lem:agree*} recovers \cite{cm2}*{Proposition 3.5}.
\end{rem}

\section{Unitary \topdf{$K_1$}{K1} of a strictly properly infinite \topdf{$*$}{*}-algebra}\label{sec:k1}

Let $R$ be a unital $*$-algebra. For $n\ge 1$, write $\cU_nR=\cU(M_nR)$ for the group of unitary elements; set $\cU_\infty(R)=\colim_n\cU_n(R)$. For $1\le n\le \infty$, let $\cU^h_n(R)=\cU_n(M_\pm R)$. Put
\begin{equation}\label{k1}
K_1(R)^*=\cU_\infty(R)_{ab},\,\, K_1^h(R)=K_1(M_{\pm}R)^*=\cU^h_\infty(R)_{ab}.
\end{equation}
\begin{lem}\label{lem:k1morite}
Let $R$ be a unital $*$-algebra and $p\in R$ a strictly full projection. Then the inclusion $pRp\to R$ induces an isomorphism $K_1(pRp)^*\iso K_1(R)^*$. 
\end{lem}
\begin{proof}
The proof is similar to that of Lemma \ref{lem:vmorite} once we observe that $K_1(R)^*=\colim_{\fP(R)}\cU(pRp)_{ab}$. 
\end{proof}
\begin{rem}\label{rem:k1=kh1}
If $R\in\ahas$ is unital and $-1$ is positive in $R$, then, as explained in Section \ref{sec:k0} after Lemma \ref{lem:morite}, $M_{\pm}R$ is $*$-isomorphic to a strictly full corner of $M_{n+1}R$, so $K_1(R)^*\to K_1^h(R)$ is an isomorphism by Lemma \ref{lem:k1morite}. 
\end{rem}
\begin{prop}\label{prop:k1pi}
Let $R$ be a strictly properly infinite $*$-algebra. Let $\cN(R)\triqui \cU(R)$ be the smallest normal subgroup containing the subset $\{u^{-1}(xux^*+1-xx^*): u\in \cU(R), x^*x=1\}$. Then $K_1(R)^*=\cU(R)/\cN(R)$.
\end{prop}
\begin{proof}
We keep the notation as in the proof of Lemma \ref{lem:k1morite}. Because $R$ is strictly properly infinite, the full subcategory $\mathbb{I}$ of $\fP(R)$ whose only object is the identity element $1\in R$ is cofinal. Hence
\[
K_1(R)^*=\colim_{\mathbb{I}}\cU(R)_{ab}=\cU(R)/[\cU(R):\cU(R)]\cdot \cN(R).
\]
It remains to show that $\cN(R)\supset [\cU(R),\cU(R)]$. Let $s_1,s_2$ be orthogonal isometries and let $u,v\in \cU(R)$. Modulo $\cN(R)$, 
\begin{align*}
uv\equiv &(s_1us_1^*+(1-s_1s_1^*))\cdot(s_2vs_2^*+(1-s_2s_2^*))\\
=&s_1us_1^*+s_2vs_2^*+1-\sum_{i=1}^2s_is_i^*\\
=&(s_2vs_2^*+(1-s_2s_2^*))(s_1us_1^*+(1-s_1s_1^*))
\equiv vu.
\end{align*}
\end{proof}
Put
\begin{equation}\label{def:kv1coeq}
KV_1(R)^*=\coker(\ev_0-\ev_1:K_1(R[t])^*\to K_1(R)^*),\,\, KV_1^h(R)=KV_1(M_{\pm}R)^*.
\end{equation}
One checks that $KV_1^h$ as defined above agrees with that of \cite{kv2}.
Let 
\begin{gather*}
\cU(R)_0=\{u\in\cU(R)|(\exists U\in \cU(R[t]))\, U(0)=1,\, U(1)=u\},\\
\cU_n(R)_0=\cU(M_nR)_0,\,\,\cU^h_n(R)_0=\cU_n(M_\pm R)_0\,\, (n\ge 1).
\end{gather*}

\begin{lem}\label{lem:kv1}
Let $R\in\ahas$ be strictly properly infinite and such that $-1$ is positive in $R$. Regard $\cU(R)\subset \cU^h_2(R)$ through the group monomorphism induced by $\iota_+\circ\iota_1$. 
Then $$KV^h_1(R)=\cU(R)/\cU(R)\cap \cU^h_2(R)_0.$$ 
\end{lem}
\begin{proof}
By \eqref{def:kv1coeq}, Proposition \ref{prop:k1pi} and Remark \ref{rem:k1=kh1}, the inclusion $\iota_+:\cU(R)\subset \cU^h_1(R)$ induces an isomorphism 
\[
KV_1(R)^*=\cU(R)/\cN(R)\cdot(\cU(R)\cap \cU_\infty(R)_0)\iso KV^h_1(R).
\] 
In particular,
\[
\cU(R)\cap \cU^h_2(R)_0\subset\cN(R)\cdot(\cU(R)\cap \cU_\infty(R)_0).
\] 
Let $x\in R$ be an isometry, $g(t)\in \Gl_2(R[t])$ as in the proof of Lemma \ref{lem:i1=i2} and $1_n$ the $n\times n$- identity matrix. 
 Let $y(t)=\ad(g(t))(x\oplus 1),z(t)=\ad(g(t))(x^*\oplus 1)\in M_2R$; then $h(t)=1_2-y(t)z(t)+y(t)\iota_1(u)z(t)\in\Gl_2(R[t])$ connects $\iota_1(u)=u\oplus 1$ with $\iota_1(1-xx^*+xux^*)$ in $\Gl_2(R)$. Let $c(t)=c(h(t),h(t)^{-1})$ be as in \cite{cv}*{Lemma 5.4}; then $c(t)\in \cU^h_2(R)$, and connects $\iota_+(\iota_1(u))$ with $\iota_+(\iota_1(1-xx^*+xux^*))$. It follows that $\cN(R)\subset \cU(R)\cap \cU^h_2(R)_0$, since $\cU(R)\cap \cU^h_2(R)_0\subset\cU(R)$ is a normal subgroup. It remains to show that $\cU^h_2(R)_0\supset\cU(R)\cap \cU_\infty(R)_0$. Let $u\in\cU(R)$ and suppose that for some $n\ge 2$ there exists $U(t)\in\cU_n(R[t])$ such that $U(0)=1$ and $U(1)=\iota_n(u)=u\oplus 1_{n-1}$. Because $R$ is strictly properly infinite, there exists $x=(x_1,\dots,x_n)\in R^{1\times n}$ such that
$xx^*=1\in R$ and $x^*x=1_n\in M_nR$. Then $V(t)=1-xx^*+xU(t)x^*\in\cU(R[t])$ satisfies $V(0)=1$ and
\[
V(1)=1-xx^*+x(u\oplus 1_{n-1})x^*=1-x_1x_1^*+x_1ux_1^*.
\]
By the first part of the proof, $\iota_+(\iota_1(V(1)))$ is connected to $\iota_+(\iota_1(u))$ by a path $W(t)\in\cU^h_2(R[t])$; hence $W(t)((\iota_+\circ\iota_1)(V(1-t))^{-1})$ connects $\iota_+(\iota_1(1))$ to $\iota_+(\iota_1(u))$. 
\end{proof}

\begin{rem}\label{rem:kv1}
Consider the particular case of Lemma \ref{lem:kv1} when $\ell=\inv(\ell_0)$. By Remark \ref{rem:i1=i2} $R=\inv(S)$ for some properly infinite $\ell_0$-algebra $S$, $KV_1^h(R)$ is Karoubi-Villamayor's $KV_1(S)$ \cite{kv1}, and one may replace $g(t)$ in the proof of the lemma by a unitary $h(t)$, obtaining
\[
KV_1(S)=KV_1^h(R)=\cU_1(R)/\cU_1(R)\cap\cU_2(R)=\Gl_1(S)/\Gl_1(S)\cap\Gl_2(S)_0. 
\]
\end{rem}
\section{Poincar\'e duality}\label{sec:duality}

Let $E$ be a finite graph. The \emph{dual graph} $E_t$ is the graph with $E_t^i=E^i$ for $i=0,1$ and with 
source and range maps $s_t=r$ and $r_t=s$. Write $e_t$ for an edge $e\in E^1$ regarded as an edge of $E_t$. 
The purpose of this section is to prove Theorem \ref{thm:duality}, which is an algebraic version of a similar result for graph $C^*$-algebras
\cite{kamiput}. First we need the following lemma.

\begin{lem}\label{lem:index}
Let $\pi:R\to S$ be a surjective, unital homomorphism of $*$-algebras, set $I=\ker(\pi)$ and let $\partial:K_1^h(S)\to K_0^h(I)$ be the connecting map. Let $u\in \cU_n(S)$. Assume that there exists a partial isometry $\hat{u}\in M_n R$ such that $\pi(\hat{u})=u$. Then $\partial\iota_+[u]=\iota_+([1-\hat{u}^*\hat{u}]-[1-\hat{u}\hat{u}^*])$.
\end{lem}
\begin{proof}
For every pair of elements $\hat{u}, \hat{v}\in M_nR$ such that $\pi(\hat{u})=u$ and $\pi(\hat{v})=u^*$, we can lift $\diag(u,u^*)\in \cU(M_{2n}S)$ to an elementary matrix 
$h=h(\hat{u},\hat{v})\in E_{2n}R$; a formula for this matrix is given in \cite{friendly}*{Formula (17)}. One checks that if $\hat{u}$ is a partial isometry, then $h(\hat{u}):=h(\hat{u},\hat{u}^*)\in \cU_{2n}R$. Thus if $p\in M_nR$ is the identity matrix, we have 
$\partial(\iota_+[u])=\iota_+([\ad(h(\hat{u}))(p)]-[p])$, and one computes that $[\ad(h(\hat{u}))(p)]-[p]=[1-\hat{u}^*\hat{u}]-[1-\hat{u}\hat{u}^*]\in K_0(I)^*$.
\end{proof}

Recall from \cite{cv}*{Example 6.12} that if $B$ is any $*$-algebra, then the functor $-\otimes B:\ahas\to\ahas$ induces a functor $kk^h\to kk^h$ which we again name $-\otimes B$, such that $j^h\circ (-\otimes B)=(-\otimes B)\circ j^h$. 

Following \cite{lpabook}*{Definition 6.3.11 (iii)} we call a graph \emph{essential} if it contains no sources and no sinks. Thus a finite graph $E$ is essential if and only if both $E$ and $E_t$ are regular. 

\begin{thm}\label{thm:duality}
Let $E$ be a finite essential graph. Then the functors $-\otimes \Omega L(E_t)$ and $-\otimes\Omega \ol{L(E_t)}:kk^h\to kk^h$ are right adjoint to the functors

\goodbreak

$-\otimes LE$ and $-\otimes\ol{LE}:kk^h\to kk^h$. Thus for every $R,S\in\ahas$ there are natural isomorphisms of $KH_0^h(\ell)$-modules
\begin{gather*}
kk^h(R\otimes LE,S)\iso kk_1^h(R,S\otimes L(E_t)),\\
kk^h(R\otimes \ol{LE},S)\iso kk_1^h(R,S\otimes \ol{L(E_t)}).
\end{gather*}
\end{thm}
\begin{proof}
Let $\cP$ be the set of finite paths in $E$ and $\cP_{\ge 1}\subset\cP$ the subset of paths of positive length; let $X=(\cP_{\ge 1})_+$ be the pointed set obtained by adding a basepoint $\bullet$. 
Let $\pi:\Gamma_X\to\Sigma_X$ be the projection. With notations as in \eqref{parribajo}, for each $e\in E^1$, put
\[
\rho_1(e_t)=\pi(\sum_{\alpha\in\cP_{s(e)}}\epsilon_{\alpha e,\alpha}),\,\,\rho_2(e)=\pi(\sum_{\alpha\in\cP^{r(e)}}\epsilon_{e\alpha,\alpha}).
\]
One checks (as in \cite{kamiput}*{Proposition 4.2}) that the assignments $e_t\mapsto \rho_1(e_t)$ and $e\mapsto \rho_2(e)$ extend to unital $*$-homomorphisms $\rho_1:L(E_t)\to \Sigma_X$ and $\rho_2:L(E)\to \Sigma_X$ and that $\rho_1(a)$ and $\rho_2(b)$ commute for every $a\in L(E_t)$ and $b\in L(E)$. Hence for $\cE=L(E_t)\otimes L(E)$ we have a $*$-homomorphism $\rho:\cE\to \Sigma_X$, which defines a class
$\kappa=j^h(\rho)[+1]\in kk^h(\Omega L(E_t)\otimes L(E),\ell)=kk_{-1}^h(L(E_t)\otimes L(E),\ell)$. Hence we have a homomorphism of $KH_0^h(\ell)$-modules
\begin{equation}\label{map:adj1}
kk_1^h(R,S\otimes L(E_t))\to kk^h(R\otimes L(E),S),\,\,\xi\mapsto (S\otimes\kappa)\circ (\xi\otimes L(E)).
\end{equation}
Consider the elements $p=\sum_{v\in E^0}v\otimes v$ and $u=\sum_{e\in E^1}e\otimes e^*_t$ of $\cE'=L(E)\otimes L(E_t)$. One checks that $u\in\cU(p\cE' p)$. Hence
$u_1=u+1-p\in \cU(\cE')$; write $\nabla$ for the image of the class $[u_1]\in K_1(\cE')^*$ under the composite of canonical maps
\[
K_1^*(\cE')\to K_1^h(\cE')\to KH^h_1(\cE')=kk_1^h(\ell,\cE').
\]
We have another $KH_0^h(\ell)$-linear homomorphism
\begin{equation}\label{map:adj2}
kk^h(R\otimes L(E),S)\to kk_1^h(R, S\otimes L(E_t)),\,\,\eta\mapsto ((\eta\otimes L(E_t))[+1])\circ (R\otimes \nabla).
\end{equation}
To prove that \eqref{map:adj1} and \eqref{map:adj2} are isomorphisms, it suffices to establish the following identities
\begin{gather}\label{eq:dual}
(\kappa\otimes  L(E_t))\circ (L(E_t)\otimes \nabla)=-\id_{j^h(L(E_t))},\\ (L(E)\otimes \kappa)\circ (\nabla\otimes L(E))=-\id_{j^h(L(E))}.\nonumber  
\end{gather}
Consider the cone extension
\[
\xymatrix{M_X\ar@{ >-}[r]&\Gamma_X\ar@{>>}[r]& \Sigma_X.}
\]
Tensor this extension on the right with $L(E_t)$; the first identity of \eqref{eq:dual} boils down to the assertion that the index map $KH^h_1(\Sigma_XLE_t)\to KH^0(M_XLE_t)\cong KH^h_0(LE_t)$ sends the class of $\sum_e\rho(1_{LE_t}\otimes e)\otimes e_t^*$ to minus the class of $\iota_+(1)$; the second identity is the analogous statement for the index of the class of $\sum_ee\otimes\rho(e_t^*\otimes 1_{L(E)})$. Both are straightforward using Lemma \ref{lem:index}. Thus the first adjunction of the theorem is proved. Observe that the homomorphisms $\rho_1$ and $\rho_2$ defined above are also $*$-homomorphisms for the involutions of $\ol{L(E)}$, 
$\ol{L(E_t)}$ and $\ol{\Sigma}_X$. Notice also that the elements $p$ and $v$ are a projection and a unitary also for the involution of $\ol{L(E)}\otimes\ol{L(E_t)}$. Hence the identities \eqref{eq:dual} also prove the second adjunction of the theorem.
\end{proof}

\begin{coro}\label{coro:duality}
Let $E$ and $S$ be as in Theorem \ref{thm:duality}. Assume that $S$ is unital and contains a central element $x$ such that $xx^*=-1$. 
There there is a natural isomorphism of $KH_0^h(\ell)$-modules
\[
kk^h(\ol{L(E)},S)\iso kk^h(L(E),S).
\] 
\end{coro}
\begin{proof}
Using Theorem \ref{thm:duality} at the first and third steps and Example \ref{ex:zgr-1} at the second step, we obtain natural isomorphisms
\[
kk^h(\ol{L(E)},S)\iso KH_1^h(S\otimes\ol{L(E_t)})\iso KH^h_1(S\otimes L(E_t))\iso kk^h(L(E),S).
\]
\end{proof}

\begin{rem}\label{rem:pdkk}
Let $E$ be a finite essential graph. Theorem \ref{thm:duality} and the equivalence of Remark \ref{rem:kkvkkh} together imply that $L(E)$ and $L(E_t)$ are Poincar\'e dual not only in $kk^h$ but also in $kk$. 
\end{rem}

\section{Universal coefficient theorem}\label{sec:uct}

Let $E$ be a graph.  Put
\[
\ol{\BF}^\vee(E)=\coker(I^t-\sigma A_E),\,\, \BF^\vee(E)=\ol{\BF}^\vee(E)\otimes_{\Z[\sigma]}\Z.
\]

Let $R\in\ahas$; write $\hom$ for homomorphisms of abelian groups. Consider the maps
\begin{gather}\label{map:xi0}
\ev:kk^h(LE,R)\to \hom(\BF(E),KH_0^h(R)),\,\,\\ \ev:kk^h(\ol{L(E)},R)\to \hom_{\Z[\sigma]}(\ol{\BF}(E),KH_0^h(R))\nonumber\\
\xi\mapsto \xi_0=KH_0^h(\xi)\circ\can.\nonumber
\end{gather}
\begin{thm}\label{thm:uct}
Let $E$ be a graph such that $|E^0|<\infty$ and let $R\in\ahas$. Then the maps 
\eqref{map:xi0} are surjective and fit into exact sequences
\begin{gather}\label{seq:kkler}
0\to KH_1^h(R)\otimes\BF^\vee(E)\to kk^h(L(E),R)\to \hom(\BF(E),KH_0^h(R))\to 0,\\
0\to KH_1^h(R)\otimes_{\Z[\sigma]}\ol{\BF}^\vee(E)\to kk^h(\ol{L(E)},R)\to \hom_{\Z[\sigma]}(\ol{\BF}(E),KH_0^h(R))\to 0.\nonumber
\end{gather}
\end{thm}
\begin{proof}
Applying $kk^h(-,R)$ to the triangle of Theorem \ref{thm:fundtriang}, and using that for any finite set $X$,
\begin{equation}\label{map:caniso}
kk^h(\ell^X,R)\cong\hom(\Z^X,KH_0^h(R))=\hom_{\Z[\sigma]}(\Z[\sigma]^X,KH_0^h(R))
\end{equation}
one obtains exact sequences as in the theorem. One checks that the surjections therein agree with the evaluation maps \eqref{map:xi0}.
\end{proof}

\begin{rem}\label{rem:etdual}
Let $E$ be a finite essential graph. Then $A_{E_t}=A_E^t$, and 
\[
\BF^\vee(E)=\BF(E_t),\,\, \ol{\BF}^\vee(E)=\ol{\BF}(E_t). 
\]
\end{rem}

\begin{lem}\label{lem:dualindex}
Let $E$ and $R$ be as in Theorem \ref{thm:duality}. Then the following diagrams, where the horizontal maps are as in Theorem \ref{thm:uct}, the slanted ones as in Theorem \ref{thm:khseq} and the vertical ones as in Theorem \ref{thm:duality}, commute.
\begin{gather*}
\xymatrix{\BF(E_t)\otimes KH_1^h(R)\ar[dr]_{\ref{thm:khseq}}\ar[r]^(.55){\ref{thm:uct}}& kk^h(LE,R)\ar[d]_{\wr}^{\ref{thm:duality}}\\
           & KH^h_1(R\otimes L(E_t))}\\					
\xymatrix{\ol{\BF(E_t)}\otimes_{\Z[\sigma]} KH_1^h(R)\ar[dr]_{\ref{thm:khseq}}\ar[r]^(.55){\ref{thm:uct}}& kk^h(\ol{LE},R)\ar[d]_{\wr}^{\ref{thm:duality}}\\
           & KH^h_1(R\otimes \ol{L(E_t)})}
\end{gather*}
\end{lem}
\begin{proof} Consider the first diagram first.
The composite of the slanted arrow with the surjection $p:\Z^{E^0}\otimes KH_1^h(R)\to \BF(E_t)\otimes KH_1^h(R)$ maps an element $\xi=\chi_v\otimes\xi_v$ to $\sum_{v\in E^0}[v]\cup \xi_v$. Let $[u_1]\in KH_1^h(LE\otimes LE_t)$ be the class  of the element $u_1$ appearing in the proof of Theorem \ref{thm:duality}. Consider the algebra extension 
$\fE$ of $L(E)\otimes L(E_t)$ which results by applying $\otimes L(E_t)$ to the Cohn extension \eqref{seq:cohnext}. Let $\eta=\partial_\fE([u_1])\in KH_0^h(\ell^{E^0}\otimes L(E_t))$ be the index of $[u_1]$. 
The composite of the horizontal and vertical arrows with $p$ sends $\chi_v\otimes\xi$ to 
\[
(\xi\otimes LE_t)\circ \eta\in kk_1^h(\ell^{E^0}\otimes L(E_t),R\otimes L(E_t)).
\] 
A straightforward calculation, using Lemma \ref{lem:index}, shows that $\eta=[\sum_{v\in E^0}\chi_v\otimes v]$. It follows that the first diagram commutes. The same argument, substituting $\ol{L(E)}$ and $\ol{L(E_t)}$ for $L(E)$ and $L(E_t)$ and the extension \eqref{seq:cohnext2} for \eqref{seq:cohnext}, proves that also the second diagram commutes.
\end{proof}

\begin{rem}\label{rem:uctkk}
It follows using the equivalence of Remark \ref{rem:kkvkkh}, that the first exact sequence of Theorem \ref{thm:uct} and the first commutative diagram of Lemma \ref{lem:dualindex} still hold if we remove the superscript $h$ everywhere.
\end{rem}

\section{Lifting \topdf{$kk^h$}{kkh}-maps to algebra maps}\label{sec:kklift}

Let $E$ be a finite graph such that $\sink(E)=\emptyset$. Set
\begin{equation}\label{Lepsilon}
L^0(E)=L(E),\,\, L^1(E)=\ol{L(E)}.
\end{equation}
For $\epsilon\in\{0,1\}$, let $\phi : L^{\epsilon}(E) \to R$ be a unital $*$-algebra homomorphism with $R$ strictly properly infinite. Assume that $\phi$ is very full. Set 
\begin{align}\label{rphi}
R_\phi = &\{ x \in R \ : \phi (e e^\ast)x=x\phi(ee^\ast),\text{ for all } e \in E^1\}\\
       =& \oplus_{e \in E^1} \phi (e e^\ast) R \phi (e e^\ast).\nonumber 
\end{align}
Because $\phi$ is very full, $\phi(ee^*)\in\proj_f(R)$ for all $e\in E^1$, 
whence the inclusion $\phi(ee^*)R\phi(ee^*)\subset R$ induces an isomorphism in $K^*_1$, by Lemma \ref{lem:k1morite}. It follows that the direct sum $R_\phi\subset R^{E^1}$ of those inclusions induces an isomorphism 
\begin{equation}\label{map:adjoint}
K_1(R_\phi)^* = \bigoplus_{e \in E^1} K_1( \phi (e e^\ast) R \phi (e e^\ast))^*\iso (K_1(R)^*)^{E^1}.
\end{equation}
Let $f:X\to Y$ be a map between finite sets and $M$ an abelian group. We write $f_*:\Z^X\to \Z^Y$, $f_*(\chi_x)=\chi_{f(x)}$; we shall abuse notation and also write $f_*$ for $f_*\otimes M$. In particular the source map $s:E^1\to E^0$ induces a homomorphism
$s_*:K_1(R)^{E^1}\to K_1(R)^{E^0}$. Consider the composite
\begin{equation}\label{map:longcompo}
\partial: K_1(R_\phi)^*\cong (K_1(R)^*)^{E^1}\overset{s_*}\lra (K_1(R)^*)^{E^0}\to K_1^h(R)^{E^0}\to kk^h(L^{\epsilon}(E),R)
\end{equation}
Let 
\begin{equation}\label{eq:elu}
u=(u_e)_{e\in E^1}\in \cU(R_{\phi})=\bigoplus_{e\in E^1}\cU(\phi(ee^*)R\phi(ee^*));
\end{equation}
consider the $*$-homomorphism 
\begin{equation}\label{map:phiu}
\phi_u:L^\epsilon(E)\to R, \phi_u(e)=u\phi(e).
\end{equation}
\begin{lem}\label{lem:borde}
Let $E$ be a finite graph, $R$ a $*$-algebra and $\phi:L(E)\to R$ a unital $*$-homomorphism. Assume that $E$ is purely infinite simple and that $R$ is strictly properly infinite. Then
\[
j^h(\phi_u)=j^h(\phi)+\partial([u]).
\]
\end{lem}
\begin{proof}
Let $n=|\sour(E)|$; we shall prove the lemma by induction on $n$. First we assume that $n=0$. By Lemma \ref{lem:dualindex}, the composite of $\partial$ with the isomorphism \eqref{map:adj2} sends $[u]$ to the class in $KH^h_1(R\otimes L^{\epsilon}(E_t))$ of the element
\[
\prod_{e\in E^1}((1-\phi(ee^*)+u_e)\otimes s(e)+\sum_{v\ne s(e)}1\otimes v).
\]
Let $\psi:L^\epsilon(E)\to R$ be a $*$-homomorphism such that $\phi(ee^*)=\psi(ee^*)$ for all $e\in E^1$. Then \eqref{map:adj2} sends $j^h(\psi)$ to the $KH_1^h$-class of 
\[
\xi(\psi):=1\otimes 1-\sum_{v\in E^0}\phi(v)\otimes v+\sum_{f\in E^1}\psi(f)\otimes f_t^*.
\]
We shall abuse notation and write $u_e$ for the element of $\bigoplus_{f\in E^1}\cU(\phi(ee^*)R\phi(ee^*))$ whose $f$-coordinate is
$u_e^{\delta_{f,e}}\phi(ff^*)$. One checks that 
\begin{align*}
((1-\phi(ee^*)+u_e)\otimes s(e)+\sum_{v\ne s(e)}1\otimes v)\xi(\psi)=\xi(\psi_{u_e}).
\end{align*}
Starting with $\psi=\phi$ and applying the identity above repeatedly we obtain that the isomorphism \eqref{map:adj2} sends the two sides of the identity of the lemma to the same element of $KH_1^h(R\otimes L^{\epsilon}(E_t))$. This concludes the proof of the case $n=0$. 
Next let $n\ge 0$, assume that the lemma holds for graphs with at most $n$ sources, and let $E$ be a finite graph without sinks and with $n+1$ sources. Let $v\in \sour(E)$ and let $F=E_{\backslash v}$ be the source elimination graph. Let $p=1-v\in L(E)$. By Lemmas \ref{lem:remove} and \ref{lem:morite}, the corner embedding $\inc_p: L^{\epsilon}(E_{\backslash v})\to L^{\epsilon}(E)$  of Lemma \ref{lem:remove}  induced by the inclusion $E_{\backslash v}\subset E$ is $kk^h$-equivalence. Let $q=\phi(p)$, $S=qRq$ and $\phi'=(\phi\circ\inc_p)^{|S}:L(E_{\backslash v})\to S$ be the restriction/corestriction of $\phi$. Let 
$$
\pi:R_\phi=\bigoplus_{e\in E^1}\phi(ee^*)R\phi(ee^*)\onto S_\phi=\bigoplus_{e\in E^1,\\ s(e)\neq v}\phi(ee^*)R\phi(ee^*)
$$
be the projection. The map $\pi$ together with the corner inclusions $\inc_p$ above and $\inc_q:S\subset R$ induce a commutative diagram
\[
\xymatrix{K_1^h(R_\phi)\ar[r]\ar[dd]^{\pi_*}& kk^h(L^{\epsilon}(E),R)\ar[dr]^{\inc_p^*}&\\
&&kk^h(L(E_{\backslash v}),R)\\
           K_1^h(R_{\phi'})\ar[r]& kk^h(L^{\epsilon}(E_{\backslash v}),R)\ar[ur]_{(\inc_q)_*}&}
\] 
By induction, the composite of the vertical map on the left with the horizontal map at the bottom sends the class of $u=(u_e)$ to 
$j^h(\phi'_{\pi(u)})-j^h(\phi')$. Composing with the upward slanted arrow, we obtain $j^h(\phi_u\circ\inc_p)-j^h(\phi\circ\inc_p)$. Because the downward slanted map is an isomorphism, it follows that the top horizontal arrow maps $[u]\mapsto j^h(\phi_u)-j^h(\phi)$. 
\end{proof}

\begin{thm}\label{thm:kklift}
Let $E$ be a finite, purely infinite simple graph and $R\in\ahas$ a $K_0^h$-regular, strictly properly infinite $*$-algebra over a ring $\ell$ satisying the $\lambda$-assumption. Assume that $-1$ is positive in $R$. Then the map
\begin{equation}\label{map:kklift}
j^h:[L(E),R]^f_{M_{\pm 2}}\to kk^h(L(E),R)
\end{equation}
is a semigroup isomorphism. For each $p\in\proj_f(R)$, we have
\begin{equation}\label{eq:ev-1p}
\ev^{-1}([p])=j^h(\{[\phi]\, :\,\phi(1)=p\}).
\end{equation}
\end{thm}
\begin{proof}
Let $\xi\in kk^h(L(E),R)$. Because by assumption, $-1$ is positive in $R$, the map $K_0(R)^*\to K_0^h(R)$ is an isomorphism, as explained in Section \ref{sec:k0}. Hence
by Theorem \ref{thm:k0lift} there is a very full $*$-homomorphism $\phi:L(E)\to R$ such that $K_0^h(\phi)\circ\can=\ev(\xi)$. Let $p=\phi(1)$ and $\inc:pRp\subset R$ the inclusion. By Lemma \ref{lem:morite}, there exists $\xi'\in kk^h(L(E),pRp)$ such that $\inc_*(\xi')=\xi$. By Lemma \ref{lem:borde}, there exists $u\in \cU((pRp)_{\phi})$ such that $\xi'=j^h(\phi_u)$. Hence $\xi=j^h(\inc\circ\phi_u)$ and the map of the theorem is surjective. Next let 
$\phi,\psi:LE\to R$ be very full $*$-homomorphisms such that $j^h(\phi)=j^h(\psi)$. Then $K_0^h(\phi)\circ\can=K_0^h(\psi)\circ\can$; because $-1$ is positive in $R$, this implies that $K_0(\phi)^*\circ\can'=K_0(\psi)^*\circ\can'$. Hence by Lemma \ref{lem:agree*}, we may assume that $\psi_{|DL(E)}=\phi_{|DL(E)}$. Thus there is $p\in\cV_f(R)$ with $\psi(1)=\phi(1)=p$; by Lemma \ref{lem:vmorite}, we may assume that $p=1$. For each $e\in E^1$, let 
\[
u_e=\psi(e)\phi(e^*)\in\cU(\phi(ee^*)R\phi(ee^*)).
\]
Put $u=(u_e)\in \bigoplus_{e\in  E^1}\cU(\phi(ee^*)R\phi(ee^*))=\cU(R_\phi)$. 
A calculation shows that $\psi=\phi_u$. Hence $\partial([u])=0$, by Lemma \ref{lem:borde}. As in \cite{ror}*{Section 5}, we consider the $*$-homomorphism $\lambda: R_\phi\to R_{\phi}$, 
$\lambda(a)=\sum_{e\in E^1}\phi(e)a\phi(e^*)$. Because $R$ is $K^h_0$-regular by assumption, $KH^h_1(R)=KV_1^h(R)$. Let $B=B_E$ be as in \eqref{mat:B}; using that, by Proposition \ref{prop:k1pi}, $\cN(R)\subset\cU(R)$ maps to the trivial subgroup in $KV^h_1(R)$, one checks that the following diagram commutes
\begin{equation}\label{diag:lambdaB}
\xymatrix{
KV_1^h(R_{\phi})\ar[d]^{\wr}\ar[r]^\lambda& KV_1^h(R_\phi)\ar[d]^\wr\\
KV^h_1(R)^{E^1}\ar[r]_{B^t}& KV^h_1(R)^{E^1}.}
\end{equation}
Thus identifying $[u]=\sum_{e\in E^1}[u_e]$, and using that $\partial([u])=0$, it follows that there exists $\nu\in \cU(R_\phi)$ such that $[u]=[\nu\lambda(\nu)^{-1}]$. Hence by Lemma \ref{lem:kv1} there is $U(t)\in\cU^h_2(R_\phi[t])$ with $U(0)=\nu\lambda(\nu)^{-1}$, $U(1)=u$. Thus by Lemma \ref{lem:abc}, for the inclusion $\inc:R_\phi\subset R$, we obtain 
\[
\psi=\phi_u\sim^*_{M_{\pm 2}}\phi_{\nu\lambda(\nu)^{-1}}=\inc \circ\ad(\nu)\circ\phi\sim^*_{M_{\pm 2}}\phi.
\]
This proves the first assertion of the theorem. Next let $p\in \proj_f(R)$ and let $\xi\in kk^h(LE,R)$ such that $\ev(\xi)=p$. By Theorem \ref{thm:k0lift} there is a very full $*$-homomorphism $\phi:LE\to R$ such that $\phi(1)=p$ and $\ev(j^h(\phi))=\ev(\xi)$, and by Lemma \ref{lem:borde} there is $u\in\cU(R_\phi)$ such that $j^h(\phi_u)=\xi$. This finishes the proof, since $\phi_u(1)=\phi(1)=p$. 
\end{proof}

In the next corollary we refer to the stable $M_{\pm 2}$-homotopy relation $\sim^s_{M_{\pm 2}}$ defined in \eqref{defi:stably}.

\begin{coro}\label{coro:kklift}
For every class $[\phi]\in [L(E),R]^*_{M_{\pm 2}}$ there exists a unique class $[\phi^f]\in [L(E),R]^f_{M_{\pm 2}}$ such that $\phi\sim^s_{M_{\pm 2}}\phi^f$. The map
\[ 
[L(E),R]^*_{M_{\pm 2}}\to [L(E),R]^f_{M_{\pm 2}},\,\, [\phi]\mapsto [\phi^f]
\]
is the group completion \eqref{map:complete}. 
\end{coro}
\begin{proof} By Theorem \ref{thm:kklift} there is a very full $*$-homomorphism $0^f:L(E)\to R$ such that $j^h(0^f)=0$. For a $*$-homomorphism $\phi:L(E)\to R$, set $\phi^f=\phi\boxplus 0^f$. Then $\phi^f$ is very full and $\phi^f\sim^s_{M_{\pm}}\phi$. It is straightforward to check that $[\phi]\mapsto [\phi^f]$ is well-defined and has the universal property of group completion. 
\end{proof}

\begin{thm}\label{thm:olkklift}
Let $E$, $R$ and $\ell$ be as in Theorem \ref{thm:kklift}. Then the map
\begin{equation}\label{map:olkklift}
j^h:[\ol{L(E)},R]^f_{M_{\pm 2}}\to kk^h(\ol{L(E)},R)
\end{equation}
is a semigroup monomorphism. For each $p\in\proj_f(R)$, we have
\[
\ev^{-1}([p])=j^h(\{[\phi]\,:\,\phi(1)=p\}).
\]
If furthermore $\sour(E)=\emptyset$ and $R$ contains a central element $x$ such that $xx^*=-1$, then \eqref{map:olkklift} is an isomorphism. 
\end{thm}
\begin{proof} Proceed as in the proof of the injectivity part of Theorem \ref{thm:kklift}, substituting Lemma \ref{lem:olagree*} for Lemma \ref{lem:agree*}. One checks that the analogue of \eqref{diag:lambdaB} holds with $\sigma B^t$ substituted for $B^t$; the rest of the proof of the first assertion of the current theorem follows as in Theorem \ref{thm:kklift}. To prove the second assertion, begin by observing that the bijection $\theta^x$ of Example \ref{ex:zgr-1} passes down to a homomorphism between the monoids of $M_{\pm 2}$-$*$ homotopy classes. Hence it induces a map $[\ol{L(E)},R]^f_{M_{\pm 2}}\to [L(E),R]^f_{M_{\pm 2}}$, which, together with the isomorphism of Corollary \ref{coro:duality} and the maps \eqref{map:kklift} and \eqref{map:olkklift}, fits into a diagram
\[
\xymatrix{ [\ol{L(E)},R]^f_{M_{\pm 2}}\ar[r]^{\eqref{map:olkklift}}\ar[d]^{\theta^{x}}_{\wr}&kk^h(\ol{L(E)},R)\ar[d]^{\wr}_{\eqref{coro:duality}}\\
           [L(E),R]^f_{M_{\pm 2}}\ar[r]^{\sim}_{\eqref{map:kklift}}& kk^h(L(E),R)}
\]
A straightforward calculation shows that the diagram above commutes; this finishes the proof. 
\end{proof}

\begin{rem}\label{rem:approx}
Let $E$, $R$ and $\ell$ be as in Theorems \ref{thm:kklift} and \ref{thm:olkklift}, $\epsilon\in\{0,1\}$ and $L^{\epsilon}(E)$
as in \eqref{Lepsilon}, and let $\phi,\psi:L^\epsilon(E)\to R$ be very full $*$-homomorphisms.  Put $p=\phi(1)$ and $S=pRp$. The common argument for the proof of injectivity of the map $j^h$ in both theorems shows that $j^h(\phi)=j^h(\psi)$ if and only if there exist:
\begin{itemize}
\item an $MvN$ equivalence $R\owns y_e:\psi(ee^*)\to \phi(ee^*)$ for every $e\in E^1$ and
\item elements $U(t)\in\cU^h_2(S_\phi[t])$, and $u,\nu\in\cU(S_{\phi})$
such that $\iota_+\iota_1(u)\oplus 1_3=U(1)$,  $U(0)=\iota_+(\iota_1(\nu\lambda(\nu)^{-1}))\oplus 1_3$ and such that for $y=\bigoplus_{e\in E^1}y_e$, we have
\[
\ad(y)\circ\psi=\phi_{u}.
\]
\end{itemize}
\end{rem}

\begin{ex}\label{ex:chris}
Let $F$ be a graph, $\epsilon\in\{0,1\}$ and $L^\epsilon F$ as in \eqref{Lepsilon}. If $F^0$ is finite, then $L^\epsilon(F)\otimes L_\infty$ is strictly properly infinite. If furthermore $E$ is finite and $\ell$ is regular supercoherent, then $L(F)$ is regular supercoherent, and thus $L(F)\otimes L_\infty$
is $K$-regular, by Lemma \ref{lem:lekreg}.  If in addition $2$ is invertible in $\ell$, then $L^{\epsilon}(F)\otimes L_\infty$ is $K^h$-regular, again by Lemma \ref{lem:lekreg}. Moreover $j^h(L^{\epsilon}(F)\otimes L_\infty)\cong j^h(L^{\epsilon}(F))$, by Theorem \ref{thm:fundtriang}. Hence by Theorem \ref{thm:kklift}, if $E$ is purely infinite simple and finite we have 
\begin{gather}\label{map:chris1}
kk^h(L(E),\ol{L(F)})=[L(E),\ol{L(F)}\otimes L_\infty]_{M_{\pm 2}}\\
kk^h(L(E),L(F))=[L(E),M_{\pm}L(F)]_{M_{\pm 2}}.\label{map:chris2}
\end{gather}
The identities \eqref{map:chris1} and \eqref{map:chris2} describe $kk^h$-groups in terms of homotopy classes, and thus provide an algebraic analogue, limited to Leavitt path algebras, of the description of $C^*$-$KK$-groups in terms of homotopy classes of \cite{chris}*{Theorem 4.1.1}. The latter theorem further describes the $KK$-groups in terms of approximate unitary equivalence of assymptotic $*$-homomorphisms; this is somewhat akin to the description of Remark \ref{rem:approx}. 
\end{ex}

\begin{ex}\label{ex:kkliftkk}
Let $E$ be a finite regular graph, $R$ a unital algebra and $\phi:L(E)\to R$ be an algebra homomorphism. We say that $\phi$ has property (P) if for every $e\in E^1$, $\phi(ee^*)$ is a very full idempotent of $R$. Let 
\[
[L(E),R]_{M_2}\supset [L(E),R]_{M_2}^P=\{[\phi]:\,\, \phi\,\,\text{has property (P) }\}.
\]
Next assume $R$ is $K_0$-regular and properly infinite, and that $E$ is purely infinite simple. By Remarks \ref{rem:kkvkkh} and \ref{rem:kv1}, the proof of Theorem \ref{thm:kklift} in the case of $\inv(\ell)$-algebras gives a semigroup isomorphism
\begin{equation}\label{map:kknoh}
[L(E),R]_{M_2}^P\iso kk(L(E),R).
\end{equation}
The analogue of \eqref{eq:ev-1p} holds verbatim. The analogue of the description in Remark \ref{rem:approx} holds, and $U(t)$ can be taken in $\Gl_2(R[t])$. If furthermore $\phi$ is unital and $R$ is purely infinite simple, then so is $R_\phi$, and by \cite{cm2}*{Proposition 2.8} we may take $U(t)\in \Gl_1(R_\phi[t])$. If $\ell$ is a field, then $R=L(F)$ is purely infinite simple if and only if $F$ is. If $\ell$ is not a field and $I\triqui \ell$ is a proper ideal, then $IL(F)$ is a proper ideal, so $L(F)$ is not simple.      
\end{ex}

\section{Classification theorems}\label{sec:class}

\begin{thm}\label{thm:class1} 
Let $E$ and $F$ be purely infinite finite graphs. Assume that $\ell$ satisfies the $\lambda$-assumption \ref{stan:lambda} 
and that $L(E)$ and $L(F)$ are $K^h_0$-regular. Let $\xi_0:\BF(E)\iso \BF(F)$ be an isomorphism and let $\ev$ be as in \eqref{map:elev}. Then there are $*$-homomorphisms $\phi:L(E)\to M_{\pm}L(F)$ and $\psi:L(F)\to M_{\pm}L(E)$ with pro- perty (P) such that $\ev(j^h(\phi))=K^h_0(\iota_+)\circ\can\circ\xi_0$, $\ev(j^h(\psi))=K_0^h(\iota_+)\circ\can\circ\xi_0^{-1}$, $M_{\pm}(\psi)\circ\phi\sim^s_{M_{\pm 2}}\iota_+^2:LE\to M_{\pm}M_{\pm}L(E)$ and $M_{\pm}(\phi)\circ\psi\sim^s_{M_{\pm 2}}\iota_+^2:L(F)\to M_{\pm}M_{\pm}L(F)$.
\end{thm}
\begin{proof} 
Because $E$ and $F$ are purely infinite simple, $\sink(E)=\sink(F)=\emptyset$. Hence $\xi_0$ lifts to an isomorphism $\xi:j^h(LE)\cong j^h(LF)$ 
such that $\ev(\xi)=\can\circ\xi_0$, by Theorem \ref{thm:isokk}. By Theorem \ref{thm:kklift}, there are very full $*$-homomorphisms
$\phi:L(E)\to M_{\pm}L(F)$ and $\psi:L(F)\to M_{\pm}L(E)$ such that $j^h(\phi)=j^h(\iota_+)\xi$ and $j^h(\psi)=j^h(\iota_+)\xi^{-1}$. 
Omitting $j^h$ for ease of notation, we have constructed the following commutative diagram in $kk^h$
\[
\xymatrix{
LE\ar[r]^\phi\ar[dr]_{\xi}& M_{\pm}LF\ar[r]^{M_\pm\psi}&M_\pm M_\pm LE\\
& LF\ar[dr]_{\xi^{-1}}\ar[r]^{\psi}\ar[u]^{\iota_+}&M_{\pm}LE\ar[u]^{\iota_+}\\
&& LE.\ar[u]^{\iota_+}
}
\]
Hence $j^h(\iota_+\iota_+)=j^h((M_\pm\psi)\phi)$, and therefore $\iota_+^2\sim^s_{M_{\pm 2}}(M_\pm\psi)\phi$, by Corollary \ref{coro:kklift}. (In fact, $(M_\pm\psi)\phi$ is very full, by Examples \ref{ex:vfull} and \ref{ex:pptyp}, so $[(M_\pm\psi)\phi]=[(\iota_+^2)^f]$.)
Similarly, $\iota_+^2\sim^s_{M_{\pm 2}}(M_\pm\phi)\psi$.
\end{proof}

\begin{thm}\label{thm:class2} Let $E$, $F$, $\xi_0$ and $\ell$ be as in Theorem \ref{thm:class1}. Further assume that $-1$ is positive in $\ell$. Then there exist very full $*$-homomorphisms $\phi:LE\to LF$ and $\psi:LF\to LE$ such that $\ev(j^h(\phi))=\xi_0$, $\ev(j^h(\psi))=\xi_0^{-1}$, $\psi\circ\phi\sim^*_{M_{\pm 2}}\id_{LE}$ and $\phi\circ\psi\sim^*_{M_{\pm 2}}\id_{LF}$. If furthermore $\xi_0([1]_E)=[1]_F$, then $\phi$ and $\psi$ can be chosen to be unital. 
\end{thm}
\begin{proof} The proof is essentially the same as in Theorem \ref{thm:class2}; the only difference is that because $-1$ is positive in $LE$ and $LF$, one can apply Theorem \ref{thm:isokk} directly, without going through $M_{\pm}LF$ and $M_{\pm}LE$; since the identity maps of $L(E)$ and $L(F)$ are very full, we get strict rather than stable $M_{\pm 2}$-$*$-homotopy equivalence. 
\end{proof}

\begin{ex}\label{ex:apliclass}
The hypothesis of Theorem \ref{thm:class1} are satisfied, for example, when $\ell$ is regular supercoherent and $2$ is invertible in $\ell$; if in addition $-1$ is positive in $\ell$, then also Theorem \ref{thm:class2} applies.
\end{ex}

\begin{thm}\label{thm:class3}
Let $E$ and $F$ be finite, purely infinite graphs, let $\ell$ be regular supercoherent and let $\xi_0:\BF(E)\iso \BF(F)$ be an isomorphism. Then there exist $\ell$-algebra homomorphisms $\phi:LE\to LF$ and $\psi:LF\to LE$ such that $\phi(ee^*)$ and $\psi(ee^*)$ are very full idempotents for every $e\in E^1$, $\psi\circ\phi\sim_{M_{2}}\id_{LE}$
and $\phi\circ\psi\sim_{M_2}\id_{LF}$. If furthermore $\xi_0([1]_E)=[1]_F$, then $\phi$ and $\psi$ can be chosen to be unital. 
\end{thm}
\begin{proof} The proof follows as in Theorem \ref{thm:class2}, using Remark \ref{ex:kkliftkk}. 
\end{proof}

\begin{rem}\label{rem:classy}
Theorem 6.1 of \cite{cm2} shows that if $\ell$ is a field, then in the last part of Theorem \ref{thm:class3} above, the unital homomorphisms $\phi:L(E)\leftrightarrows L(F):\psi$ can be chosen so that $\phi\circ\psi$ and $\psi\circ\phi$ are not just $M_2$-homotopic, but strictly homotopic to the identity maps. The proof uses the fact, straighforward from Example \ref{ex:kkliftkk}, that a unital endomorphism of the Leavitt path algebra of a finite purely infinite simple graph over a field goes to the corresponding identity map in $kk$ if and only if it is homotopic to an inner automorphism.  
\end{rem}

\begin{rem}\label{rem:kjh}
The proofs of Theorems \ref{thm:class1}, \ref{thm:class2} and \ref{thm:class3} use only particular cases of Theorem \ref{thm:kklift} and Example \ref{ex:kkliftkk}. For example in Theorem \ref{thm:class1} we use that for $E$ and $F$ as in the theorem, the maps $[L(E),L(F)]^f_{M_2}\to kk^h(L(E),L(F))$ and $[L(F),L(E)]^f_{M_2}\to kk^h(L(F),L(E))$ are surjective and that $[L(E),L(E)]^f_{M_2}\to kk^h(L(E),L(E))$ and $[L(F),L(F)]^f_{M_2}\to kk^h(L(F),L(F))$ are bijective. The $C^*$-analogue of the needed bijectivity result is due to Joachim Cuntz \cite{chomo}*{Proposition 3.3}; Cuntz' argument, as expanded by R\o rdam in \cite{ror}*{Theorems 3.1 and 5.2}, was adapted to the algebraic setting in \cite{cm2}. Our proof of surjectivity of \eqref{map:kklift} follows a similar argument; the proof of the injectivity of \eqref{map:kklift} given above is different, as it uses the Poincar\'e duality theorem \ref{thm:duality}
and Lemma \ref{lem:dualindex} to identify the injective map in the first exact sequence of \eqref{seq:kkler}. 
\end{rem}


\begin{bibdiv}  
\begin{biblist}
\bib{ac}{article}{
author={Abadie, Beatriz},
author={Cortiñas, Guillermo},
title={Homotopy invariance through small stabilizations},
journal={J. Homotopy Relat. Struct.},
volume={10},
number={3},
pages={459--453},
year={2015},
}

\bib{amorir}{article}{
 AUTHOR = {Abrams, G. D.},
author={\'{A}nh, P. N.},
author={M\'{a}rki, L.},
     TITLE = {A topological approach to {M}orita equivalence for rings with
              local units},
   JOURNAL = {Rocky Mountain J. Math.},
    VOLUME = {22},
      YEAR = {1992},
    NUMBER = {2},
     PAGES = {405--416},
      ISSN = {0035-7596},
       DOI = {10.1216/rmjm/1181072737},
       URL = {https://doi.org/10.1216/rmjm/1181072737},
}
\bib{aap}{article}{
AUTHOR = {Abrams, G.},
author={Aranda Pino, G.},
     TITLE = {The {L}eavitt path algebras of arbitrary graphs},
   JOURNAL = {Houston J. Math.},
    VOLUME = {34},
      YEAR = {2008},
    NUMBER = {2},
     PAGES = {423--442},
      ISSN = {0362-1588},
			}
\bib{lpabook}{book}{
author={Abrams, Gene},
author={Ara, Pere},
author={Siles Molina, Mercedes},
title={Leavitt path algebras}, 
date={2017},
series={Lecture Notes in Math.},
volume={2008},
publisher={Springer},
doi={$10.1007/978-1-4471-7344-1$},
}
\bib{aalp}{article}{
  author={Abrams, Gene},
  author={Louly, Adel},
  author={\'Ahn, Pham Ngoc},
	author={Pardo, Enrique},
  title={The classification question for Leavitt path algebras},
	journal={J. Algebra},
	volume={320},
	date={2008},
	pages={1983--2026},
}
 \bib{alps}{article}{
   author={Abrams, Gene},
   author={Louly, Adel},
   author={Pardo, Enrique},
   author={Smith, Christopher},
   title={Flow invariants in the classification of Leavitt path algebras},
   journal={J. Algebra},
   volume={333},
   date={2011},
   pages={202--231},
   issn={0021-8693},
   review={\MR{2785945}},
}
\bib{abc}{article}{
   author={Ara, Pere},
   author={Brustenga, Miquel},
   author={Corti\~nas, Guillermo},
   title={$K$-theory of Leavitt path algebras},
   journal={M\"unster J. Math.},
   volume={2},
   date={2009},
   pages={5--33},
   issn={1867-5778},
   review={\MR{2545605}},

}
\bib{gradstein}{article}{
  author={Ara, Pere},
	author={Hazrat, Roozbeh},
	author={Li, Huanhuan},
	author={Sims, Aidan},
	title={Graded Steinberg algebras and their representations}, 
	journal={Algebra Number Theory},
	volume={12},
	year={2018},
	pages={131--172},
}
\bib{k0pis}{article}{
   author={Ara, Pere},
	 author={Pardo, Enrique},
	 author={Goodearl, Kenneth R.},
	 title={$K_0$ of purely infinite simple regular rings},
   journal={$K$-theory},
	volume={26},
	 date={2002},
	pages={69--100},
}
\bib{black}{book}{
AUTHOR = {Blackadar, Bruce},
     TITLE = {{$K$}-theory for operator algebras},
    SERIES = {Mathematical Sciences Research Institute Publications},
    VOLUME = {5},
 PUBLISHER = {Springer-Verlag, New York},
      YEAR = {1986},
     PAGES = {viii+338},
      ISBN = {0-387-96391-X},
  review = {\MR{859867}},
       URL = {http://dx.doi.org/10.1007/978-1-4613-9572-0},
}

\bib{friendly}{article}{
   author={Corti\~nas, Guillermo},
   title={Algebraic v. topological $K$-theory: a friendly match},
   conference={
      title={Topics in algebraic and topological $K$-theory},
   },
   book={
      series={Lecture Notes in Math.},
      volume={2008},
      publisher={Springer, Berlin},
   },
   date={2011},
   pages={103--165},
   review={\MR{2762555}},
}

\bib{cm1}{article}{
author={Corti\~nas, Guillermo},
author={Montero, Diego},
title={Algebraic bivariant $K$-theory and Leavitt path algebras},
journal={J. Noncommut. Geom.},
status={to appear},
eprint={arXiv:1806.09204},
}
\bib{cm2}{article}{
author={Corti\~nas, Guillermo},
author={Montero, Diego},
title={Homotopy classification of Leavitt path algebras},
journal={Adv. Math.},
volume={362},
date={2020}
}

\bib{cr}{article}{
   author={Corti\~nas, Guillermo},
   author={Rodr\'\i guez, Mar\'\i a Eugenia},
   title={$L^p$-operator algebras associated with oriented graphs},
   journal={J. Operator Theory},
   volume={81},
   date={2019},
   pages={225--254},
   }
	
\bib{ct}{article}{
    AUTHOR = {Cortiñas, Guillermo},
    author={Thom, Andreas},
     TITLE = {Bivariant algebraic {$K$}-theory},
  JOURNAL = {J. Reine Angew. Math.},
    VOLUME = {610},
      YEAR = {2007},
     PAGES = {71--123},
      ISSN = {0075-4102},
       DOI = {10.1515/CRELLE.2007.068},
       URL = {https://doi.org/10.1515/CRELLE.2007.068},
}

\bib{cv}{article}{
author={Corti\~nas, Guillermo},
author={Vega, Santiago},
title={Hermitian bivariant $K$-theory and Karoubi's fundamental theorem},
eprint={arXiv:2012.09260},
}
\bib{on}{article}{
AUTHOR = {Cuntz, Joachim},
     TITLE = {{$K$}-theory for certain {$C\sp{\ast} $}-algebras},
   JOURNAL = {Ann. of Math. (2)},
%  FJOURNAL = {Annals of Mathematics. Second Series},
    VOLUME = {113},
      YEAR = {1981},
    NUMBER = {1},
     PAGES = {181--197},
      ISSN = {0003-486X},
 %  MRCLASS = {46L05 (16A54 18G99 46M20 58G12)},
  %MRNUMBER = {604046},
%MRREVIEWER = {Vern Paulsen},
       DOI = {10.2307/1971137},
       URL = {https://doi.org/10.2307/1971137},
}

\bib{chomo}{article}{
 AUTHOR = {Cuntz, J.},
     TITLE = {On the homotopy groups of the space of endomorphisms of a
              {$C\sp{\ast} $}-algebra (with applications to topological
              {M}arkov chains)},
 BOOKTITLE = {Operator algebras and group representations, {V}ol. {I}
              ({N}eptun, 1980)},
    SERIES = {Monogr. Stud. Math.},
    VOLUME = {17},
     PAGES = {124--137},
 PUBLISHER = {Pitman, Boston, MA},
      YEAR = {1984},
}
\bib{ck2}{article}{
AUTHOR = {Cuntz, J.},
     TITLE = {A class of {$C\sp{\ast} $}-algebras and topological {M}arkov
              chains. {II}. {R}educible chains and the {E}xt-functor for
              {$C\sp{\ast} $}-algebras},
   JOURNAL = {Invent. Math.},
  %FJOURNAL = {Inventiones Mathematicae},
    VOLUME = {63},
      YEAR = {1981},
    NUMBER = {1},
     PAGES = {25--40},
      ISSN = {0020-9910},
 %  MRCLASS = {46L55 (54H20 58G12)},
 % MRNUMBER = {608527},
%MRREVIEWER = {R. Nagel},
       DOI = {10.1007/BF01389192},
       URL = {https://doi.org/10.1007/BF01389192},
}
\bib{cclass}{article}{
    AUTHOR = {Cuntz, J.},
     TITLE = {The classification problem for the {$C^\ast$}-algebras {${\scr
              O}_A$}},
 BOOKTITLE = {Geometric methods in operator algebras ({K}yoto, 1983)},
    SERIES = {Pitman Res. Notes Math. Ser.},
    VOLUME = {123},
     PAGES = {145--151},
 PUBLISHER = {Longman Sci. Tech., Harlow},
      YEAR = {1986},
  % MRCLASS = {46L80 (18F25 19K14 58F25)},
  %MRNUMBER = {866492},
%MRREVIEWER = {Mi-Soo Smith},
}
\bib{dritom}{article}{
AUTHOR = {Drinen, D.},
author={ Tomforde, M.},
     TITLE = {Computing {$K$}-theory and {${\rm Ext}$} for graph
              {$C^*$}-algebras},
   JOURNAL = {Illinois J. Math.},
  %FJOURNAL = {Illinois Journal of Mathematics},
    VOLUME = {46},
      YEAR = {2002},
    NUMBER = {1},
     PAGES = {81--91},
      ISSN = {0019-2082},
       URL = {http://projecteuclid.org/euclid.ijm/1258136141},

}
\bib{fawag}{article}{
 AUTHOR = {Farrell, F. T.},
author={Wagoner, J. B.},
     TITLE = {Infinite matrices in algebraic {$K$}-theory and topology},
   JOURNAL = {Comment. Math. Helv.},
    VOLUME = {47},
      YEAR = {1972},
     PAGES = {474--501},
      ISSN = {0010-2571},
       DOI = {10.1007/BF02566819},
       URL = {https://doi.org/10.1007/BF02566819},
}

\bib{splice}{article}{
AUTHOR = {Johansen, Rune},
author={S\o rensen, Adam P. W.},
     TITLE = {The {C}untz splice does not preserve {$\ast$}-isomorphism of
              {L}eavitt path algebras over {$\mathbb{Z}$}},
   JOURNAL = {J. Pure Appl. Algebra},
    VOLUME = {220},
      YEAR = {2016},
    NUMBER = {12},
     PAGES = {3966--3983},
      ISSN = {0022-4049},
       DOI = {10.1016/j.jpaa.2016.05.023},
       URL = {https://doi.org/10.1016/j.jpaa.2016.05.023},
}
\bib{kamiput}{article}{
author={Kaminker, Jerome},
author={Putnam, Ian},
title={$K$-theoretic duality for shifts of finite type},
journal={Commun. Math. Phys.},
volume={187},
year={1997},
pages={505--541},
}

\bib{kv1}{article}{
    AUTHOR = {Karoubi, Max},
    author={Villamayor, Orlando},
		TITLE = {{$K$}-th\'{e}orie alg\'{e}brique et {$K$}-th\'{e}orie topologique. {I}},
   JOURNAL = {Math. Scand.},
    VOLUME = {28},
      YEAR = {1971},
     PAGES = {265--307 (1972)},
      ISSN = {0025-5521},
       DOI = {10.7146/math.scand.a-11024},
       URL = {https://doi.org/10.7146/math.scand.a-11024},
			}
			
\bib{kv2}{article}{
    AUTHOR = {Karoubi, Max},
    author={Villamayor, Orlando},
     TITLE = {{$K$}-th\'{e}orie alg\'{e}brique et {$K$}-th\'{e}orie topologique. {II}},
   JOURNAL = {Math. Scand.},
    VOLUME = {32},
      YEAR = {1973},
     PAGES = {57--86},
      ISSN = {0025-5521},
       DOI = {10.7146/math.scand.a-11446},
       URL = {https://doi.org/10.7146/math.scand.a-11446},
       }
\bib{chris}{article}{
AUTHOR = {Phillips, N. Christopher},
     TITLE = {A classification theorem for nuclear purely infinite simple
              {$C^*$}-algebras},
   JOURNAL = {Doc. Math.},
    VOLUME = {5},
      YEAR = {2000},
     PAGES = {49--114},
      ISSN = {1431-0635},
}

\bib{ranisemi}{article}{
author={Ranicki, Andrew},
title={On the algebraic $L$-theory of semisimple rings},
journal={J. Algebra},
volume={50},
number={1},
year={1978},
pages={242--243},
}
\bib{ror}{article}{
  title={Classification of Cuntz-Krieger algebras},
  author={R{\o}rdam, Mikael},
  journal={K-theory},
  volume={9},
  number={1},
  pages={31--58},
  year={1995},
  publisher={Springer}
}
\bib{marcofest}{article}{
author={Schilichting, Marco},
title={Hermitian $K$-theory, derived equivalences and Karoubi's fundamental theorem},
journal={J. Pure Appl. Algebra},
volume={221},
number={7},
year={2017},
pages={1729--1844},
doi={https://doi.org/10.1016/j.jpaa.2016.12.026},
}
\bib{santi}{thesis}{
author={Vega, Santiago},
title={Hermitian bivariant $K$-theory},
type={PhD thesis},
organization={Universidad de Buenos Aires},
year={2021},
}

\bib{walter}{article}{
author={Walter, Charles},
title={Grothendieck-Witt groups of triangulated categories},
status={preprint},
eprint={https://faculty.math.illinois.edu/K-theory/0643/},
}

\bib{kh}{article}{
   author={Weibel, Charles A.},
   title={Homotopy algebraic $K$-theory},
   conference={
      title={Algebraic $K$-theory and algebraic number theory (Honolulu, HI,
      1987)},
   },
   book={
      series={Contemp. Math.},
      volume={83},
      publisher={Amer. Math. Soc.},
      place={Providence, RI},
   },
   date={1989},
   pages={461--488},
   review={\MR{991991 (90d:18006)}},
}

\end{biblist}
\end{bibdiv}
\end{document}